\documentclass[12pt,a4paper]{amsart}
\pdfoutput=1
\usepackage[centering, margin=1in]{geometry}
\usepackage{amsmath,amsfonts,amsthm,amssymb,enumitem,graphicx,mathrsfs,mathtools}
\usepackage[stretch=10]{microtype}
\usepackage[unicode]{hyperref}
\usepackage{xcolor}
\definecolor{dark-red}{rgb}{0.4,0.15,0.15}
\definecolor{dark-blue}{rgb}{0.15,0.15,0.4}
\definecolor{medium-blue}{rgb}{0,0,0.5}
\hypersetup{
	pdftitle={Optimal Small Scale Equidistribution of Lattice Points on the Sphere, Heegner Points, and Closed Geodesics},
	pdfauthor={Peter Humphries and Maksym Radziwi\l{}\l{}},
	colorlinks,
	linkcolor={dark-red},
	citecolor={dark-blue},
	urlcolor={medium-blue}
}
\makeatletter
\DeclareRobustCommand{\widecheck}[1]{{\mathpalette\@widecheck{#1}}}
\def\@widecheck#1#2{%
	\setbox\z@\hbox{\m@th$#1#2$}%
	\setbox\tw@\hbox{\m@th$#1%
		\widehat{%
			\vrule\@width\z@\@height\ht\z@
			\vrule\@height\z@\@width\wd\z@}$}%
	\dp\tw@-\ht\z@
	\@tempdima\ht\z@ \advance\@tempdima2\ht\tw@ \divide\@tempdima\thr@@
	\setbox\tw@\hbox{%
		\raise\@tempdima\hbox{\scalebox{1}[-1]{\lower\@tempdima\box
				\tw@}}}%
	{\ooalign{\box\tw@ \cr \box\z@}}}
\makeatother
\allowdisplaybreaks
\newcommand{\A}{\mathbb{A}}
\newcommand{\BB}{\mathcal{B}}
\renewcommand{\C}{\mathbb{C}}
\newcommand{\CC}{\mathcal{C}}
\newcommand{\DD}{\mathcal{D}}
\newcommand{\Dgp}{\mathrm{D}}
\newcommand{\EE}{\mathcal{E}}
\newcommand{\GL}{\mathrm{GL}}
\newcommand{\Hb}{\mathbb{H}}
\newcommand{\hol}{\textnormal{hol}}
\newcommand{\JJ}{\mathcal{J}}
\newcommand{\Ks}{\mathscr{K}}
\newcommand{\Ls}{\mathscr{L}}
\newcommand{\MM}{\mathcal{M}}
\newcommand{\N}{\mathbb{N}}
\newcommand{\OO}{\mathcal{O}}
\newcommand{\Q}{\mathbb{Q}}
\newcommand{\R}{\mathbb{R}}
\newcommand{\RR}{\mathcal{R}}
\newcommand{\SL}{\mathrm{SL}}
\newcommand{\SO}{\mathrm{SO}}
\newcommand{\spec}{\textnormal{spec}}
\renewcommand{\SS}{\mathcal{S}}
\newcommand{\WW}{\mathcal{W}}
\newcommand{\Z}{\mathbb{Z}}
\newcommand{\Zgp}{\mathrm{Z}}
\newcommand{\e}{\varepsilon}
\DeclareMathOperator{\Ram}{Ram}
\DeclareMathOperator{\sech}{sech}
\DeclareMathOperator{\sgn}{sgn}
\DeclareMathOperator{\Si}{Si}
\DeclareMathOperator{\sym}{sym}
\DeclareMathOperator{\Var}{Var}

\numberwithin{equation}{section}
\newtheorem{theorem}[equation]{Theorem}

\newtheorem{conjecture}[equation]{Conjecture}
\newtheorem{corollary}[equation]{Corollary}
\newtheorem{lemma}[equation]{Lemma}
\newtheorem{proposition}[equation]{Proposition}
\theoremstyle{remark}
\newtheorem{remark}[equation]{Remark}
\begin{document}

\title[Optimal Small Scale Equidistribution of Lattice Points on the Sphere]{Optimal Small Scale Equidistribution of Lattice Points on the Sphere, Heegner Points, and Closed Geodesics}

\author{Peter Humphries}

\address{Department of Mathematics, University of Virginia, Charlottesville, VA 22904, USA}

\email{\href{mailto:pclhumphries@gmail.com}{pclhumphries@gmail.com}}

\author{Maksym Radziwi\l{}\l{}}

\address{Department of Mathematics Caltech, 1200 E California Blvd Pasadena, CA 91125, USA}

\email{\href{mailto:maksym.radziwill@gmail.com}{maksym.radziwill@gmail.com}}

\subjclass[2010]{11E16 (primary); 11F67 (secondary)}

\thanks{The first author is supported by the European Research Council grant agreement 670239.
The second author acknowledges support of NSF grant DMS-1902063 and of a Sloan fellowship.}

\begin{abstract}
We asymptotically estimate the variance of the number of lattice points in a thin, randomly rotated annulus lying on the surface of the sphere. This partially resolves a conjecture of Bourgain, Rudnick, and Sarnak. We also obtain estimates that are valid for all balls and annuli that are not too small. Our results have several consequences: for a conjecture of Linnik on sums of two squares and a ``microsquare'', a conjecture of Bourgain and Rudnick on the number of lattice points lying in small balls on the surface of the sphere, the covering radius of the sphere, and the distribution of lattice points in almost all thin regions lying on the surface of the sphere. Finally, we show that for a density $1$ subsequence of squarefree integers, the variance exhibits a different asymptotic behaviour for balls of volume $(\log n)^{-\delta}$ with $0 < \delta < \tfrac {1}{16}$.

We also obtain analogous results for Heegner points and closed geodesics. Interestingly, we are able to prove some slightly stronger results for closed geodesics than for Heegner points or lattice points on the surface of the sphere. A crucial observation that underpins our proof is the different behaviour of weighting functions for annuli and for balls. 
\end{abstract}

\maketitle

\section{Introduction}

\label{introsect}

\subsection*{I. Lattice points on the surface of the sphere}
\subsection{Variance and equidistribution results}

For a positive odd squarefree integer $n$, let 
\[
\mathcal{E}(n) \coloneqq \{(x_1, x_2, x_3) \in \mathbb{Z}^3 : x_1^2 + x_2^2 + x_3^2 = n\}
\]
denote the set of lattice points lying on the surface of a sphere of radius $\sqrt{n}$ and centred at the origin.
This set is nonempty whenever $n \not \equiv 7 \pmod{8}$. 
For convenience and ease of exposition, we impose throughout the additional condition that $n \equiv 3 \pmod{8}$, so that
$-n$ is a fundamental discriminant\footnote{With additional effort, all of our results
can be extended to handle the general case of $n \not \equiv 7 \pmod{8}$ squarefree.}. 
By classical work of Gauss together with Dirichlet's class number formula, 
\[
\# \mathcal{E}(n) = \frac{24 \sqrt{n} L(1, \chi_{-n})}{\pi}, 
\]
where $\chi_{-n}$ is the primitive quadratic Dirichlet character associated to the fundamental discriminant $-n$. In particular, Siegel's theorem implies that $\# \mathcal{E}(n) = n^{1/2 - o(1)}$ as $n$ goes to infinity.

Linnik \cite{Lin68} used a novel ``ergodic method'' to show that if in addition $n \equiv \pm 1 \pmod{5}$, then the set
\[
\widehat{\mathcal{E}}(n) \coloneqq \left\{ \left( \frac{x_1}{\sqrt{n}}, \frac{x_2}{\sqrt{n}}, \frac{x_3}{\sqrt{n}}\right) \in S^2 : (x_1, x_2, x_3) \in \mathcal{E}(n) \right\} 
\]
is equidistributed on $S^2$ as $n \to \infty$. Removing this additional 
congruence condition proved quite challenging and was accomplished only twenty years later
by Duke \cite{Duk88, DS-P90} and Golubeva--Fomenko \cite{GF90} following
a breakthrough of Iwaniec \cite{Iwa87}. 

It is desirable, both from a theoretical and applied point of view, to understand the finer distribution of the normalised
lattice points $\widehat{\mathcal{E}}(n)$ on $S^2$. Bourgain, Rudnick, and Sarnak \cite{BRS17} proposed that the distribution of the points
$\widehat{\mathcal{E}}(n)$ should be essentially similar to that of ``random points'', that is, points thrown uniformly at random
on the surface of the unit sphere. In order to make this precise, they stated the following conjecture (among others).

\begin{conjecture}[Bourgain--Rudnick--Sarnak] \label{conj:brs}
Let $\varepsilon > 0$ be given. Let $\Omega_n \subset S^2$ be a sequence of balls or annuli. Let $\sigma(\Omega_n)$ denote the surface measure of $\Omega_n$ on $S^2$ normalised such that $\sigma(S^2) = 4\pi$. If $\# \widehat{\mathcal{E}}(n)^{-1 + \varepsilon} \leq \sigma(\Omega_n) \leq \# \widehat{\mathcal{E}}(n)^{-\varepsilon}$, then
\begin{equation} \label{eq:conjbrs}
\int_{\SO(3)} \left( \# (\widehat{\mathcal{E}}(n) \cap g \Omega_n ) - \# \widehat{\mathcal{E}}(n) \frac{\sigma(\Omega_n)}{\sigma(S^2)} \right)^2 \, d g \sim \# \widehat{\mathcal{E}}(n) \frac{\sigma(\Omega_n)}{\sigma(S^2)}
\end{equation}
as $n \to \infty$ along integers for which $n \not\equiv 0,4,7 \pmod{8}$. 
\end{conjecture}

The left-hand side of \eqref{eq:conjbrs} corresponds to the variance of the number of points in $\widehat{\mathcal{E}}(n)$ lying in a randomly rotated set $\Omega_n$. If the points $\widehat{\mathcal{E}}(n)$ are distributed as ``random points'', then we expect this variance to coincide asymptotically with $\# \widehat{\mathcal{E}}(n) \sigma(\Omega_n) / \sigma(S^2)$. This motivates \hyperref[conj:brs]{Conjecture \ref*{conj:brs}}. The restriction $\# \widehat{\mathcal{E}}(n)^{-1 + \varepsilon} \leq \sigma (\Omega_n)$ ensures that on average over all rotations $g \in \SO(3)$, $\# (\widehat{\mathcal{E}}(n) \cap g \Omega_n)$ tends to infinity. However, this restriction appears to be unnecessary for the validity of \eqref{eq:conjbrs} as $n \to \infty$. 

One can draw a parallel between \hyperref[conj:brs]{Conjecture \ref*{conj:brs}} and certain classical results of analytic number theory such as the Barban--Davenport--Halberstam theorem \cite[Chapter 29]{Dav80}.
As in the case of Barban--Davenport--Halberstam theorem, the most interesting range is that in which $\sigma(\Omega_n)$ is as small as possible, close to $\#\widehat{\mathcal{E}}(n)^{-1 + \varepsilon}$ (respectively the arithmetic progression is as short as possible). However, the most difficult range, associated to challenging problems about $L$-functions (see \cite{GV97}), is the range in which $\sigma(\Omega_n)$ is large of size $n^{-\varepsilon}$ (respectively the arithmetic progression is long). Indeed, establishing \hyperref[conj:brs]{Conjecture \ref*{conj:brs}} in the full range implies the Lindel\"of hypothesis at the central point for a certain family of $L$-functions. We will therefore focus on microscopic $\Omega_n$ for which $\sigma(\Omega_n)$ is close to $\# \widehat{\mathcal{E}}(n)^{-1 + \varepsilon}$. Our first result is the following.

\begin{theorem} \label{variancethm1}
Let $\delta,\e > 0$ be given. Let $A_{r, R}(w)$ denote the annulus on $S^2$ centred at a fixed point $w \in S^2$ with inner radius $r$ and outer radius $R$. Suppose that $n^{-1/12 + \delta} \leq r \leq \pi - \e$ and $\sigma(A_{r,R}) \leq r n^{-5/12 - \delta}$. Then 
\[
\int_{\SO(3)} \left( \# ( \widehat{\mathcal{E}}(n) \cap g A_{r,R}(w) ) - \# \widehat{\mathcal{E}}(n) \frac{\sigma(A_{r,R})}{\sigma(S^2)} \right)^2 \, dg \sim \# \widehat{\mathcal{E}}(n) \frac{\sigma(A_{r, R})}{\sigma(S^2)}
\]
as $n \to \infty$ along squarefree $n \equiv 3 \pmod{8}$.
\end{theorem}

\hyperref[variancethm1]{Theorem \ref*{variancethm1}} verifies \hyperref[conj:brs]{Conjecture \ref*{conj:brs}} for annuli with large inner radius in the (nontrivial) regime of $\sigma(A_{r,R})$ slightly larger than $\# \widehat{\mathcal{E}}(n)^{-1}$. Establishing \hyperref[variancethm1]{Theorem \ref*{variancethm1}} for balls with volume slightly larger than $\# \widehat{\mathcal{E}}(n)^{-1}$ appears to be currently out of reach as it is equivalent with estimating asymptotically a first moment of $L$-functions that implies \textit{sub-Weyl subconvexity} (see \hyperref[connectionssect]{Section \ref*{connectionssect}} for details). This gives a natural geometric interpretation of the meaning of sub-Weyl subconvexity. We believe that this is a point that deserves further exploration. We refer the reader to the forthcoming work of Shubin \cite{Shu21} for conditional results for balls; specifically, he establishes upper bounds of the correct order of magnitude for balls in \eqref{eq:conjbrs} conditionally on the generalised Riemann hypothesis. 

Minor modifications of the proof of \hyperref[variancethm1]{Theorem \ref*{variancethm1}} allow us to show that almost all annuli $A_{r, R}$ with large inner radius $r$ contain the expected number of lattice points. Note that we do not impose any significant constraints on the volume of $A_{r, R}$.

\begin{theorem} \label{aethm1}
Let $0 < \delta < \tfrac {1}{12}$ and $c,\e > 0$ be given. Suppose that $n^{-1/12 + \delta} \ll r < R \leq \pi - \e$ and $n^{-1/2 + \delta} \ll \sigma(A_{r,R}) \ll 1$. Then as $n \to \infty$ along squarefree $n \equiv 3 \pmod{8}$, 
\[\sigma\left(\left\{w \in S^2 : \left|\frac{\sigma(S^2)}{\sigma(A_{r,R})} \frac{\# (\widehat{\EE}(n) \cap A_{r,R}(w))}{\# \widehat{\EE}(n)} - 1\right| > c\right\}\right) = o_{\delta} (1).\]
\end{theorem}

\begin{remark}
Ellenberg, Michel, and Venkatesh \cite[Theorem 1.8]{EMV13} have proven an analogue of \hyperref[aethm1]{Theorem \ref*{aethm1}} concerning the equidistribution of lattice points on the discrete sphere modulo a fixed integer $q$. Their result uses an extension of Linnik's ergodic method; in particular, it requires the additional hypothesis $n \equiv \pm 1 \pmod{5}$.
\end{remark}

\hyperref[aethm1]{Theorem \ref*{aethm1}} can be interpreted as a result for the \textit{average covering exponent} of lattice points. The \textit{average covering exponent} of lattice points on the $2$-sphere is
\[\bar{K}_3 \coloneqq \lim_{\delta \to 0} \limsup_{R \to 0} \frac{\log \# \widehat{\EE}(n(R^{-\delta}))}{\log \frac{1}{\sigma(B_R)}}\]
where $n(R^{-\delta})$ denotes the smallest integer for which the $w$-volume of the exceptional set of balls $B_R(w)$ of radius $R$ on $S^{2}$ that do not contain a point in $\widehat{\EE}(n(R^{-\delta}))$ is at most $R^{-\delta}$. 
Bourgain, Rudnick, and Sarnak have shown that $\bar{K}_3 = 1$ assuming the generalised Lindel\"{o}f hypothesis \cite[Theorem 1.8]{BRS17}. With this in mind, we may interpret \hyperref[aethm1]{Theorem \ref*{aethm1}} as an unconditional (and optimal) version of the result $\bar{K}_3 = 1$ for annuli instead of balls. For work on the case of higher dimensional spheres, see \cite[Corollary 1.6]{Sar19}.

Using estimates for the third moment of $L(\tfrac 12, f \otimes \chi_{-n})$ and the same set-up as in the proof of \hyperref[aethm1]{Theorem \ref*{aethm1}} allows us to obtain estimates for the number of lattice points lying in a given shrinking ball.

\begin{theorem} \label{Youngthm} 
Fix $w \in S^2$. Let $0 \leq \delta < \tfrac{1}{12}$. Then as $n \to \infty$ along squarefree $n \equiv 3 \pmod{8}$, and for $n^{-\delta} \ll \sigma(B_R) \ll 1$, 
\[\frac{\sigma(S^2)}{\sigma(B_R)} \frac{\# (\widehat{\EE}(n) \cap B_R(w))}{\# \widehat{\EE}(n)} = 1 + o_{\delta}(1).\]
Assuming the generalised Lindel\"of hypothesis, we obtain the same result on the weaker assumption that $0 \leq \delta < \tfrac 14$. 
\end{theorem}

\hyperref[Youngthm]{Theorem \ref*{Youngthm}} can be seen as the exact quaternion algebra analogue of the results of Young \cite[Theorem 2.1]{You17}, \cite[Theorem 1.24]{Hum18}, which concern the distribution of Heegner points and closed geodesics on the modular surface. We will return to this analogy in the second part of this introduction.

\hyperref[Youngthm]{Theorem \ref*{Youngthm}} has consequences for the so-called \textit{covering exponent} of lattice points on the $2$-sphere, defined as
\[K_3 \coloneqq \limsup_{n \to \infty} \frac{\log \# \widehat{\EE}(n)}{\log \frac{1}{\sigma(B_{\RR(n)})}},\]
where the covering radius $\RR(n)$ of $\widehat{\EE}(n)$ is the least $R > 0$ for which every point $w \in S^2$ is within distance $R$ of some point in $\widehat{\EE}(n)$. One expects that $K_3 = 1$ \cite[Conjecture 1.9]{BRS17}; towards this, Bourgain, Rudnick, and Sarnak have shown that \hyperref[Youngthm]{Theorem \ref*{Youngthm}} holds under the assumption of the generalised Lindel\"of hypothesis for $\delta < \tfrac{1}{4}$, so that $K_3 \leq 2$. \hyperref[Youngthm]{Theorem \ref*{Youngthm}} reproves this conditional result via slightly different means and unconditionally shows that $K_3 \leq 6$.

\hyperref[Youngthm]{Theorem \ref*{Youngthm}} also establishes for $R \gg n^{-1/24 + \varepsilon}$ a case of a conjecture of Bourgain--Rudnick \cite{BR12}, according to which 
\[
\# (\widehat{\EE}(n) \cap B_R(w)) \ll_{\e} n^{\e} \left(1 + \sqrt{n} R^2\right)
\]
for $R < n^{-\delta}$ for some fixed $\delta > 0$. The method of proof shows more generally that
\[
\# (\widehat{\EE}(n) \cap B_R(w)) \ll_{\e} n^{\frac{1}{2} + \e} R^2 + n^{\frac{5}{12} + \e} R^{1 - \e}
\]
for all $R \ll 1$. Assuming the generalised Lindel\"{o}f hypothesis, $5/12$ may be replaced by $1/4$.

\subsection{A diophantine conjecture of Linnik}

Using the results of the previous section, we are able to make progress on a conjecture of Linnik on the representation of any positive odd squarefree integer $n \not \equiv 7 \pmod{8}$ as a sum of two squares and a ``microsquare''. 

\begin{conjecture}[{Linnik \cite[Chapter XI]{Lin68}}] \label{conj:linnik} 
Fix $\varepsilon > 0$. For each sufficiently large odd squarefree integer $n \not \equiv 7 \pmod{8}$, there exists a solution $(x_1, x_2, x_3) \in \mathbb{Z}^3$ to the equation $x_1^2 + x_2^2 + x_3^2 = n$ with $|x_3| \leq n^{\varepsilon}$. 
\end{conjecture}

Wooley \cite[Corollary 1.3]{Woo14} has shown that such a solution exists (and in a stronger form) for \emph{almost every} positive odd squarefree integer $n \not\equiv 7 \pmod{8}$. Using \hyperref[variancethm1]{Theorem \ref*{variancethm1}}, we are able to establish almost all ``rotated'' versions of this conjecture. 

\begin{theorem}
\label{Linnikthm1}Let $\delta > 0$ and $0 < \psi(n) < n^{1/12 - \delta}$ be a nondecreasing function. Let $n \equiv 3 \pmod{8}$ be squarefree, sufficiently large with respect to $1 / \delta$. Then the volume of the set of $w \in S^2$ for which there exists $x = (x_1, x_2, x_3) \in \mathbb{Z}^3$ with $x_1^2 + x_2^2 + x_3^2 = n$ and $|x \cdot w| \leq \psi(n)$ is $\sigma(S^2) + O(\psi(n)^{-1} L(1, \chi_{-n})^{-1})$. 
\end{theorem}

Siegel's theorem implies that $L(1, \chi_{-n}) \gg_{\varepsilon} n^{-\varepsilon}$ for every $\varepsilon > 0$. 
Thus for every $\delta > 0$, the measure of the set of $w \in S^2$ for which there exists $x = (x_1, x_2, x_3) \in \mathbb{Z}^3$ with $x_1^2 + x_2^2 + x_3^2 = n$ and $|x \cdot w| \leq n^{\delta}$ is $\sigma(S^2) + O_{\varepsilon}(n^{-\delta + \varepsilon})$. This establishes that almost all ``rotated Linnik conjectures'' hold (\hyperref[conj:linnik]{Conjecture \ref*{conj:linnik}} corresponds to $w = (0,0,1)$).
On the assumption of the generalised Riemann hypothesis, the requirement $|x \cdot w | \leq n^{\delta}$ could be weakened to $| x \cdot w | \leq (\log \log n)^{1 + \varepsilon}$ for any $\varepsilon > 0$, since the generalised Riemann hypothesis implies that $L(1, \chi_{-n}) \gg (\log\log n)^{-1}$ for all squarefree integers $n \equiv 3 \pmod{8}$. 

Using \hyperref[Youngthm]{Theorem \ref*{Youngthm}}, we can also make the following partial progress on Linnik's conjecture proper.

\begin{theorem} \label{Linnikthm2}
Let $0 \leq \delta < \tfrac{1}{18}$. Let $n \equiv 3 \pmod{8}$ be squarefree and sufficiently large with respect to $1 / \delta$. Then there exists $x = (x_1, x_2, x_3) \in \mathbb{Z}^3$ with $|x_3| \leq n^{1/2 - \delta}$ and such that $x_1^2 + x_2^2 + x_3^2 = n$. Under the assumption of the generalised Lindel\"of hypothesis, we can assume that $0 \leq \delta < \tfrac {1}{6}$.
\end{theorem}

This improves upon the same result with the weaker condition $|x_3| \leq \delta \sqrt{n}$ for any fixed $\delta > 0$, which follows from the equidistribution result of Duke and Golubeva--Fomenko, as well as the subsequent refinement $|x_3| \leq n^{1/2 - \delta}$ for $0 \leq \delta < \tfrac{7}{705}$ due to Golubeva--Fomenko \cite[Corollary]{GF94}.

\subsection{Failure of equidistribution and variance asymptotics}

We end this part of the introduction by discussing negative results on the regimes in which equidistribution and variance estimates fail to behave as expected. 

First of all, we cannot expect equidistribution for regions $\Omega = A_{r, R}$ or $\Omega = B_{R}$ with volume smaller than
$\# \widehat{\mathcal{E}}(n)^{-1}$, since then on average $\Omega$ contains only a bounded number of rational points. Nonetheless, the variance of the number of points in $\Omega$ still behaves like the variance of points thrown randomly on the surface of the sphere; this is implicit in our proof of \hyperref[variancethm1]{Theorem \ref*{variancethm1}} and consistent with the probabilistic predictions. 

Secondly, when the region $\Omega = A_{r, R}$ or $\Omega = B_{R}$ has large volume, say $\sigma (\Omega) \gg (\log n)^{-\delta}$ with $0 < \delta < \tfrac {1}{16}$, we naturally have equidistribution of the number of rational points within $\Omega$ by the results of Duke and Golubeva--Fomenko. However, for a density $1$ sequence of squarefree integers $n \equiv 3 \pmod{8}$, the variance of the number of rational points in such randomly rotated regions is asymptotically arbitrarily small compared to the variance of the number of random points in such a randomly rotated region. This uses crucially the fact that for a Hecke modular form $f$ and almost all squarefree integers $n \equiv 3 \pmod{8}$, the central values $L(\tfrac 12, f \otimes \chi_{-n})$ are bounded by $\varepsilon$ for any given $\varepsilon > 0$ (see \cite{RS15}). 

We state these negative results in the theorem below. 

\begin{theorem} \label{pigeonholethm}
\hspace{1em}
\begin{enumerate}[leftmargin=*]
\item[\textnormal{(1)}] Let $\delta > 0$. Let $B_{R} \subset S^2$ be a ball of volume $n^{-1/2 - \delta}$. Let $0 < \varepsilon < 1$. Then as $n \to \infty$ along squarefree integers $n \equiv 3 \pmod{8}$,
\[
\sigma \left(\left\{w \in S^2 :
\left| \# (\widehat{\mathcal{E}}(n) \cap B_R(w)) - \# \widehat{\mathcal{E}}(n) \frac{\sigma(B_R)}{\sigma(S^2)} \right| > \varepsilon \# \widehat{\mathcal{E}}(n) \frac{\sigma(B_R)}{\sigma(S^2)} \right\}\right) = \sigma(S^2) + o(1).
\]
\item[\textnormal{(2)}] \label{pigeonsecond} Let $0 < \delta < \tfrac{1}{16}$ be given. Let $B_{R} \subset S^2$ be a ball of volume $(\log n)^{-\delta}$. Then there exists a density $1$ subset $S$ of squarefree integers $n \equiv 3 \pmod{8}$ such that as $n \to \infty$ along $n \in S$, 
\[
\int_{\SO(3)}
\left( \# ( \widehat{\mathcal{E}}(n) \cap g B_{R} ) - \# \widehat{\mathcal{E}}(n) \frac{\sigma(B_R)}{\sigma(S^2)} \right)^2 \, dg = o \left( \# \widehat{\mathcal{E}}(n) \frac{\sigma(B_R)}{\sigma(S^2)} \right).
\]
\end{enumerate}
\end{theorem}

Moreover, it should be possible to show using the methods in \cite{Sou08} and \cite{GS03} that for any fixed $\Delta > 0$ and for a ball $B_{R}$ of volume greater than $\exp(- (\log n)^{1/2 - \delta})$ with $\delta \in (0, 1/2)$, there exists a subsequence of the squarefree integers $n \equiv 3 \pmod{8}$ such that as $n \to \infty$ along this subsequence, 
\[
\int_{\SO(3)} \left( \# (\widehat{\mathcal{E}}(n) \cap g B_{R}) - \# \widehat{\mathcal{E}}(n) \frac{\sigma(B_{R})}{\sigma(S^2)} \right)^2 \, d g \geq \Delta \# \widehat{\mathcal{E}}(n) \frac{\sigma(B_R)}{\sigma(S^2)}.
\]
Therefore the case of balls of large volume exhibits truly chaotic behaviour. One can still draw an analogy here with the case of the Barban--Davenport--Halberstam theorem, which is also expected to fail for extremely long arithmetic progressions \cite{Fio15}.

It is a fairly delicate question to determine even conjecturally the optimal threshold at which we expect \hyperref[conj:brs]{Conjecture \ref*{conj:brs}} to hold. We believe, based on conjectures about the maximal size of an $L$-function, that this threshold is around $\exp(-(\log n)^{1/2 + o(1)})$; that is, \hyperref[conj:brs]{Conjecture \ref*{conj:brs}} holds for all balls $B_{R}$ with $\sigma(B_R) \leq \exp( - (\log n)^{1/2 + \varepsilon})$ for any given $\varepsilon > 0$ and $n \to \infty$ along squarefree integers $n \equiv 3 \pmod{8}$.

\subsection*{II. Heegner points and closed geodesics}
\subsection{Variance and equidistribution results}

The results of the first part have analogues for Heegner points and closed geodesics. An interesting feature that we highlight is that we are able to obtain equidistribution of closed geodesics in almost every ball in all regimes, whereas we are unable to obtain such a result for lattice points on the sphere or for Heegner points. We start by recalling some standard results about Heegner points and closed geodesics. 

Let $D < 0$ be a fundamental discriminant. Each ideal class in the class group of the imaginary quadratic field $\Q(\sqrt{D})$ is associated to an orbit of primitive irreducible integral binary quadratic forms $Q(x,y) = ax^2 + bxy + cy^2$ of discriminant $b^2 - 4ac = D$ under the action of the modular group $\Gamma = \SL_2(\Z)$. In turn, such an orbit is associated to a $\Gamma$-orbit of points $(-b + \sqrt{D})/2a$ in the upper half-plane $\Hb$, or equivalently a single Heegner point in the modular surface $\Gamma \backslash \Hb$. We denote by $\Lambda_D$ the set of Heegner points of discriminant $D$ in $\Gamma \backslash \Hb$.

Similarly, let $\Q(\sqrt{D})$ be a real quadratic field of discriminant $D > 0$. Each narrow ideal class in the narrow class group of $\Q(\sqrt{D})$ is associated to a $\Gamma$-orbit of primitive irreducible integral binary quadratic forms of discriminant $b^2 - 4ac = D$; in turn, such an orbit is associated to a $\Gamma$-orbit of closed geodesics in the upper half-plane that intersect the boundary at $(-b \pm \sqrt{D})/2a$, or equivalently a single closed geodesic $\CC \subset \Gamma \backslash \Hb$. We again let $\Lambda_D$ denote the set of closed geodesics of discriminant $D$ in $\Gamma \backslash \Hb$.

We can count the number of Heegner points via the class number formula:
\[\# \Lambda_D = h(D) = \frac{w_D \sqrt{|D|} L(1, \chi_D)}{2\pi}, \qquad w_D = \begin{dcases*}
4 & for $D = -4$,	\\
6 & for $D = -3$,	\\
2 & otherwise,
\end{dcases*}\]
where $\chi_D$ is the primitive quadratic character modulo $|D|$. We can also measure the total length of the closed geodesics in $\Lambda_D$:
\[\sum_{\CC \in \Lambda_D} \ell(\CC) = 2 \sqrt{D} L(1, \chi_D),\]
where $\ell(\CC) := \int_{\CC} \, ds$ denotes the length of a curve $\CC$ with respect to the measure $ds^2 := y^{-2} \, dx^2 + y^{-2} \, dy^2$. 
These quantities can then be tightly bounded via the bounds $|D|^{-\e} \ll_{\e} L(1,\chi_D) \ll \log |D|$. We will also denote by $d\mu(w)$ the hyperbolic volume, so that for $w = x + i y$,
\[
d \mu(w) = \frac{dx \, dy}{y^2}. 
\]

Similarly to \hyperref[variancethm1]{Theorem \ref*{variancethm1}}, we are able to obtain an asymptotic estimate for the variance of Heegner points intersecting shrinking annuli. We let $A_{r,R}(w)$ denote the annulus centred at $w \in \Gamma \backslash \Hb$ with inner radius $r$ and outer radius $R$.

\begin{theorem}
\label{variancethm2}
If $|D|^{-1/12 + \delta} \ll r \ll 1$ and $\mu(A_{r, R}) \ll r |D|^{-5 /12 - \delta}$ for some $\delta > 0$, then as $D \to - \infty$ through negative squarefree fundamental discriminants, 
\[
\frac{1}{\mu(\Gamma \backslash \Hb)} \int_{\Gamma \backslash \Hb} \left( \# (\Lambda_D \cap A_{r,R}(w)) - \# \Lambda_{D} \frac{\mu(A_{r,R})}{\mu(\Gamma \backslash \Hb)} \right)^2 \, d\mu(w) \sim \# \Lambda_{D} \frac{\mu(A_{r,R})}{\mu(\Gamma \backslash \Hb)}.
\]
\end{theorem}

It should be possible to obtain analogous results for closed geodesics but we have not investigated this in any detail. 

A consequence of this variance estimate is the following equidistribution results for almost all balls and annuli.
An interesting feature is that in the case of closed geodesics, we are able to obtain such a result for almost all balls. 

\begin{theorem}
\label{aethm2}
Fix $c > 0$.
\begin{enumerate}[leftmargin=*]
\item[\textnormal{(1)}] Suppose that $|D|^{-1/12 + \delta} \ll r < R \ll 1$ and $|D|^{-1/2 + \delta} \ll \mu(A_{r,R}) \ll 1$ for some $\delta > 0$. Then as $D \to -\infty$ through squarefree fundamental discriminants,
\[\mu\left(\left\{w \in \Gamma \backslash \Hb : \left|\frac{\mu(\Gamma \backslash \Hb)}{\mu(A_{r,R})} \frac{\# (\Lambda_D \cap A_{r,R}(w))}{\# \Lambda_D} - 1\right| > c\right\}\right) = o_{\delta}(1).\]
\item[\textnormal{(2)}] Suppose that $0 \leq r < R \ll 1$ and $D^{-1 + \delta} \ll \mu(A_{r,R}) \ll 1$ for some $\delta > 0$. Then as $D \to \infty$ through squarefree fundamental discriminants,
\[\mu\left(\left\{w \in \Gamma \backslash \Hb : \left|\frac{\mu(\Gamma \backslash \Hb)}{\mu(A_{r,R})} \frac{\sum_{\CC \in \Lambda_D} \ell(\CC \cap A_{r,R}(w))}{\sum_{\CC \in \Lambda_D} \ell(\CC)} - 1\right| > c\right\}\right) = o_{\delta}(1).\]
\end{enumerate}
\end{theorem}

We highlight the fact that the condition $D^{-1 + \delta} \ll \mu(A_{r,R})$ in \hyperref[aethm2]{Theorem \ref*{aethm2} (2)} ensures that when $r = 0$, so that $A_{r,R} = B_R$ is a ball, the radius $R$ is at least of size $D^{-1/2 + \delta/2}$, and hence that $1/R$ is smaller than $\sum_{\CC \in \Lambda_D} \ell(\CC)$.

\section{Reduction to bounds for moments of \texorpdfstring{$L$}{L}-functions}

Throughout we will normalise our variances slightly differently and consider
\begin{align}
\label{VarD<0defeq}
\Var(\Lambda_D; A_{r,R}) & \coloneqq \frac{1}{\mu(\Gamma \backslash \Hb)} \int_{\Gamma \backslash \Hb} \left(\frac{\mu(\Gamma \backslash \Hb)}{\mu(A_{r,R})} \frac{\# (\Lambda_D \cap A_{r,R}(w))}{\# \Lambda_D} - 1\right)^2 \, d\mu(w)	\\
\intertext{for $D < 0$,}
\label{VarD>0defeq}
\Var(\Lambda_D; A_{r,R}) & \coloneqq \frac{1}{\mu(\Gamma \backslash \Hb)} \int_{\Gamma \backslash \Hb} \left(\frac{\mu(\Gamma \backslash \Hb)}{\mu(A_{r,R})} \frac{\sum_{\CC \in \Lambda_D} \ell(\CC \cap A_{r,R}(w))}{\sum_{\CC \in \Lambda_D} \ell(\CC)} - 1\right)^2 \, d\mu(w)	\\
\intertext{for $D > 0$,}
\label{Varndefeq}
\Var(\widehat{\EE}(n); A_{r,R}) & \coloneqq \frac{1}{\sigma(S^2)} \int_{S^2} \left(\frac{\sigma(S^2)}{\sigma(A_{r,R})} \frac{\# (\widehat{\EE}(n) \cap A_{r,R}(w))}{\# \widehat{\EE}(n)} - 1\right)^2 \, d\sigma(w).
\end{align}
With this normalisation, the expected results are respectively
\begin{gather*}
\Var(\Lambda_D; A_{r,R}) \sim \frac{\mu(\Gamma \backslash \Hb)}{\mu(A_{r,R}) \# \Lambda_D}, \qquad \Var(\Lambda_D; A_{r,R}) \sim \frac{\mu(\Gamma \backslash \Hb)}{\mu(A_{r,R}) \sum_{\CC \in \Lambda_D} \ell(\CC)},	\\
\Var(\widehat{\EE}(n); A_{r,R}) \sim \frac{\sigma(S^2)}{\sigma(A_{r,R}) \# \widehat{\EE}(n)}.
\end{gather*}

Our basic approach towards the computation of the variances is to use Parseval's identity to spectrally expand the variances in terms of an orthonormal basis of Laplacian eigenfunctions on $\Gamma \backslash \Hb$ and $S^2$. For $\Gamma \backslash \Hb$, we denote by $\BB_0(\Gamma)$ an orthonormal basis of the space of Maa\ss{} cusp forms, which we may choose to consist of Hecke--Maa\ss{} cusp forms, while for $S^2$, we let $\BB$ denote an orthonormal basis of Laplacian eigenfunctions, which we may assume to be Hecke eigenfunctions. The spectral expansion of the variances involves the square of the absolute value of the Weyl sums
\begin{align}
\label{eq:WDf}
W_{D,f} & \coloneqq \begin{dcases*}
\sum_{z \in \Lambda_D} f(z) & for $D < 0$,	\\
\sum_{\CC \in \Lambda_D} \int_{\CC} f(z) \, ds & for $D > 0$,
\end{dcases*}	\\
\label{eq:WDt}
W_{D,t} & \coloneqq \begin{dcases*}
\sum_{z \in \Lambda_D} E\left(z, \frac{1}{2} + it\right) & for $D < 0$,	\\
\sum_{\CC \in \Lambda_D} \int_{\CC} E\left(z, \frac{1}{2} + it\right) \, ds & for $D > 0$,
\end{dcases*}	\\
\intertext{where $f \in \BB_0(\Gamma)$ and $t \in \R$, and}
\label{eq:Wnphi}
W_{n,\phi} & \coloneqq \sum_{x \in \widehat{\EE}(n)} \phi(x)
\end{align}
for $\phi \in \BB$. Period formul\ae{} allow us to express the square of the absolute value of the Weyl sums in terms of $L$-functions. This leads us to write the variances as sums of $L$-functions weighted by the square of the Selberg--Harish-Chandra transform of the normalised indicator function of the annulus $A_{r,R}$. We explicitly work out the asymptotic behaviour of the Selberg--Harish-Chandra transform, then break up these weighted sums of $L$-functions into dyadic ranges; in this way \hyperref[variancethm1]{Theorems \ref*{variancethm1}}, \ref{aethm1}, \ref{variancethm2}, and \ref{aethm2} are reduced to proving bounds for certain moments of $L$-functions.

\subsection{The Selberg--Harish-Chandra transform for \texorpdfstring{$\Hb$}{H}}

We follow \cite[Chapter 1]{Iwa02}. For $z,w \in \Hb$, set
\[\rho(z,w) \coloneqq \log \frac{\left|z - \overline{w}\right| + |z - w|}{\left|z - \overline{w}\right| - |z - w|}, \qquad u(z,w) \coloneqq \frac{|z - w|^2}{4 \Im(z) \Im(w)} = \sinh^2 \frac{\rho(z,w)}{2}.\]
The function $u : \Hb \times \Hb \to [0,\infty)$ is a point-pair invariant for the symmetric space $\Hb \cong \SL_2(\R) / \SO(2)$; that is, $u(g z, g w) = u(z,w)$ for all $g \in \SL_2(\R)$ and $z,w \in \Hb$. From this, a function $k : [0,\infty) \to \C$ gives rise to a point-pair invariant $k(u(z,w))$ on $\Hb$.

We take $k(u(z,w)) = k_{r,R}(u(z,w))$ to be equal to the indicator function of an annulus of inner radius $r$ and outer radius $R$ centred at a point $w$,
\[A_{r,R}(w) \coloneqq \{z \in \Hb : r \leq \rho(z,w) \leq R\} = \left\{z \in \Hb : \sinh^2 \frac{r}{2} \leq u(z,w) \leq \sinh^2 \frac{R}{2}\right\},\]
normalised by the volume of this annulus,
\[\mu(A_{r,R}) = \mu(A_{r,R}(w)) = 4\pi \left(\sinh^2 \frac{R}{2} - \sinh^2 \frac{r}{2}\right),\]
namely
\[k_{r,R}(u(z,w)) \coloneqq \begin{dcases*}
\dfrac{1}{\mu(A_{r,R})} & if $\sinh^2 \dfrac{r}{2} \leq u(z,w) \leq \sinh^2 \dfrac{R}{2}$,	\\
0 & otherwise.
\end{dcases*}\]

Given $k : [0,\infty) \to \C$, we define the automorphic kernel $K : \Gamma \backslash \Hb \times \Gamma \backslash \Hb \to \C$ by
\[K(z,w) \coloneqq \sum_{\gamma \in \Gamma} k(u(\gamma z,w)).\]
The spectral expansion for the automorphic kernel $K = K_{r,R}$ associated to the point-pair invariant $k = k_{r,R}$ involves a sum over an orthonormal basis $\BB_0(\Gamma)$ of the space of Maa\ss{} cusp forms (which we may choose to consist of Hecke--Maa\ss{} eigenforms), where the inner product is
\[\langle f,g\rangle \coloneqq \int_{\Gamma \backslash \Hb} f(z) \overline{g(z)} \, d\mu(z),\]
and an integral over $t \in \R$ indexing the Eisenstein series $E(z,1/2 + it)$. It also involves the Selberg--Harish-Chandra transform $h_{r,R}$ of $k_{r,R}$. The Selberg--Harish-Chandra transform takes sufficiently well-behaved functions $k : [0,\infty) \to \C$ to functions $h : \R \to \C$ via
\[h(t) \coloneqq 2\pi \int_{0}^{\infty} P_{-\frac{1}{2} + it}(\cosh \rho) k\left(\sinh^2 \frac{\rho}{2}\right) \sinh \rho \, d\rho,\]
where $P_{\lambda}^{\mu}(z)$ denotes the associated Legendre function. In particular,
\begin{equation}
\label{hrRtPeq}
h_{r,R}(t) = \frac{2\pi}{\mu(A_{r,R})} \int_{r}^{R} P_{-\frac{1}{2} + it}(\cosh \rho) \sinh \rho \, d\rho.
\end{equation}

\begin{lemma}[{\cite[Theorems 1.14 and 7.4]{Iwa02}}]
\label{kernelHlemma}
The automorphic kernel $K_{r,R}$ satisfies
\begin{align*}
\int_{\Gamma \backslash \Hb} K_{r,R}(z,w) \, d\mu(z) & = h_{r,R}\left(\frac{i}{2}\right) = 1,	\\
\int_{\Gamma \backslash \Hb} f(z) K_{r,R}(z,w) \, d\mu(z) & = h_{r,R}(t_f) f(w),	\\
\int_{\Gamma \backslash \Hb} E\left(z,\frac{1}{2} + it\right) K_{r,R}(z,w) \, d\mu(z) & = h_{r,R}(t) E\left(w,\frac{1}{2} + it\right)
\end{align*}
for every $f \in \BB_0(\Gamma)$, $t \in \R$, and $w \in \Gamma \backslash \Hb$ and has the $L^2$-spectral expansion
\begin{multline*}
K_{r,R}(z,w) = \frac{1}{\mu(\Gamma \backslash \Hb)} + \sum_{f \in \BB_0(\Gamma)} h_{r,R}(t_f) f(z) \overline{f(w)}	\\
+ \frac{1}{4\pi} \int_{-\infty}^{\infty} h_{r,R}(t) E\left(z,\frac{1}{2} + it\right) \overline{E\left(z,\frac{1}{2} + it\right)} \, dt.
\end{multline*}
\end{lemma}

\subsection{The Selberg--Harish-Chandra transform for \texorpdfstring{$S^2$}{S\80\262}}

We now work on the symmetric space $S^2 \cong \SO(3) / \SO(2)$ instead of $\Hb \cong \SL_2(\R) / \SO(2)$. We follow \cite{LPS86}. Given $z,\zeta \in S^2$, we let
\[\theta(z,\zeta) \coloneqq \arccos z \cdot \zeta, \qquad \tilde{u}(z,\zeta) = \frac{1 - z \cdot \zeta}{2} = \sin^2 \frac{\theta(z,\zeta)}{2},\]
so that $\theta(z,\zeta) \in [0,\pi]$ is the angle subtended at the origin of the vectors $z$ and $\zeta$. The function $\tilde{u} : S^2 \times S^2 \to \C$ is a point-pair invariant. From this, a function $\tilde{k} : [0,1] \to \C$ gives rise to a point-pair invariant $\tilde{k}(\tilde{u}(z,\zeta))$ on $S^2$.

We take $\tilde{k}(\tilde{u}(z,\zeta)) = \tilde{k}_{r,R}(\tilde{u}(z,\zeta))$ to be equal to the indicator function of an annulus of inner radius $r$ and outer radius $R$ centred at a point $\zeta$,
\[A_{r,R}(\zeta) \coloneqq \{z \in S^2 : r \leq \theta(z,\zeta) \leq R\} = \left\{z \in S^2 : \sin^2 \frac{r}{2} \leq \tilde{u}(z,\zeta) \leq \sin^2 \frac{R}{2}\right\},\]
normalised by the volume of this annulus,
\[\sigma(A_{r,R}) = \sigma(A_{r,R}(\zeta)) = 4\pi \left(\sin^2 \frac{R}{2} - \sin^2 \frac{r}{2}\right),\]
namely
\[\tilde{k}_{r,R}(\tilde{u}(z,\zeta)) = \begin{dcases*}
\frac{1}{\sigma(A_{r,R})} & if $\sin^2 \frac{r}{2} \leq \tilde{u}(z,\zeta) \leq \sin^2 \frac{R}{2}$,	\\
0 & otherwise.
\end{dcases*}\]

The spectral expansion for $\tilde{k} = \tilde{k}_{r,R}$ involves a sum over an orthonormal basis $\BB$ of $L^2(S^2)$ consisting of spherical harmonics $\phi$ of degree $m_{\phi} \geq 0$, where the inner product is
\[\langle \phi,\psi\rangle \coloneqq \int_{S^2} \phi(z) \overline{\psi(z)} \, d\sigma(z).\]
It also involves the Selberg--Harish-Chandra transform $\tilde{h}_{r,R}$ of $\tilde{k}_{r,R}$ given by
\[\tilde{h}(m) \coloneqq 2\pi \int_{0}^{\pi} P_m(\cos \theta) \tilde{k}\left(\sin^2 \frac{\theta}{2}\right) \sin \theta \, d\theta,\]
where $P_m(x)$ denotes the Legendre polynomial. In particular,
\begin{equation}
\label{hrRmPeq}
\tilde{h}_{r,R}(m) = \frac{2\pi}{\sigma(A_{r,R})} \int_{r}^{R} P_m(\cos \theta) \sin \theta \, d\theta.
\end{equation}

\begin{lemma}[{\cite[(1.8) and (1.7')]{LPS86}}]
\label{kernelS2lemma}
The kernel $\tilde{k}_{r,R}$ satisfies
\begin{align*}
\int_{S^2} \tilde{k}_{r,R}(z,\zeta) \, dz & = \tilde{h}_{r,R}(0) = 1,	\\
\int_{S^2} \phi(z) \tilde{k}_{r,R}(z,\zeta) \, dz & = \tilde{h}_{r,R}(m_{\phi}) \phi(\zeta)
\end{align*}
for every $\phi \in \BB$ with $m_{\phi} \geq 1$ and $\zeta \in S^2$ and has the $L^2$-spectral expansion
\[\tilde{k}_{r,R}(z,\zeta) = \frac{1}{\sigma(S^2)} + \sum_{\substack{\phi \in \BB \\ m_{\phi} \geq 1}} \tilde{h}_{r,R}(m_{\phi}) \phi(z) \overline{\phi(\zeta)}.\]
\end{lemma}

We also consider the spherical convolution
\[\tilde{k}_1 \ast \tilde{k}_2(\tilde{u}(z,\zeta)) = \int_{S^2} \tilde{k}_1(\tilde{u}(z,w)) \tilde{k}_2(\tilde{u}(w,\zeta)) \, d\sigma(w)\]
of two point-pair invariants on $S^2$. The Selberg--Harish-Chandra transform of the convolution $\tilde{k}_1 \ast \tilde{k}_2$ is the product $\tilde{h}_1(m) \tilde{h}_2(m)$ of the individual Selberg--Harish-Chandra transforms. We will use this in the following setting.

\begin{lemma}
\label{convlemma}
The convolution $\tilde{k}_{r,R} \ast \tilde{k}_{0,\rho}(\tilde{u}(z,\zeta))$ is nonnegative, bounded by $1/\sigma(A_{r,R})$, and satisfies
\[\tilde{k}_{r,R} \ast \tilde{k}_{0,\rho}(\tilde{u}(z,\zeta)) = \begin{dcases*}
\frac{1}{\sigma(A_{r,R})} & if $\sin^2 \frac{r + \rho}{2} \leq \tilde{u}(z,\zeta) \leq \sin^2 \frac{R - \rho}{2}$,	\\
0 & if $\tilde{u}(z,\zeta) \leq \sin^2 \frac{r - \rho}{2}$ or $\tilde{u}(z,\zeta) \geq \sin^2 \frac{R + \rho}{2}$.
\end{dcases*}\]
In particular, for $0 < r,\rho < R$, we have that
\begin{gather*}
\frac{\sigma(B_{R - \rho})}{\sigma(B_R)} \tilde{k}_{0,R - \rho} \ast k_{0,\rho}(\tilde{u}(z,\zeta)) \leq \tilde{k}_{0,R}(\tilde{u}(z,\zeta)) \leq \frac{\sigma(B_{R + \rho})}{\sigma(B_R)} \tilde{k}_{0, R + \rho} \ast \tilde{k}_{0,\rho}(\tilde{u}(z,\zeta)),	\\
\tilde{k}_{r,R}(\tilde{u}(z,\zeta)) \geq \frac{\sigma(A_{r + \rho,R - \rho})}{\sigma(A_{r,R})} \tilde{k}_{r + \rho,R - \rho} \ast k_{0,\rho}(\tilde{u}(z,\zeta)),	\\
\tilde{k}_{r,R}(\tilde{u}(z,\zeta)) \leq \frac{\sigma(A_{{\max\{r - \rho,0\}, R + \rho}})}{\sigma(A_{r,R})} \tilde{k}_{\max\{r - \rho,0\}, R + \rho} \ast \tilde{k}_{0,\rho}(\tilde{u}(z,\zeta))
\end{gather*}
for all $z,\zeta \in S^2$.
\end{lemma}

\begin{proof}
This follows from the triangle inequality for the spherical distance function $\theta(z,\zeta)$.
\end{proof}

The advantage of convolving is that it smooths the point-pair invariant and improves the decay of the Selberg--Harish-Chandra transform. This ensures that the convolved kernel has a spectral expansion on $L^2(S^2)$ that not only converges in $L^2$ but uniformly.

\begin{lemma}[{\cite[(1.7')]{LPS86}}]
\label{kernelS2uniformlemma}
The convolved kernel $\tilde{k}_{r,R} \ast \tilde{k}_{0,\rho}$ has the spectral expansion
\[\tilde{k}_{r,R} \ast \tilde{k}_{0,\rho}(z,\zeta) = \frac{1}{\sigma(S^2)} + \sum_{\substack{\phi \in \BB \\ m_{\phi} \geq 1}} \tilde{h}_{r,R}(m_{\phi}) \tilde{h}_{0,\rho}(m_{\phi}) \phi(z) \overline{\phi(\zeta)},\]
which converges absolutely and uniformly.
\end{lemma}

\subsection{Weyl sums and \texorpdfstring{$L$}{L}-functions}

Our method to deal with the Weyl sums is to relate them to $L$-functions. A famous result of Waldspurger \cite{Wal81} shows that the Weyl sums $W_{D,f}$, $W_{D,t}$ and $W_{n,\phi}$ in \eqref{eq:WDf}, \eqref{eq:WDt}, and \eqref{eq:Wnphi} are each equal, up to certain normalising factors, to Fourier coefficients of half-integral weight forms. This is the key identity via which Duke \cite{Duk88, DS-P90} and Golubeva--Fomenko \cite{GF90} prove the equidistribution of lattice points on the sphere. Waldspurger \cite{Wal85} subsequently proved another identity that is more pertinent for our needs: he showed that the squares of the absolute values of the Weyl sums $W_{D,f}$, $W_{D,t}$ and $W_{n,\phi}$ are equal, up to certain normalising factors, to products of $L$-functions.

We now state an exact formula for the Weyl sums $W_{D,f}$ and $W_{D,t}$ in terms of $L$-functions, which is proven in \cite{DIT16} (and also follows from \cite[Theorem 4.1]{MW09}); it is an explicit form of Waldspurger's formula \cite{Wal85}.

\begin{lemma}[{\cite[Theorems 3 and 5 and (5.17)]{DIT16}}]
\label{DukeWaldspurgerlemma}
For fundamental discriminants $D < 0$,
\begin{align*}
\left|\frac{W_{D,f}}{\# \Lambda_D}\right|^2 & = \frac{\pi^2}{4 \sqrt{|D|} L(1,\chi_D)^2} \frac{L\left(\frac{1}{2},f\right) L\left(\frac{1}{2}, f \otimes \chi_D\right)}{L(1, \sym^2 f)},	\\
\left|\frac{W_{D,t}}{\# \Lambda_D}\right|^2 & = \frac{\pi^2}{2 \sqrt{|D|} L(1,\chi_D)^2} \left|\frac{\zeta\left(\frac{1}{2} + it\right) L\left(\frac{1}{2} + it, \chi_D\right)}{\zeta(1 + 2it)}\right|^2,
\end{align*}
while for fundamental discriminants $D > 0$,
\begin{align*}
\left|\frac{W_{D,f}}{\sum_{\CC \in \Lambda_D} \ell(\CC)}\right|^2 & = \frac{H(t_f)}{8 \sqrt{D} L(1,\chi_D)^2} \frac{L\left(\frac{1}{2},f\right) L\left(\frac{1}{2}, f \otimes \chi_D\right)}{L(1, \sym^2 f)},	\\
\left|\frac{W_{D,t}}{\sum_{\CC \in \Lambda_D} \ell(\CC)}\right|^2 & = \frac{H(t)}{4 \sqrt{D} L(1,\chi_D)^2} \left|\frac{\zeta\left(\frac{1}{2} + it\right) L\left(\frac{1}{2} + it, \chi_D\right)}{\zeta(1 + 2it)}\right|^2,
\end{align*}
where
\begin{equation}
\label{H(t)asympeq}
H(t) \coloneqq \frac{\Gamma\left(\frac{1}{4} + \frac{it}{2}\right)^2 \Gamma\left(\frac{1}{4} - \frac{it}{2}\right)^2}{\Gamma\left(\frac{1}{2} + it\right) \Gamma\left(\frac{1}{2} - it\right)} = \frac{4\pi}{|t| + 1} + O\left(\frac{1}{(|t| + 1)^2}\right).
\end{equation}
\end{lemma}

Here the last line follows from Stirling's formula.

\begin{remark}
The additional decay in $t$ in \eqref{H(t)asympeq} is the source of the strengthening in \hyperref[aethm2]{Theorem \ref*{aethm2} (2)} to hold not just for annuli with inner radii that do not shrink too rapidly but for all annuli, including the degenerate case of balls.
\end{remark}

We also require an explicit form of Waldspurger's formula for the Weyl sums $W_{n,\phi}$. We may choose an orthonormal basis $\BB \ni \phi$ of $L^2(S^2)$ consisting of spherical harmonics of degree $m \geq 0$ that are Hecke eigenfunctions by viewing these as functions on the subspace $\Dgp^0(\R)$ of the Hamiltonion quaternion algebra $\Dgp(\R)$ consisting of elements with trace zero; see \cite[Section 2]{BSS-P03}. The Weyl sum $W_{n,\phi}$ trivially vanishes if $m_{\phi}$ is odd or if $\phi$ is not invariant under the action of the unit group $\OO^{\times}$ of the maximal order $\OO$ of $\Dgp(\Q)$. If $m_{\phi} \geq 2$ is even and $\phi$ is $\OO^{\times}$-invariant, then the Jacquet--Langlands correspondence gives a bijective correspondence between such Hecke eigenfunctions $\phi$ and holomorphic newforms $f = f_{\phi}$ of weight $2 + 2m_{\phi}$ and level $2$. We let $\BB_{\hol}^{\ast}(\Gamma_0(2))$ denote an orthonormal basis of holomorphic newforms of level $2$ and trivial nebentypus.

\begin{lemma}
\label{Waldspurgerlemma}
Let $-n = D \equiv 5 \pmod{8}$ be a negative squarefree fundamental discriminant. Let $\phi \in \BB$ be an $\OO^{\times}$-invariant Hecke eigenfunction of even degree $m_{\phi} \geq 2$, and let $f = f_{\phi} \in \BB_{\hol}^{\ast}(\Gamma_0(2))$ of weight $k_f = 2 + 2m_{\phi}$ denote the corresponding Jacquet--Langlands transfer. Then
\[\left|\frac{W_{n,\phi}}{\# \widehat{\EE}(n)}\right|^2 = \frac{\pi^2}{96 \sqrt{n} L(1,\chi_{-n})^2} \frac{L\left(\frac{1}{2},f\right) L\left(\frac{1}{2},f \otimes \chi_{-n}\right)}{L(1,\sym^2 f)}.\]
\end{lemma}

\begin{proof}
Let $\varphi$ denote the ad\`{e}lic lift of $\phi$ to an automorphic form on $\Zgp(\A_{\Q}) \Dgp^{\times}(\Q) \backslash \Dgp^{\times}(\A_{\Q})$, so that $\varphi$ is the ad\`{e}lic newform in a cuspidal automorphic representation $\pi^{\Dgp}$ of $\Dgp^{\times}(\A_{\Q})$; the Jacquet--Langlands correspondence associates to $\pi^{\Dgp}$ a cuspidal automorphic representation $\pi$ of $\GL_2(\A_{\Q})$ whose ad\`{e}lic newform is the lift of $f = f_{\phi}$. Define the period integral
\[P^{\Dgp}(\varphi) \coloneqq \int\limits_{\A_{\Q}^{\times} E^{\times} \backslash \A_E^{\times}} \varphi(\Psi_{\A_{\Q}}(t)) \, dt\]
where $E = \Q(\sqrt{D})$ and the measure $dt$ is normalised such that the volume of $\A_{\Q}^{\times} E^{\times} \backslash \A_E^{\times}$ is $2 L(1,\chi_D)/\pi$. Here we have fixed an optimal embedding $\Psi : E \hookrightarrow \Dgp(\Q)$ of the ring of integers $\OO_E$ of $E$ into the maximal order of Hurwitz quaternions and tensored with $\A_{\Q}$ to form an embedding $\Psi_{\A_{\Q}} : \A_E \hookrightarrow \Dgp(\A_{\Q})$. The optimal embedding corresponds to a fixed solution $(x_1,x_2,x_3) \in \Z^3$ to the equation $x_1^2 + x_2^2 + x_3^2 = |D|$ via $a + b\sqrt{|D|} \mapsto a + bx_1 i + bx_2 j + bx_3 k$. We refer the reader to \cite{BB20} for further details on this period integral, viewed both ad\`{e}lically and classically.

Up to multiplication by a constant, $P^{\Dgp}(\varphi)$ is precisely the Weyl sum $W_{n,\phi}$. We apply \cite[Theorem 4.1]{MW09} with $F = \Q$ and $E = \Q(\sqrt{D})$, $\Omega$ the trivial character, and $\varphi \in \pi^{\Dgp}$ as above, so that $S'(\pi) = S(\Omega) = \emptyset$, $\Delta_F = 1$, $\Delta_E = |D|$, $c(\Omega) = 1$, $\Ram(\pi) = \{2\}$, and $\Sigma_{\infty}^F = \{\infty\}$ in the notation of \cite{MW09}; this gives us the identity
\[\left|P^{\Dgp}(\varphi)\right|^2 = \frac{\pi}{12 \sqrt{|D|}} \frac{L\left(\frac{1}{2},f\right) L\left(\frac{1}{2},f \otimes \chi_D\right)}{L(1,\sym^2 f)} \int\limits_{\Zgp(\A_{\Q}) \Dgp^{\times}(\Q) \backslash \Dgp^{\times}(\A_{\Q})} |\varphi(g)|^2 \, dg,\]
where the measure $dg$ is normalised such that the volume of $\Zgp(\A_{\Q}) \Dgp^{\times}(\Q) \backslash \Dgp^{\times}(\A_{\Q})$ is $2$. It remains to recall that
\begin{equation}
\# \widehat{\EE}(n) = 
\frac{48 h(D)}{w_D} = \frac{24 \sqrt{n} L(1,\chi_{-n})}{\pi} \text{ for } n \equiv 3 \pmod{8} \text{ with } D = -n.
\end{equation}
and note that with these normalisations,
\[\left|P^{\Dgp}(\varphi)\right|^2 = \frac{1}{144 |D|} \left|W_{n,\phi}\right|^2, \qquad \int\limits_{\Zgp(\A_{\Q}) \Dgp^{\times}(\Q) \backslash \Dgp^{\times}(\A_{\Q})} |\varphi(g)|^2 \, dg = \frac{1}{2\pi}\]
by comparing these measures with $\varphi$ equal to the constant function and using the fact that $\int_{S^2} |\phi(z)|^2 \, d\sigma(z) = 1$.
\end{proof}

\begin{remark}
\label{Rudnickremark}
The generalised Lindel\"{o}f hypothesis implies that $W_{D,f} \ll_{\e} |D|^{1/4 + \e} t_f^{\e}$ for $D < 0$, $W_{D,f} \ll_{\e} D^{1/4 + \e} t_f^{-1/2 + \e}$ for $D > 0$, and $W_{n,\phi} \ll_{\e} n^{1/4 + \e} m_{\phi}^{\e}$. By comparing with the bounds $\# \Lambda_D \gg_{\e} |D|^{1/2 - \e}$ for $D < 0$, $\sum_{\CC \in \Lambda_D} \ell(\CC) \gg_{\e} D^{1/2 - \e}$ for $D > 0$, and $\# \widehat{\EE}(n) \gg_{\e} n^{1/2 - \e}$, we may interpret this as square-root cancellation for individual Weyl sums.
\end{remark}

\begin{remark}
The fact that the squares of the absolute values of the Weyl sums factorise as the product of two $L$-functions, $L(\tfrac{1}{2},f)$ and $L(\tfrac{1}{2},f \otimes \chi_D)$, is crucial to our method. It allows us to separate these $L$-functions when faced with sums of these products of $L$-functions by applying H\"{o}lder's inequality.
\end{remark}

\subsection{Spectral expansions of the variances}

Combining the explicit expressions for the Weyl sums in terms of $L$-functions with the spectral expansions of the kernels $k_{r,R}$ and $\tilde{k}_{r,R}$, we are able to explicitly express the variances as sums of $L$-functions.

\begin{lemma}
\label{DukeVarlemma}
For $D < 0$,
\begin{multline}
\label{VarD<0eq}
\Var(\Lambda_D; A_{r,R}) = \frac{\pi^2 \mu(\Gamma \backslash \Hb)}{4 \sqrt{|D|} L(1,\chi_D)^2} \sum_{f \in \BB_0(\Gamma)} \frac{L\left(\frac{1}{2},f\right) L\left(\frac{1}{2}, f \otimes \chi_D\right)}{L(1, \sym^2 f)} \left|h_{r,R}\left(t_f\right)\right|^2	\\
+ \frac{\pi \mu(\Gamma \backslash \Hb)}{8 \sqrt{|D|} L(1,\chi_D)^2} \int_{-\infty}^{\infty} \left|\frac{\zeta\left(\frac{1}{2} + it\right) L\left(\frac{1}{2} + it, \chi_D\right)}{\zeta(1 + 2it)}\right|^2 \left|h_{r,R}(t)\right|^2 \, dt.
\end{multline}
For $D > 0$,
\begin{multline}
\label{VarD>0eq}
\Var(\Lambda_D; A_{r,R}) = \frac{\mu(\Gamma \backslash \Hb)}{8 \sqrt{D} L(1,\chi_D)^2} \sum_{f \in \BB_0(\Gamma)} \frac{L\left(\frac{1}{2},f\right) L\left(\frac{1}{2}, f \otimes \chi_D\right)}{L(1, \sym^2 f)} H(t_f) \left|h_{r,R}\left(t_f\right)\right|^2	\\
+ \frac{\mu(\Gamma \backslash \Hb)}{16 \pi \sqrt{D} L(1,\chi_D)^2} \int_{-\infty}^{\infty} \left|\frac{\zeta\left(\frac{1}{2} + it\right) L\left(\frac{1}{2} + it, \chi_D\right)}{\zeta(1 + 2it)}\right|^2 H(t) \left|h_{r,R}(t)\right|^2 \, dt.
\end{multline}
Finally, for squarefree $n \equiv 3 \pmod{8}$,
\begin{equation}
\label{Varneq}
\Var(\widehat{\EE}(n); A_{r,R}) = \frac{\pi^2 \sigma(S^2)}{96 \sqrt{n} L(1,\chi_{-n})^2} \sum_{\substack{f \in \BB_{\hol}^{\ast}(\Gamma_0(2)) \\ k_f \equiv 2 \hspace{-.25cm} \pmod{4}}} \frac{L\left(\frac{1}{2},f\right) L\left(\frac{1}{2},f \otimes \chi_{-n}\right)}{L(1,\sym^2 f)} \left|\tilde{h}_{r,R}\left(\frac{k_f}{2} - 1\right)\right|^2.
\end{equation}
\end{lemma}

\begin{proof}
We prove first prove \eqref{VarD<0eq}. Recalling \eqref{VarD<0defeq}, we write the left-hand side of \eqref{VarD<0eq} as
\[\frac{\mu(\Gamma \backslash \Hb)}{(\# \Lambda_D)^2} \sum_{z_1,z_2 \in \Lambda_D} \int_{\Gamma \backslash \Hb} K_{r,R}(z_1,w) \overline{K_{r,R}(z_2,w)} \, d\mu(w) - \frac{2}{\# \Lambda_D} \sum_{z \in \Lambda_D} \int_{\Gamma \backslash \Hb} K_{r,R}(z,w) \, d\mu(w) + 1.\]
We apply Parseval's identity to spectrally expand the first integral and use \hyperref[kernelHlemma]{Lemma \ref*{kernelHlemma}} to see that this is
\[\mu(\Gamma \backslash \Hb) \sum_{f \in \BB_0(\Gamma)} \left|\frac{W_{D,f}}{\# \Lambda_D}\right|^2 \left|h_{r,R}(t_f)\right|^2 + \frac{\mu(\Gamma \backslash \Hb)}{4\pi} \int_{-\infty}^{\infty} \left|\frac{W_{D,t}}{\# \Lambda_D}\right|^2 \left|h_{r,R}(t)\right|^2 \, dt.\]
The identity \eqref{VarD<0eq} then follows from \hyperref[DukeWaldspurgerlemma]{Lemma \ref*{DukeWaldspurgerlemma}}. The same method yields \eqref{VarD>0eq}, recalling \eqref{VarD>0defeq}, and also yields \eqref{Varneq}, recalling \eqref{Varndefeq}, applying \hyperref[kernelS2lemma]{Lemma \ref*{kernelS2lemma}} in place of \hyperref[kernelHlemma]{Lemma \ref*{kernelHlemma}}, and identifying $\phi \in \BB$ with $f \in \BB_{\hol}^{\ast}(\Gamma_0(2))$.
\end{proof}

\subsection{Bounds and asymptotics for the Selberg--Harish-Chandra transform}

To understand the behaviour of the Selberg--Harish-Chandra transforms $h_{r,R}(t)$ and $\tilde{h}_{r,R}(m)$ for various ranges of $r$, $R$, $t$, and $m$, we must first understand the uniform behaviour of the associated Legendre functions $P_{-1/2 + it}(\cosh \rho)$ and $P_m(\cos \theta)$. Hilb's formula relates these functions to the Bessel function.

\begin{lemma}[Hilb's Formula]
Fix $\e > 0$. For $t \in \R$ and $0 < \rho < 1/\e$,
\begin{equation}
\label{associatedLegendreboundseq}
P_{-\frac{1}{2} + it}(\cosh \rho) = \sqrt{\frac{\rho}{\sinh \rho}} J_0(\rho t) + \begin{dcases*}
O(\rho^2) & for $|t| \leq \dfrac{1}{\rho}$,	\\
O_{\e}\left(\frac{\sqrt{\rho}}{|t|^{3/2}}\right) & for $|t| \geq \dfrac{1}{\rho} \geq \e$.
\end{dcases*}
\end{equation}
For $m \in \N$ and $0 < \theta < \pi - \e$,
\begin{equation}
\label{Legendreboundseq}
P_m(\cos \theta) = \sqrt{\frac{\theta}{\sin \theta}} J_0\left(\theta \left(m + \frac{1}{2}\right)\right) + \begin{dcases*}
O(\theta^2) & for $m \leq \dfrac{1}{\theta}$,	\\
O_{\e}\left(\frac{\sqrt{\theta}}{m^{3/2}}\right) & for $m \geq \dfrac{1}{\theta} \geq \dfrac{1}{\pi - \e}$.
\end{dcases*}
\end{equation}
\end{lemma}

\begin{proof}
This follows via the Liouville--Stekloff method. For the Legendre polynomial $P_m(\cos \theta)$, this is \cite[Theorem 8.21.6]{Sze75}; the proof is given in \cite[Section 8.62]{Sze75}. The same method yields \eqref{associatedLegendreboundseq} with minimal modifications.
\end{proof}

We use this to prove the following.

\begin{lemma}[{Cf.~\cite[Lemma 2.4]{Cha96}, \cite[(2.13)]{LPS86}}]
\label{hrRasymplemma}
Suppose that $R - r \ll r < R \leq \pi - \e$ for some fixed $\e > 0$. Then
\begin{equation}
\label{hrRtupperboundseq}
h_{r,R}(t) \ll \begin{dcases*}
1 & for $|t| \leq \dfrac{1}{r}$,	\\
\frac{1}{\sqrt{r |t|}} & for $\dfrac{1}{r} \leq |t| \leq \dfrac{1}{R - r}$,	\\
\frac{1}{\sqrt{r} (R - r) |t|^{3/2}} & for $|t| \geq \dfrac{1}{R - r}$
\end{dcases*}
\end{equation}
for $t \in \R$, while for $m \in \N$,
\begin{equation}
\label{hrRmupperboundeq}
\tilde{h}_{r,R}(m) \ll \begin{dcases*}
1 & for $m \leq \dfrac{1}{r}$,	\\
\frac{1}{\sqrt{r m}} & for $\dfrac{1}{r} \leq m \leq \dfrac{1}{R - r}$,	\\
\frac{1}{\sqrt{r} (R - r) m^{3/2}} & for $m \geq \dfrac{1}{R - r}$.
\end{dcases*}
\end{equation}
Moreover, for $t \in \R$,
\begin{multline}
\label{hrRtrefinedeq}
h_{r,R}(t)^2 = \frac{8}{\sinh \frac{R - r}{2} \mu(A_{r,R})} \frac{1}{|t|^3} \sin^2 \frac{(R - r) t}{2} \sin^2 \frac{(R + r) t}{2}	\\
+ \begin{dcases*}
O\left(\frac{1}{r^3 |t|^3}\right) & for $\dfrac{1}{r} \leq |t| \leq \dfrac{1}{R - r}$,	\\
O\left(\frac{1}{r^3 (R - r)^2 |t|^5}\right) & for $|t| \geq \dfrac{1}{R - r}$,
\end{dcases*}
\end{multline}
and for $m \in \N$,
\begin{multline}
\label{hrRmrefinedeq}
\tilde{h}_{r,R}(m)^2 = \frac{8}{\sin \frac{R - r}{2} \sigma(A_{r,R})} \frac{1}{\left(m + \frac{1}{2}\right)^3} \sin^2 \frac{(R - r) \left(m + \frac{1}{2}\right)}{2} \sin^2 \frac{(R + r) \left(m + \frac{1}{2}\right)}{2}	\\
+ \begin{dcases*}
O\left(\frac{1}{r^3 m^3}\right) & for $\dfrac{1}{r} \leq m \leq \dfrac{1}{R - r}$,	\\
O\left(\frac{1}{r^3 (R - r)^2 m^5}\right) & for $m \geq \dfrac{1}{R - r}$.
\end{dcases*}
\end{multline}
\end{lemma}

\begin{proof}
From \eqref{hrRtPeq} and \eqref{associatedLegendreboundseq}, we have that
\[h_{r,R}(t) = \frac{2\pi}{\mu(A_{r,R})} \int_{r}^{R} \sqrt{\rho \sinh \rho} J_0(\rho t) \, d\rho + \begin{dcases*}
O(r^2) & for $|t| \leq \dfrac{1}{r}$,	\\
O\left(\frac{\sqrt{r}}{|t|^{3/2}}\right) & for $|t| \geq \dfrac{1}{r}$.
\end{dcases*}\]
We use the bounds
\begin{equation}
\label{J0zboundseq}
J_0(x) = \begin{dcases*}
1 + O(x^2) & for $|x| \leq 1$,	\\
\sqrt{\frac{2}{\pi |x|}} \cos\left(|x| - \frac{\pi}{4}\right) + O\left(\frac{1}{|x|^{3/2}}\right) & for $|x| \geq 1$
\end{dcases*}
\end{equation}
for $x \in \R$ \cite[8.441.1 and 8.451.1]{GR15}, which immediately gives the desired upper bound for $|t| \leq 1/r$. For $|t| \geq 1/r$, we use \eqref{J0zboundseq} and then integrate by parts, antidifferentiating the cosine term. After some simple manipulations, we obtain \eqref{hrRtrefinedeq}; the desired upper bounds for $h_{r,R}(t)$ in the regimes $1/r \leq |t| \leq 1/(R - r)$ and $|t| \geq 1/(R - r)$ then follow immediately. Finally, the same method works for $\tilde{h}_{r,R}(m)$, using \eqref{hrRmPeq} and \eqref{Legendreboundseq} in place of \eqref{hrRtPeq} and \eqref{associatedLegendreboundseq}.
\end{proof}

A similar argument may be used for when $r \ll R - r \ll 1$, including the degenerate case of balls, so that $r = 0$.

\begin{lemma}
\label{hrRupperboundslemma}
Suppose that $r \ll R - r \leq R \leq \pi - \e$ for some fixed $\e > 0$. For $t \in \R$ and for $m \in \N$,
\begin{equation}
\label{hrRupperboundseq}
h_{r,R}(t) \ll \begin{dcases*}
1 & for $|t| \leq \dfrac{1}{R}$,	\\
\frac{1}{R^{3/2} |t|^{3/2}} & for $|t| \geq \dfrac{1}{R}$,
\end{dcases*} \qquad \tilde{h}_{r,R}(m) \ll \begin{dcases*}
1 & for $m \leq \dfrac{1}{R}$,	\\
\frac{1}{R^{3/2} m^{3/2}} & for $m \geq \dfrac{1}{R}$.
\end{dcases*}
\end{equation}
\end{lemma}

\subsection{Bounds for moments of \texorpdfstring{$L$}{L}-functions}

Finally, we require bounds for moments of $L$-functions in dyadic ranges.

\begin{proposition}
\label{momentprop}
Let $D$ be a squarefree fundamental discriminant, and let $\chi_D$ denote the quadratic character modulo $|D|$.
\begin{enumerate}[leftmargin=*]
\item[\textnormal{(1)}] For $T \geq 1$, we have that
\begin{multline*}
\sum_{\substack{f \in \BB_0(\Gamma) \\ T \leq t_f \leq 2T}} \frac{L\left(\frac{1}{2},f\right) L\left(\frac{1}{2},f \otimes \chi_D\right)}{L(1,\sym^2 f)} + \frac{1}{2\pi} \int\limits_{T \leq |t| \leq 2T} \left|\frac{\zeta\left(\frac{1}{2} + it\right) L\left(\frac{1}{2} + it,\chi_D\right)}{\zeta(1 + 2it)}\right|^2 \, dt	\\
\ll_{\e} \begin{dcases*}
|D|^{\frac{1}{3} + \e} T^{2 + \e} & for $T \ll |D|^{\frac{1}{12}}$,	\\
|D|^{\frac{1}{2} + \e} & for $|D|^{\frac{1}{12}} \ll T \ll |D|^{\frac{1}{4}}$,	\\
|D|^{\e} T^{2 + \e} & for $T \gg |D|^{\frac{1}{4}}$.
\end{dcases*}
\end{multline*}
\item[\textnormal{(2)}] For $D < 0$ and $T \geq 1$, we have that
\[\sum_{\substack{f \in \BB_{\hol}^{\ast}(\Gamma_0(2)) \\ T \leq k_f \leq 2T \\ k_f \equiv 2 \hspace{-.25cm} \pmod{4}}} \frac{L\left(\frac{1}{2},f\right) L\left(\frac{1}{2},f \otimes \chi_D\right)}{L(1,\sym^2 f)} \ll_{\e} \begin{dcases*}
|D|^{\frac{1}{3} + \e} T^{2 + \e} & for $T \ll |D|^{\frac{1}{12}}$,	\\
|D|^{\frac{1}{2} + \e} & for $|D|^{\frac{1}{12}} \ll T \ll |D|^{\frac{1}{4}}$,	\\
|D|^{\e} T^{2 + \e} & for $T \gg |D|^{\frac{1}{4}}$.
\end{dcases*}\]
\end{enumerate}
\end{proposition}

The proof of \hyperref[momentprop]{Proposition \ref*{momentprop}} is given in \hyperref[Proofmomentsboundsect]{Section \ref*{Proofmomentsboundsect}}. These bounds imply subconvexity for the associated $L$-functions, as we shall expand upon in \hyperref[connectionssect]{Section \ref*{connectionssect}}.

\begin{remark}
For $T \gg |D|^{1/4}$, \hyperref[momentprop]{Proposition \ref*{momentprop}} implies bounds that are as strong as the generalised Lindel\"{o}f hypothesis on average. Equivalently, \hyperref[momentprop]{Proposition \ref*{momentprop}} implies square-root cancellation on average for the Weyl sums; cf.~\hyperref[Rudnickremark]{Remark \ref*{Rudnickremark}}.
\end{remark}

\section{Proofs}
\label{proofsect}

In this section, we prove the results stated in \hyperref[introsect]{Section \ref*{introsect}} (except for our result for the variances, namely \hyperref[variancethm1]{Theorems \ref*{variancethm1}} and \ref{variancethm2}) assuming \hyperref[momentprop]{Proposition \ref*{momentprop}} and the bound \eqref{PYholcubiceq}. We defer the proofs of \hyperref[variancethm1]{Theorems \ref*{variancethm1}} and \ref{variancethm2} to \hyperref[asympsect]{Section \ref*{asympsect}}; they require delicate improvements of \hyperref[momentprop]{Proposition \ref*{momentprop}} involving asymptotics for these moments of $L$-functions weighted by particular choices of test functions.

\begin{proof}[Proof of {\hyperref[Youngthm]{Theorem \ref*{Youngthm}}}]
Via \hyperref[convlemma]{Lemmata \ref*{convlemma}} and \ref{kernelS2uniformlemma}, we have that for $0 < \rho < R$,
\[\frac{\sigma(S^2)}{\sigma(B_R)} \frac{\# (\widehat{\EE}(n) \cap B_R(w))}{\# \widehat{\EE}(n)} \geq \frac{\sigma(B_{R - \rho})}{\sigma(B_R)} + \frac{\sigma(B_{R - \rho}) \sigma(S^2)}{\sigma(B_R)} \sum_{\substack{\phi \in \BB \\ m_{\phi} \geq 1}} \tilde{h}_{0,R - \rho}(m_{\phi}) \tilde{h}_{0,\rho}(m_{\phi}) \frac{W_{n,\phi}}{\# \widehat{\EE}(n)} \overline{\phi(w)}.\]
We claim that
\begin{equation}
\label{hhsumeq}
\frac{\sigma(B_{R - \rho}) \sigma(S^2)}{\sigma(B_R)} \sum_{\substack{\phi \in \BB \\ m_{\phi} \geq 1}} \tilde{h}_{0,R - \rho}(m_{\phi}) \tilde{h}_{0,\rho}(m_{\phi}) \frac{W_{n,\phi}}{\# \widehat{\EE}(n)} \overline{\phi(w)} \ll_{\e} \frac{1}{R^{3/2 + \e} \rho^{1/2 + \e} n^{1/12 - \e}}.
\end{equation}
To prove this, we use the triangle inequality and replace every element inside the sum with its absolute value. The Selberg--Harish-Chandra transforms $\tilde{h}_{0,R - \rho}$ and $\tilde{h}_{0,\rho}$ may be bounded via \eqref{hrRupperboundseq}, while \hyperref[Waldspurgerlemma]{Lemma \ref*{Waldspurgerlemma}} expresses the square of the absolute value of the Weyl sum in terms of $L$-functions. We then break up this sum into dyadic ranges $m_{\phi} \in [\frac{T}{2} - 1,T - 1]$ and apply H\"{o}lder's inequality with exponents $(2,4,6,12)$. We use the local Weyl law to see that
\[\sum_{\substack{\phi \in \BB \\ \frac{T}{2} - 1 \leq m_{\phi} \leq T - 1}} |\phi(w)|^2 \ll T^2.\]
The large sieve in conjunction with the approximate functional equation yields
\[\sum_{\substack{f \in \BB_{\hol}^{\ast}(\Gamma_0(2)) \\ T \leq k_f \leq 2T \\ k_f \equiv 2 \hspace{-.25cm} \pmod{4}}} \frac{L\left(\frac{1}{2},f\right)^2}{L(1,\sym^2 f)} \ll_{\e} T^{2 + \e},\]
where we have identified $\phi \in \BB$ with $f \in \BB_{\hol}^{\ast}(\Gamma_0(2))$. We shall show in \hyperref[lem:PYcubic]{Lemma \ref*{lem:PYcubic} (2)}, from work of Petrow and Young \cite{PY19,PY20}, that we have the bound
\[\sum_{\substack{f \in \BB_{\hol}^{\ast}(\Gamma_0(2)) \\ T \leq k_f \leq 2T \\ k_f \equiv 2 \hspace{-.25cm} \pmod{4}}} \frac{L\left(\frac{1}{2},f \otimes \chi_{-n}\right)^3}{L(1,\sym^2 f)} \ll_{\e} n^{1 + \e} T^{2 + \e}.\]
Finally, since there are $\ll k$ elements of $\BB_{\hol}^{\ast}(\Gamma_0(2))$ of weight $k$, we have the bound
\[\sum_{\substack{f \in \BB_{\hol}^{\ast}(\Gamma_0(2)) \\ T \leq k_f \leq 2T \\ k_f \equiv 2 \hspace{-.25cm} \pmod{4}}} \frac{1}{L(1,\sym^2 f)} \ll T^2.\]
Combined, we obtain \eqref{hhsumeq}. A similar argument may be used with $\tilde{h}_{0,R + \rho}$ in place of $\tilde{h}_{0,R - \rho}$; recalling \hyperref[convlemma]{Lemma \ref*{convlemma}}, noting that $\sigma(B_{R \pm \rho}) = \sigma(B_R) + O(R\rho)$, and taking $\rho = R^{-1/3} n^{-1/18}$, we deduce that
\[\frac{\sigma(S^2)}{\sigma(B_R)} \frac{\# (\widehat{\EE}(n) \cap B_R(w))}{\# \widehat{\EE}(n)} = 1 + O_{\e}\left(\frac{1}{R^{4/3 + \e} n^{1/18 - \e}}\right).\]
This proves the desired unconditional result.

For the conditional result, the generalised Lindel\"{o}f hypothesis bounds $W_{\phi,n} / \# \widehat{\EE}(n)$ by $O_{\e}(m_{\phi}^{\e} n^{-1/4 + \e})$, at which point we may use the Cauchy--Schwarz inequality and the local Weyl law to see that
\[\frac{\sigma(S^2)}{\sigma(B_R)} \frac{\# (\widehat{\EE}(n) \cap B_R(w))}{\# \widehat{\EE}(n)} = 1 + O_{\e}\left(\frac{1}{R^{4/3 + \e} n^{1/6 - \e}}\right).\qedhere\]
\end{proof}

\begin{proof}[Proof of {\hyperref[aethm1]{Theorems \ref*{aethm1}} and \ref{aethm2}}]
First let us deal with the proof of \hyperref[aethm2]{Theorem \ref*{aethm2} (1)}. Via Chebyshev's inequality, it suffices to prove that $\Var(\Lambda_D; A_{r,R}) = o(1)$. To prove this bound, we use the spectral expansion in \hyperref[DukeVarlemma]{Lemma \ref*{DukeVarlemma}} together with the identities for the Weyl sums in terms of $L$-functions in \hyperref[DukeWaldspurgerlemma]{Lemma \ref*{DukeWaldspurgerlemma}} and the upper bounds for the Selberg--Harish-Chandra transform in \hyperref[hrRasymplemma]{Lemmata \ref*{hrRasymplemma}} and \ref{hrRupperboundslemma}. For the case $R - r \ll r \ll 1$, this reduces the problem to showing that
\begin{multline*}
\sum_{\substack{f \in \BB_0(\Gamma) \\ t_f \leq \frac{1}{r}}} \frac{L\left(\frac{1}{2},f\right) L\left(\frac{1}{2},f \otimes \chi_D\right)}{L(1,\sym^2 f)} + \frac{1}{2\pi} \int\limits_{|t| \leq \frac{1}{r}} \frac{\left|\zeta\left(\frac{1}{2} + it\right) L\left(\frac{1}{2} + it,\chi_D\right)\right|}{\left|\zeta(1 + 2it)\right|^2} \, dt	\\
+ \frac{1}{r} \sum_{\substack{f \in \BB_0(\Gamma) \\ \frac{1}{r} \leq t_f \leq \frac{1}{R - r}}} \frac{L\left(\frac{1}{2},f\right) L\left(\frac{1}{2},f \otimes \chi_D\right)}{t_f L(1,\sym^2 f)} + \frac{1}{r} \frac{1}{2\pi} \int\limits_{\frac{1}{r} \leq |t| \leq \frac{1}{R - r}} \frac{\left|\zeta\left(\frac{1}{2} + it\right) L\left(\frac{1}{2} + it,\chi_D\right)\right|}{|t| \left|\zeta(1 + 2it)\right|^2} \, dt	\\
+ \frac{1}{r(R - r)^2} \sum_{\substack{f \in \BB_0(\Gamma) \\ t_f \geq \frac{1}{R - r}}} \frac{L\left(\frac{1}{2},f\right) L\left(\frac{1}{2},f \otimes \chi_D\right)}{t_f^3 L(1,\sym^2 f)} + \frac{1}{r(R - r)^2} \frac{1}{2\pi} \int\limits_{|t| \geq \frac{1}{R - r}} \frac{\left|\zeta\left(\frac{1}{2} + it\right) L\left(\frac{1}{2} + it,\chi_D\right)\right|}{|t|^3 \left|\zeta(1 + 2it)\right|^2} \, dt 
\end{multline*}
is $O(|D|^{1/2 - \alpha})$ for some $\alpha > 0$. In turn, this estimate is proven by breaking up these terms into dyadic ranges and applying \hyperref[momentprop]{Proposition \ref*{momentprop} (1)}. The case $r \ll R - r \ll 1$ is similar. \hyperref[aethm2]{Theorem \ref*{aethm2} (2)} follows by the same method, noting that we must additionally multiply the Maa\ss{} cusp form terms by $t_f^{-1}$ and the Eisenstein terms by $(|t| + 1)^{-1}$ due to \eqref{H(t)asympeq}. Finally, \hyperref[aethm1]{Theorem \ref*{aethm1}} follows similarly, using \hyperref[Waldspurgerlemma]{Lemma \ref*{Waldspurgerlemma}} in place of \hyperref[DukeWaldspurgerlemma]{Lemma \ref*{DukeWaldspurgerlemma}} and \hyperref[momentprop]{Proposition \ref*{momentprop} (2)} in place of \hyperref[momentprop]{Proposition \ref*{momentprop} (1)}.
\end{proof}

\begin{proof}[Proof of {\hyperref[pigeonholethm]{Theorem \ref*{pigeonholethm} (1)}}]
We observe that
\[\left\{w \in S^2 : \widehat{\mathcal{E}}(n) \cap B_{R}(w) \neq \emptyset\right\} \subset \bigcup_{w \in S^2} B_{R}(w),\]
while for $0 < c < 1$,
\begin{multline*}
\left\{w \in S^2 : \left|\frac{\sigma(S^2)}{\sigma(B_{R})} \frac{\# (\widehat{\mathcal{E}}(n) \cap B_{R}(w))}{\# \mathcal{E}(n)} - 1\right| > c\right\} \supset \left\{w \in S^2 : \widehat{\mathcal{E}}(n) \cap B_{R}(w) = \emptyset\right\}	\\
= S^2 \setminus \left\{w \in S^2 : \widehat{\mathcal{E}}(n) \cap B_{R}(w) \neq \emptyset\right\}.
\end{multline*}
This yields \hyperref[pigeonholethm]{Theorem \ref*{pigeonholethm} (1)}.
\end{proof}

In order to prove \hyperref[pigeonholethm]{Theorem \ref*{pigeonholethm} (2)}, we will need the following simple lemma.
\begin{lemma} \label{pomykala}
Let $\e > 0$ be given. Let $f \in \BB_{\hol}^{\ast}(\Gamma_0(2))$ be a holomorphic newform of weight $k$. Then
\begin{equation}
\label{eq:simplefunct}
\sum_{\substack{n \leq X \\ n \equiv 3 \hspace{-.25cm} \pmod{8} \\ \textup{squarefree}}} L(\tfrac 12, f \otimes \chi_{-n}) \ll_{\e} X^{1 + \varepsilon} k^{1/2 + \varepsilon}.
\end{equation}
\end{lemma}

\begin{proof}
Using the approximate functional equation \cite[Theorem 5.3]{IK04} and splitting into dyadic intervals, we can bound the left-hand side of \eqref{eq:simplefunct} by
\[
\ll_{\varepsilon} X^{\varepsilon} + X^{\varepsilon} \sup_{\substack{M \leq (k X)^{1 + \eta} \\ |u| \leq (k X)^{\varepsilon}}} \sum_{\substack{n \leq X \\ n \equiv 3 \hspace{-.25cm} \pmod{8} \\ \text{squarefree}}} \Big | \sum_{M \leq m \leq 2M} \frac{\lambda_f(m) \chi_{-n}(m)}{m^{1/2 + iu}} \Big | 
\]
for any given $\varepsilon > 0$. By the Cauchy--Schwarz inequality and Heath-Brown's quadratic large sieve \cite[Corollary 3]{HB95}, the above is
\[
\ll_{\varepsilon} X^{\varepsilon} \sup_{\substack{M \leq (k X)^{1 + \varepsilon}}}
\sqrt{X} \sqrt{(X + M )} X^{\varepsilon},
\]
and the claim follows. 
\end{proof}

\begin{proof}[Proof of {\hyperref[pigeonholethm]{Theorem \ref*{pigeonholethm} (2)}}]
For $\varepsilon, \delta, \kappa > 0$, consider the set $\mathcal{D}_{\varepsilon, \delta, \kappa}(X)$ of squarefree integers $n \equiv 3 \pmod{8}$ in $[1, X]$ such that either of these conditions hold:
\begin{enumerate}
\item \label{prop1} There exists a holomorphic newform $f \in \BB_{\hol}^{\ast}(\Gamma_0(2))$ of weight $k_f \leq (\log X)^{\delta^2 / 4 - \varepsilon}$ such that
\[
L(\tfrac 12 , f \otimes \chi_{-n}) > (\log X)^{\delta - \tfrac 12}.
\]
\item \label{prop2} We have
\[
\sum_{\substack{f \in \BB_{\hol}^{\ast}(\Gamma_0(2)) \\ k_f \equiv 2 \hspace{-.25cm} \pmod{4} \\ k_f \geq (\log X)^{\kappa + 10 \varepsilon}}} k_f^{-3} \cdot \frac{L(\tfrac 12, f) L(\tfrac 12 , f \otimes \chi_{-n})}{L(1,\sym^2 f)} > (\log X)^{-\kappa - 2\varepsilon}.
\]
\item \label{prop3} We have $L(1, \chi_{-n}) < (\log X)^{-\varepsilon}$. 
\end{enumerate}

Notice that by \eqref{Varneq} and \hyperref[hrRupperboundslemma]{Lemma \ref*{hrRupperboundslemma}}, for $R = (\log X)^{-\kappa}$, 
\begin{multline*}
\text{Var}(\widehat{\mathcal{E}}(n); B_{R}) \ll \frac{1}{\sqrt{n} L(1, \chi_{-n})^2} \sum_{\substack{f \in \BB_{\hol}^{\ast} (\Gamma_0(2)) \\ k_f \equiv 2 \hspace{-.25cm} \pmod{4} \\ k_f \leq (\log X)^{\kappa + 10 \varepsilon}}} \frac{L(\tfrac 12, f) L(\tfrac 12 , f \otimes \chi_{-n})}{L(1,\sym^2 f)} \\
+ \frac{1}{\sqrt{n} L(1, \chi_{-n})} \sum_{\substack{f \in \BB_{\hol}^{\ast}(\Gamma_0(2)) \\ k_f \equiv 2 \hspace{-.25cm} \pmod{4} \\ k_f \geq (\log X)^{\kappa + 10 \varepsilon}}} \frac{(\log X)^{3 \kappa}}{k_f^3} \frac{L(\tfrac 12, f) L(\tfrac 12 , f \otimes \chi_{-n})}{L(1,\sym^2 f)}.
\end{multline*}
Therefore using the first moment estimate
\begin{equation} \label{firstmoment}
\sum_{\substack{f \in \BB_{\hol}^{\ast}(\Gamma_0(2)) \\ k_f \leq K}} \frac{L(\tfrac 12, f)}{L(1,\sym^2 f)} \ll K^2,
\end{equation}
we find that for squarefree integers $n \equiv 3 \pmod{8}$ with $n \in [1, X] \backslash \mathcal{D}_{\varepsilon, \delta, \kappa}(X)$, as long as $\kappa + 10 \varepsilon < \delta^2 / 4 - \varepsilon$, 
\[
\Var(\widehat{\mathcal{E}}(n); B_{R}) \ll \frac{1}{\sqrt{n}L (1, \chi_{-n})} \cdot \Big ( (\log X)^{2 \kappa + \delta - 1/2 + 21 \varepsilon} + (\log X)^{3 \kappa - \kappa - \varepsilon} \Big ), 
\]
and since $\sigma(B_R) \asymp R^2 \asymp (\log X)^{- 2 \kappa}$, the above expression is $o ( (\sigma(B_R) \# \mathcal{E}(n))^{-1} ) $ provided that $\delta$ and $\kappa$ are chosen so that $2 \kappa + \delta - \tfrac 12 + 21 \varepsilon \leq 2 \kappa - \varepsilon$ (and we maintain our previous condition $\kappa + 10 \varepsilon < \delta^2 / 4 - \varepsilon$). In particular, for any $\kappa < \tfrac{1}{16}$, an admissible choice of $\delta, \varepsilon > 0$ can be made. 

We will now conclude the proof by showing that $|\mathcal{D}_{\varepsilon, \delta, \kappa}(X)| = o(X)$ as $X \to \infty$ for any given $\varepsilon, \kappa > 0$ and $1 > \delta > 0$.
By the union bound, it suffices to show that for any given $\varepsilon, \kappa > 0$ and $1 > \delta > 0$, each of the properties \eqref{prop1}, \eqref{prop2}, \eqref{prop3} holds for at most a density zero subset of squarefree integers $n \equiv 3 \pmod{8}$ with $n \leq X$.

It is a classical result that the third property \eqref{prop3} holds at most for a density zero subset of squarefree integers $n \equiv 3 \pmod{8}$; see \cite{Ell73}.

Now let us show that the first property holds for at most a zero density subset of squarefree integers $n \equiv 3 \pmod{8}$ with $n \leq X$. Let $f \in \BB_{\hol}^{\ast}(\Gamma_0(2))$. By Chernoff's inequality applied to a minor variant of \cite[Theorem 1]{RS15}, the number of squarefree integers $n \equiv 3 \pmod{8}$ with $n \leq X$ for which $L(\tfrac 12, f \otimes \chi_{-n}) > (\log X)^{\delta - 1/2}$ is bounded by $O(X (\log X)^{-\delta^2 / 2})$. Therefore by the union bound, the number of squarefree integers $n \equiv 3 \pmod{8}$ with $n \leq X$ for which there exists a holomorphic newform $f \in \BB_{\hol}^{\ast}(\Gamma_0(2))$ with $k_f \leq (\log X)^{\delta^2 / 4 - \varepsilon}$ and $L(\tfrac 12, f \otimes \chi_{-n}) > (\log X)^{\delta - 1/2}$ is $O(X (\log X)^{-\varepsilon})$. 

Finally, by Chebyshev's inequality, the number of squarefree integers $n \equiv 3 \pmod{8}$ with $n \leq X$ for which the second property \eqref{prop2} holds is bounded by
\begin{equation} \label{eqthird}
(\log X)^{\kappa + 2 \varepsilon} \sum_{K > (\log X)^{\kappa + 10 \varepsilon}} \frac{1}{K^3}
\sum_{\substack{f \in \BB_{\hol}^{\ast}(\Gamma_0(2)) \\ K \leq k_f \leq 2K \\ k_f \equiv 2 \hspace{-.25cm} \pmod{4}}} \frac{L(\tfrac 12, f)}{L(1,\sym^2 f)} \sum_{\substack{n \leq X \\ n \equiv 3 \hspace{-.25cm} \pmod{8} \\ \text{squarefree}}} L(\tfrac 12, f \otimes \chi_{-n})
\end{equation}
with $K$ running over powers of two.

It remains therefore to estimate the above expression. It follows from a minor variant of \cite[Proposition 2]{RS15} with $u = 1$ that for $f \in \BB_{\hol}^{\ast}(\Gamma_0(2))$ of weight $k_f \leq X^{1/100}$, 
\[
\sum_{\substack{n \leq X \\ n \equiv 3 \hspace{-.25cm} \pmod{8} \\ \text{squarefree}}} L(\tfrac 12, f \otimes \chi_{-n}) \ll L(1,\sym^2 f) X,
\]
while it follows from \hyperref[pomykala]{Lemma \ref*{pomykala}} that for $k_f > X^{1/100}$ and for any fixed $\eta > 0$, 
\[
\sum_{\substack{n \leq X \\ n \equiv 3 \hspace{-.25cm} \pmod{8} \\ \text{squarefree}}} L(\tfrac 12, f \otimes \chi_{-n}) \ll_{\eta} X^{1 + \eta} k_f^{1/2 + \eta}.
\]
We therefore split the sum according to whether $K \leq X^{1/100}$ or $K > X^{1/100}$. As a result, for any $\eta > 0$, we bound \eqref{eqthird} by
\begin{align}\label{eqfour}
\begin{aligned} \ll_{\eta} & (\log X)^{2 \kappa + 2 \varepsilon} \sum_{X^{1/100} > K > (\log X)^{2 \kappa + 10 \varepsilon}} \frac{1}{K^3} \sum_{\substack{f \in \BB_{\hol}^{\ast}(\Gamma_0(2)) \\ K \leq k_f \leq 2K \\ k_f \equiv 2 \hspace{-.25cm} \pmod{4}}} \frac{L(\tfrac 12, f)}{L(1,\sym^2 f)} \cdot L(1,\sym^2 f) X \\
& + (\log X)^{2 \kappa + 2 \varepsilon} \sum_{X \geq K > X^{1/100}} \frac{1}{K^3} \sum_{\substack{f \in \BB_{\hol}^{\ast}(\Gamma_0(2)) \\ K \leq k_f \leq 2K \\ k_f \equiv 2 \hspace{-.25cm} \pmod{4}}} \frac{L(\tfrac 12, f)}{L(1,\sym^2 f)} \cdot X^{1 + \eta} K^{1/2 + \eta}.
\end{aligned}
\end{align}
Using the first moment estimate \eqref{firstmoment} and 
\[
\sum_{\substack{f \in \BB_{\hol}^{\ast}(\Gamma_0(2)) \\ K \leq k_f \leq 2K \\ k_f \equiv 2 \hspace{-.25cm} \pmod{4}}} L(\tfrac 12, f) \ll K^2,
\]
we conclude that for any $\eta > 0$ sufficiently small, \eqref{eqfour} is $\ll X (\log X)^{-\varepsilon}$.
This shows that the second property \eqref{prop2} holds for almost all squarefree integers $n \equiv 3 \pmod{8}$ with $n \leq X$. 
\end{proof}

\begin{proof}[Proof of {\hyperref[Linnikthm1]{Theorems \ref*{Linnikthm1}}} and {\ref{Linnikthm2}}]
The proof of \hyperref[Linnikthm2]{Theorem \ref*{Linnikthm2}} follows by the same method as the proof of \hyperref[Youngthm]{Theorem \ref*{Youngthm}} except that instead of a ball $B_R(w)$ at a point $w \in S^2$, we take the annulus $A_{r,R}(w)$ at $w = (0,0,1) \in S^2$ with $r = \arccos(n^{-\delta})$ and $R = \arccos(-n^{-\delta})$; the only change in the proof is that \eqref{hrRmupperboundeq} is used to bound $\tilde{h}_{r + \rho,R - \rho}(m)$ in place of \eqref{hrRupperboundseq} to bound $\tilde{h}_{0,R - \rho}(m)$. \hyperref[Linnikthm1]{Theorem \ref*{Linnikthm1}} is a direct consequence of \hyperref[variancethm1]{Theorem \ref*{variancethm1}} with $r = \arccos(n^{-1/2} \psi(n))$ and $R = \arccos(-n^{-1/2} \psi(n))$ together with an application of Chebyshev's inequality.
\end{proof}

\section{Automorphic preliminaries}

\subsection{The Kuznetsov formula}

The proof of \hyperref[momentprop]{Proposition \ref*{momentprop} (1)} makes use of the opposite sign Kuznetsov formula.

\begin{theorem}[{\cite[Theorem 16.3]{IK04}}]
\label{Kuznetsovthm}
Let $\delta > 0$, and let $h$ be a function that is even, holomorphic in the horizontal strip $|\Im(t)| \leq 1/2 + \delta$, and satisfies $h(t) \ll (|t| + 1)^{-2 - \delta}$. Then for $m,n \in \N$,
\begin{multline*}
\sum_{f \in \BB_0(\Gamma)} \epsilon_f \frac{\lambda_f(m) \lambda_f(n)}{L(1,\sym^2 f)} h(t_f) + \frac{1}{2\pi} \int_{-\infty}^{\infty} \frac{\lambda(m,t) \lambda(n,t)}{\zeta(1 + 2it) \zeta(1 - 2it)} h(t) \, dt	\\
= \sum_{c = 1}^{\infty} \frac{S(m,-n;c)}{c} \left(\Ks^{-} h\right)\left(\frac{\sqrt{mn}}{c}\right),
\end{multline*}
where $\epsilon_f \in \{1,-1\}$ denotes the root number of the Hecke--Maa\ss{} cusp form $f \in \BB_0(\Gamma)$ and $\lambda_f(n)$ denotes its $n$-th Hecke eigenvalue, $\lambda(n,t) \coloneqq \sum_{ab = n} a^{it} b^{-it}$ denotes the $n$-th Hecke eigenvalue of $E(z,1/2 + it)$,
\begin{gather*}
S(m,n;c) \coloneqq \sum_{d \in (\Z/c\Z)^{\times}} e\left(\frac{md + n\overline{d}}{c}\right), \qquad (\Ks^{-} h)(x) \coloneqq \int_{-\infty}^{\infty} \JJ_t^{-}(x) h(t) \, d_{\spec}t,	\\
\JJ_t^{-}(x) \coloneqq 4 \cosh \pi t K_{2it}(4\pi x), \qquad d_{\spec}t \coloneqq \frac{1}{2\pi^2} t \tanh \pi t \, dt.
\end{gather*}
\end{theorem}

The opposite sign Kuznetsov formula includes the root number $\epsilon_f$ in the spectral sum,; in our applications, this will eventually be counteracted by the fact that $L(\tfrac{1}{2},f) = 0$ whenever $\epsilon_f = -1$. This root number trick is well-known; in particular, this is exploited in \cite{BLM19,DK20,HK20}.

\subsection{The Petersson formula}

Similarly, the proof of \hyperref[momentprop]{Proposition \ref*{momentprop} (2)} makes use of an explicit form of the Petersson formula for squarefree level associated to the $(\infty,1)$-pair of cusps. This naturally introduces the Atkin--Lehner eigenvalue $\eta_f(N_1)$ into the expression of the sum over holomorphic cusp forms.

\begin{theorem}[{\cite[Theorem 9.6]{Iwa02}, \cite[Lemma A.9]{HK20}}]
\label{Peterssoninfty1thm}
Let $\delta > 0$, and let $h^{\hol} : 2\N \to \C$ be a sequence satisfying $h^{\hol}(k) \ll (k + 1)^{-2 - \delta}$. Let $N > 1$ be squarefree. Then for $m,n \in \N$,
\begin{multline*}
\sum_{N_1 N_2 = N} \frac{N_2^{3/2}}{\nu(N_2)} \sum_{f \in \BB_{\hol}^{\ast}(\Gamma_0(N_1))} \eta_f(N_1) \frac{h^{\hol}(k_f)}{L(1,\sym^2 f)} \sum_{\substack{\ell \mid N_2 \\ \frac{N_2}{\ell} \mid n}} L_{\ell}(1,\sym^2 f) \frac{\varphi(\ell)}{\ell^3}	\\
\times \sum_{v_1 \mid (\ell,m)} \nu(v_1) \mu\left(\frac{\ell}{v_1}\right) \lambda_f\left(\frac{\ell}{v_1}\right) \lambda_f\left(\frac{m}{v_1}\right) \sum_{v_2 \mid (\ell,n)} \nu\left(\frac{\ell}{v_2}\right) v_2 \mu(v_2) \lambda_f(v_2) \lambda_f\left(\frac{\ell n}{v_2 N_2}\right)	\\
= \sqrt{N} \sum_{\substack{c = 1 \\ (c,N) = 1}}^{\infty} \frac{S(m,n\overline{N};c)}{c} \left(\Ks^{\hol} h^{\hol}\right)\left(\frac{\sqrt{mn}}{c\sqrt{N}}\right),
\end{multline*}
where $N \overline{N} \equiv 1 \pmod{c}$ and
\begin{align*}
\left(\Ks^{\hol} h^{\hol}\right)(x) & \coloneqq \sum_{\substack{k = 2 \\ k \equiv 0 \hspace{-.25cm} \pmod{2}}}^{\infty} \frac{k - 1}{2 \pi^2} \JJ_k^{\hol}(x) h^{\hol}(k),	\\
\JJ_k^{\hol}(x) & \coloneqq 2\pi i^{-k} J_{k - 1}(4\pi x).
\end{align*}
\end{theorem}

We have written $L_p(s,\pi)$ to denote the $p$-component of the Euler product of an $L$-function $L(s,\pi)$, while $L_q(s,\pi) \coloneqq \prod_{p \mid q} L_p(s,\pi)$ and $L^q(s,\pi) \coloneqq L(s,\pi)/L_q(s,\pi)$.

\subsection{Mellin transforms}

We recall that the Mellin transform $\widehat{W}$ of a function $W : (0,\infty) \to \C$ is given by
\[\widehat{W}(s) \coloneqq \int_{0}^{\infty} W(x) x^s \, \frac{dx}{x}\]
for $s \in \C$ for which this converges absolutely, while the inverse Mellin transform $\widecheck{\WW}$ of a holomorphic function $\WW : \{z \in \C : a < \Re(s) < b\}$ is given by
\[\widecheck{\WW}(x) \coloneqq \frac{1}{2\pi i} \int_{\sigma - i\infty}^{\sigma + i\infty} \WW(s) x^{-s} \, ds\]
for $a < \sigma < b$ and $x \in (0,\infty)$ for which this converges absolutely.

Define
\[\JJ_0^+(s) \coloneqq -2\pi Y_0(4\pi x).\]
By \cite[12.43.17 and 12.43.18]{GR15}, we have that
\begin{align*}
\widehat{\JJ_0^{+}}(s) & = (2\pi)^{-s} \Gamma\left(\frac{s}{2}\right)^2 \cos \frac{\pi s}{2},	\\
\widehat{\JJ_t^{-}}(s) & = (2\pi)^{-s} \Gamma\left(\frac{s}{2} + it\right) \Gamma\left(\frac{s}{2} - it\right) \cosh \pi t,	\\
\widehat{\JJ_k^{\hol}}(s) & = (2\pi)^{-s} \Gamma\left(\frac{s + k - 1}{2}\right) \Gamma\left(\frac{s - k + 1}{2}\right) \times \begin{dcases*}
\cos \frac{\pi s}{2} & for $k \equiv 0 \pmod{2}$,	\\
\sin \frac{\pi s}{2} & for $k \equiv 1 \pmod{2}$.
\end{dcases*}
\end{align*}
We note that $\widehat{\JJ_k^{\hol}}(s)$ has simple zeroes at $s = k - 1 + 2\ell$ and simple poles at $s = 3 - k - 2\ell$ for $\ell \in \N$.

We also require bounds for the Mellin transform of the function $\Ks^{-} h$ appearing in the opposite sign Kuznetsov formula, \hyperref[Kuznetsovthm]{Theorem \ref*{Kuznetsovthm}}. That we can achieve quite strong bounds proves quite advantageous and is the main reason that we use the opposite sign Kuznetsov formula instead of the same sign Kuznetsov formula, where such strong bounds are unattainable (cf.~\cite[(2.16)]{Mot03}).

\begin{lemma}[{\cite[Lemma 2]{BLM19}}]
\label{MellinKsextendlemma}
Suppose that $h(t)$ is an even holomorphic function in the strip $-2M < \Im(t) < 2M$ for some $M \geq 20$ with zeroes at $\pm (n - 1/2)i$ for $n \in \{1,\ldots,2M\}$ and satisfies $h(t) \ll (|t| + 1)^{-2M}$ in this region. Then the Mellin transform of $\Ks^{-} h$ extends to a holomorphic function in the strip $1 - M < \Re(s) < M - 1$, in which it satisfies
\begin{equation}
\label{KsMellinboundeq}
\widehat{\Ks^{-} h}(s) \ll (|\Im(s)| + 1)^{1 - M}.
\end{equation}
\end{lemma}

\subsection{The Vorono\u{\i} summation formula}

For $c \in \N$, $d \in (\Z/c\Z)^{\times}$, and primitive Dirichlet characters $\chi_1$ and $\chi_2$ modulo $q_1$ and $q_2$ respectively, we require the Vorono\u{\i} summation formula for the Vorono\u{\i} $L$-series
\[L\left(s,E_{\chi_1,\chi_2},\frac{d}{c}\right) \coloneqq \sum_{m = 1}^{\infty} \frac{\lambda_{\chi_1,\chi_2}(m,0) e\left(\frac{md}{c}\right)}{m^s},\]
associated to the Eisenstein series $E_{\chi_1,\chi_2}(z)$ with Hecke eigenvalues $\lambda_{\chi_1,\chi_2}(m,0)$, where
\[\lambda_{\chi_1,\chi_2}(m,t) \coloneqq \sum_{ab = m} \chi_1(a) a^{it} \chi_2(b) b^{-it}.\]

\begin{lemma}[{\cite[Appendix A]{KMV02}, \cite[Section 2.4]{HM06}, \cite[Theorem A]{LT05}}]
\label{Voronoilemma}
Suppose that $q = q_1 q_2$ is squarefree and $\chi = \chi_1 \chi_2$ is a primitive Dirichlet character modulo $q$ satisfying $\chi(-1) = (-1)^{\kappa}$ for $\kappa \in \{0,1\}$. Then for $c \in \N$ with $(c,q) = q_1$, the Vorono\u{\i} $L$-series $L(s,E_{\chi,1},d/c)$ is absolutely convergent for $\Re(s) > 1$ and extends to a meromorphic function of $s \in \C$. The only possible pole is at $s = 1$, with
\begin{equation}
\label{Voronoireseq}
\lim_{s \to 1} (s - 1) L\left(s,E_{\chi,1},\frac{d}{c}\right) = \begin{dcases*}
\frac{\tau(\chi) \overline{\chi}(d) L(1, \overline{\chi})}{c} & if $c \equiv 0 \pmod{q}$,	\\
\frac{\chi(c) L(1,\chi)}{c} & if $(c,q) = 1$,	\\
0 & otherwise.
\end{dcases*}
\end{equation}
Moreover, we have the functional equation
\begin{equation}
\label{Voronoifunc0eq}
L\left(s,E_{\chi,1},\frac{d}{c}\right) = \frac{2 \overline{\chi_1}(d) \tau(\chi_2)}{q_2^s c^{2s - 1}} \sum_{\pm} \chi_2(\mp c) \widehat{\JJ_0^{\pm}}(2(1 - s)) L\left(1 - s, E_{\chi_1,\overline{\chi_2}}, \mp \frac{\overline{q_2 d}}{c}\right)
\end{equation}
for $\kappa = 0$ and
\begin{equation}
\label{Voronoifunc1eq}
L\left(s,E_{\chi,1},\frac{d}{c}\right) = -\frac{2 \overline{\chi_1}(d) \tau(\chi_2)}{q_2^s c^{2s - 1}} \chi_2(-c) \widehat{\JJ_1^{\hol}}(2(1 - s)) L\left(1 - s, E_{\chi_1,\overline{\chi_2}}, - \frac{\overline{q_2 d}}{c}\right)
\end{equation}
for $\kappa = 1$, where $q_2 \overline{q_2} \equiv d \overline{d} \equiv 1 \pmod{c}$.
\end{lemma}

We need some control over the size of the Vorono\u{\i} $L$-series $L(s,E_{\chi,1},d/c)$ in vertical strips.

\begin{lemma}
For $s = \sigma + i\tau$ and for fixed $M \in \N$, we have that
\begin{equation}
\label{Phrageq}
\left(\frac{s - 1}{s + M}\right) L\left(s,E_{\chi,1},\frac{d}{c}\right) \ll_{q,\sigma,M,\e} \begin{dcases*}
(c(|\tau| + 1))^{\e} & for $\sigma \geq 1$,	\\
(c(|\tau| + 1))^{1 - \sigma + \e} & for $0 \leq \sigma \leq 1$,	\\
(c(|\tau| + 1))^{1 - 2\sigma + \e} & for $-M < \sigma \leq 0$.
\end{dcases*}
\end{equation}
\end{lemma}

\begin{proof}
Stirling's formula implies that for $\sigma > 0$,
\begin{equation}
\label{JJ0boundseq}
\widehat{\JJ_0^+}(s) \ll (|\tau| + 1)^{\sigma - 1}, \qquad \widehat{\JJ_0^-}(s) \ll (|\tau| + 1)^{\sigma - 1} e^{-\frac{\pi}{2} |\tau|}, \qquad \widehat{\JJ_1^{\hol}}(s) \ll (|\tau| + 1)^{\sigma - 1};
\end{equation}
see \cite[Corollary A.27]{HK20}. The bounds \eqref{Phrageq} then follow from this, the functional equations \eqref{Voronoifunc0eq} and \eqref{Voronoifunc1eq}, and the Phragm\'{e}n--Lindel\"{o}f convexity principle.
\end{proof}

We also require the following identity for Gauss sums.

\begin{lemma}[{\cite[Lemma 3.1.3]{Miy06}}]\label{Miyakelemma}
Let $\chi$ be a primitive Dirichlet character modulo $q$ and $c \equiv 0 \pmod{q}$. We have that
\[\sum_{a \in (\Z/c\Z)^{\times}} \chi(a) e\left(\frac{ma}{c}\right) = \tau(\chi) \sum_{d \mid \left(\frac{c}{q}, m\right)} d \mu\left(\frac{c}{qd}\right) \chi\left(\frac{c}{qd}\right) \overline{\chi}\left(\frac{m}{d}\right).\]
\end{lemma}

\subsection{A multiple Dirichlet series}

In the course of the proof of \hyperref[momentprop]{Proposition \ref*{momentprop}}, we shall come across the multiple Dirichlet series
\begin{align}
\label{Dswchidefeq}
\DD_{N,\chi}^{\pm}(s,w) & \coloneqq \sum_{\substack{c = 1 \\ c \equiv 0 \hspace{-.25cm} \pmod{N}}}^{\infty} \sum_{m = 1}^{\infty} \frac{\lambda_{\chi,1}(m,0)}{m^{\frac{s}{2} + w}} \frac{S(m,\pm 1;c)}{c^{1 - s}},	\\
\label{Doverlineswchidefeq}
\DD_{\overline{N},\chi}^{\pm}(s,w) & \coloneqq \sum_{\substack{c = 1 \\ (c,N) = 1}}^{\infty} \sum_{m = 1}^{\infty} \frac{\lambda_{\chi,1}(m,0)}{m^{\frac{s}{2} + w}} \frac{S(m,\pm \overline{N};c)}{c^{1 - s}},
\end{align}
where $N \overline{N} \equiv 1 \pmod{c}$. Via the Weil bound for Kloosterman sums, these are absolutely convergent in the region
\[\Omega_0 \coloneqq \{(s,w) \in \C^2 : 2 - 2\Re(w) < \Re(s) < -1/2\},\]
in which they are holomorphic in the complex variables $s$ and $w$. We study the meromorphic continuation of these multiple Dirichlet series.

\begin{lemma}
Suppose that $q > 1$ and $N$ are squarefree and coprime and that $\chi$ is a primitive character modulo $q$ satisfying $\chi(-1) = (-1)^{\kappa}$ for $\kappa \in \{0,1\}$. As a function of the complex variables $(s,w) \in \C^2$, the functions $\DD_{N,\chi}^{\pm}(s,w)$ and $\DD_{\overline{N},\chi}^{\pm}(s,w)$ extend holomorphically in an open neighbourhood of the union of the regions
\begin{align*}
\Omega_1 & \coloneqq \{(s,w) \in \C^2 : -2\Re(w) \leq \Re(s) < \min\{2 - 2\Re(w), 2\Re(w) - 4\}\},	\\
\Omega_2 & \coloneqq \{(s,w) \in \C^2 : \Re(s) < -2\Re(w), \ \Re(w) \geq 1/2\}.
\end{align*}
For $\Re(w) \geq 1/2$, these functions satisfy
\begin{align}
\label{Dswchiresidueeq}
\lim_{\frac{s}{2} + w \to 1} \left(\frac{s}{2} + w - 1\right) \DD_{N,\chi}^{\pm}(s,w) & = \frac{\mu(N) \chi(N)}{N^{2w} L^N(2w,\chi)} \left(\frac{\chi(\pm 1) \tau(\chi)^2 L(1,\overline{\chi})}{q^{2w}} + L(1,\chi)\right),	\\
\label{Doverlineswchiresidueeq}
\lim_{\frac{s}{2} + w \to 1} \left(\frac{s}{2} + w - 1\right) \DD_{\overline{N},\chi}^{\pm}(s,w) & = \frac{1}{L^N(2w,\chi)} \left(\frac{\chi(\pm N) \tau(\chi)^2 L(1,\overline{\chi})}{q^{2w}} + L(1,\chi)\right).
\end{align}
In an open neighbourhood of $\Omega_2$ and for $\kappa = 0$, $\DD_{N,\chi}^{\pm}(s,w)$ is equal to
\begin{multline}
\label{DswchipostVoronoieq}
\frac{2 \mu(N) \chi(N) \tau(\chi)}{q^{2w} N^{2w} L^N(2w,\chi)} \sum_{\pm_1} \chi(\mp_1 1) \widehat{\JJ_0^{\pm_1}}(2 - s - 2w) \sum_{N_1 N_2 = N} \mu(N_2) N_2 \sum_{q_1 q_2 = q} \overline{\chi_1}(N_2) q_2^{w - \frac{s}{2}}	\\
\times \sum_{\substack{m = 1 \\ m \equiv \pm \pm_1 q_2 \hspace{-.25cm} \pmod{N_2} \\ m \neq \pm \pm_1 q_2}}^{\infty} \frac{\lambda_{\chi_1,\overline{\chi_2}}(m,0)}{m^{1 - \frac{s}{2} - w}} \sum_{d \mid \frac{m \pm \mp_1 q_2}{N_2}} d^{1 - 2w} \overline{\chi_1}\left(\frac{m \pm \mp_1 q_2}{dN_2}\right) \chi_2(d),
\end{multline}
which is absolutely convergent, and $\DD_{\overline{N},\chi}^{\pm}(s,w)$ is equal to
\begin{multline}
\label{DoverlineswchipostVoronoieq}
\frac{2 \tau(\chi)}{q^{2w} L^N(2w,\chi)} \sum_{\pm_1} \chi(\mp_1 1) \widehat{\JJ_0^{\pm_1}}(2 - s - 2w) \sum_{q_1 q_2 = q} \chi_1(N) q_2^{w - \frac{s}{2}}	\\
\times \sum_{\substack{m = 1 \\ m \neq \pm \pm_1 \frac{q_2}{N}}}^{\infty} \frac{\lambda_{\chi_1,\overline{\chi_2}}(m,0)}{m^{1 - \frac{s}{2} - w}} \sum_{d \mid (Nm \pm \mp_1 q_2)} d^{1 - 2w} \overline{\chi_1}\left(\frac{Nm \pm \mp_1 q_2}{d}\right) \chi_2(d),
\end{multline}
while the same holds for $\kappa = 1$ with $\widehat{\JJ_0^+}$ and $\widehat{\JJ_0^-}$ replaced by $-\widehat{\JJ_1^{\hol}}$ and $0$ respectively.
\end{lemma}

\begin{proof}
We prove this only for $\DD_{N,\chi}^{\pm}(s,w)$; the proof for $\DD_{\overline{N},\chi}^{\pm}(s,w)$ follows by the same argument. In the region $\{(s,w) \in \C^2 : 2 - 2\Re(w) < \Re(s) < -1\}$, we may open up the Kloosterman sum to see that
\begin{equation}
\label{DswchiVoronoieq}
\DD_{N,\chi}^{\pm}(s,w) = \sum_{q_1 q_2 = q} \sum_{\substack{c = 1 \\ c \equiv 0 \hspace{-.25cm} \pmod{Nq_1} \\ \left(\frac{c}{q_1},q_2\right) = 1}}^{\infty} \frac{1}{c^{1 - s}} \sum_{d \in (\Z/c\Z)^{\times}} e\left(\pm \frac{\overline{d}}{c}\right) L\left(\frac{s}{2} + w, E_{\chi,1},\frac{d}{c}\right).
\end{equation}
From \eqref{Voronoireseq},
\begin{multline*}
\lim_{\frac{s}{2} + w \to 1} \left(\frac{s}{2} + w - 1\right) \DD_{N,\chi}^{\pm}(s,w) = \tau(\chi) L(1,\overline{\chi}) \sum_{\substack{c = 1 \\ c \equiv 0 \hspace{-.25cm} \pmod{Nq}}}^{\infty} \frac{1}{c^{2w}} \sum_{d \in (\Z/c\Z)^{\times}} \overline{\chi}(d) e\left(\pm \frac{\overline{d}}{c}\right)	\\
+ L(1,\chi) \sum_{\substack{c = 1 \\ c \equiv 0 \hspace{-.25cm} \pmod{N}}}^{\infty} \frac{\chi(c)}{c^{2w}} \sum_{d \in (\Z/c\Z)^{\times}} e\left(\pm \frac{\overline{d}}{c}\right).
\end{multline*}
We obtain \eqref{Dswchiresidueeq} after making the change of variables $d \mapsto \pm \overline{d}$, applying \hyperref[Miyakelemma]{Lemma \ref*{Miyakelemma}}, making the change of variables $c \mapsto cNq$ and $c \mapsto cN$ respectively, and noting that $\mu(cN) = \mu(c) \mu(N) 1_{(c,N) = 1}$. Next, the bounds \eqref{Phrageq} imply that the expression \eqref{DswchiVoronoieq} for $\DD_{N,\chi}^{\pm}(s,w)$ is absolutely convergent and holomorphic in an open neighbourhood of $\Omega_1$ and in the region $\{(s,w) \in \C^2 : \Re(s) < -2\Re(w), \ \Re(w) > 1\}$.

In this latter region, we may use the functional equations \eqref{Voronoifunc0eq} and \eqref{Voronoifunc1eq} to see that
\begin{multline*}
\DD_{N,\chi}^{\pm}(s,w) = 2 \sum_{\pm_1} \widehat{\JJ_0^{\pm_1}}(2 - s - 2w) \sum_{q_1 q_2 = q} \frac{\chi_2(\mp_1 1) \tau(\chi_2)}{q_2^{\frac{s}{2} + w}} \sum_{m = 1}^{\infty} \frac{\lambda_{\chi_1,\overline{\chi_2}}(m,0)}{m^{1 - \frac{s}{2} - w}}	\\
\times \sum_{\substack{c = 1 \\ c \equiv 0 \hspace{-.25cm} \pmod{Nq_1} \\ \left(\frac{c}{q_1},q_2\right) = 1}}^{\infty} \frac{\chi_2(c)}{c^{2w}} \sum_{d \in (\Z/c\Z)^{\times}} \overline{\chi_1}(d) e\left(\frac{(\mp_1 \overline{q_2}m \pm 1) \overline{d}}{c}\right)
\end{multline*}
for $\kappa = 0$, while the same holds for $\kappa = 1$ with $\widehat{\JJ_0^+}$ and $\widehat{\JJ_0^-}$ replaced by $-\widehat{\JJ_1^{\hol}}$ and $0$ respectively. By making the change of variables $d \mapsto \mp_1 \overline{q_2 d}$, the sum over $d$ is equal to
\[\chi_1(\mp_1 q_2) \tau(\chi_1) \sum_{d \mid \left(\frac{c}{q_1},m \pm \mp_1 q_2\right)} d \mu\left(\frac{c}{q_1 d}\right) \chi_1\left(\frac{c}{q_1 d}\right) \overline{\chi_1}\left(\frac{m \pm \mp_1 q_2}{d}\right)\]
by \hyperref[Miyakelemma]{Lemma \ref*{Miyakelemma}}. We introduce a sum over $N_1 N_2 = N$ such that $d \equiv 0 \pmod{N_1}$ and $(d/N_1,N_2) = 1$, then make the change of variables $c \mapsto cdN_2 q_1$. The ensuing sum over $c$ is $1/L^{N_2}(2w,\chi)$, so that this is
\begin{multline*}
\frac{2 \tau(\chi)}{q^{2w} L^N(2w,\chi)} \sum_{\pm_1} \chi(\mp_1 1) \widehat{\JJ_0^{\pm_1}}(2 - s - 2w) \sum_{N_1 N_2 = N} \frac{\mu(N_2) \chi(N_2)}{N_2^{2w} L_{N_1}(2w,\chi)} \sum_{q_1 q_2 = q} q_2^{w - \frac{s}{2}}	\\
\times \sum_{\substack{m = 1 \\ m \neq \pm \pm_1 q_2}}^{\infty} \frac{\lambda_{\chi_1,\overline{\chi_2}}(m,0)}{m^{1 - \frac{s}{2} - w}} \sum_{\substack{d \mid (m \pm \mp_1 q_2) \\ d \equiv 0 \hspace{-.25cm} \pmod{N_1} \\ \left(\frac{d}{N_1},N_2\right) = 1}} d^{1 - 2w} \overline{\chi_1}\left(\frac{m \pm \mp_1 q_2}{d}\right) \chi_2(d).
\end{multline*}
Since $N$ is squarefree, $(d/N_1,N_2) = 1$ if and only if $(d,N_2) = 1$. We replace this condition by the sum $\sum_{N_3 \mid (N_2,d)} \mu(N_3)$, so that $\frac{N_2}{N_3} N_5 = N$ and $N_1 N_3 = N_5$, then make the change of variables $d \mapsto dN_5$. The resulting sum over $N_1 N_3 = N_5$ is simply $1$, so after relabelling, we arrive at \eqref{DswchipostVoronoieq}. This is absolutely convergent and holomorphic in an open neighbourhood of $\Omega_2$.
\end{proof}

\section{Identities for moments of \texorpdfstring{$L$}{L}-functions}
\label{momentsect}

We wish to prove bounds and asymptotics for the moments
\begin{gather}
\label{moment1eq}
\sum_{f \in \BB_0(\Gamma)} \frac{L\left(\frac{1}{2},f\right) L\left(\frac{1}{2},f \otimes \chi_D\right)}{L(1,\sym^2 f)} h(t_f) + \frac{1}{2\pi} \int_{-\infty}^{\infty} \left|\frac{\zeta\left(\frac{1}{2} + it\right) L\left(\frac{1}{2} + it,\chi_D\right)}{\zeta(1 + 2it)}\right|^2 h(t) \, dt,	\\
\label{moment2eq}
\sum_{f \in \BB_{\hol}^{\ast}(\Gamma_0(2))} \frac{L\left(\frac{1}{2},f\right) L\left(\frac{1}{2},f \otimes \chi_D\right)}{L(1,\sym^2 f)} h^{\hol}(k_f),
\end{gather}
where $\chi_D$ is the primitive quadratic character modulo $|D|$ with $D$ a squarefree fundamental discriminant and $h : \R \to \C$ and $h^{\hol} : 2\N \to \C$ are appropriately chosen test functions. One approach would be to use the Kuznetsov and Petersson formul\ae{} in conjunction with approximate functional equations for $L$-functions and then apply the Vorono\u{\i} summation formula, as is done in similar situations in \cite{HT14} and \cite[Section 6]{HK20}. This would require some care, since the latter moment involves a sum over \emph{newforms}, so one must use the Petersson formula for newforms; see \cite[Lemma 5]{HT14} and \cite[Theorem 3]{PY19}. Another approach would be to proceed directly via the relative trace formula, as in \cite{FW09,MR12,RR05}.

We instead proceed via a combination of the Kuznetsov and Petersson formul\ae{}, the Vorono\u{\i} summation formula, and analytic continuation, as is done in similar situations in \cite{Byk98,BF17,GZ99,Nel10} (and, after this paper was written, in \cite{HLN20}). This has the advantage of giving \emph{exact} identities for moments of $L$-functions: \eqref{moment1eq} and \eqref{moment2eq} are each shown to be equal to the sum of a main term and a shifted convolution sum. We avoid the use of the Petersson formula for newforms by using the Petersson formula for the $(\infty,1)$-pair of cusps, which has both the effect of inserting an Atkin--Lehner eigenvalue, which is ultimately harmless, and removing the contribution of the oldforms during the process of analytic continuation.

\subsection{An identity for a moment of \texorpdfstring{$L$}{L}-functions associated to Maa\ss{} forms}

\begin{lemma}
\label{Maassmomentwlemma}
Suppose that $q > 1$ is squarefree and $\chi$ is a primitive Dirichlet character modulo $q$ satisfying $\chi(-1) = (-1)^{\kappa}$ for $\kappa \in \{0,1\}$. Let $h(t)$ be an even holomorphic function in the strip $-2M < \Im(t) < 2M$ for some $M \geq 20$ with zeroes at $\pm (n - 1/2)i$ for $n \in \{1,\ldots,2M\}$ and satisfies $h(t) \ll (|t| + 1)^{-2M}$ in this region. Then for $5/4 < \Re(w) < (M - 1)/2$,
\begin{multline}
\label{momentweq}
\sum_{f \in \BB_0(\Gamma)} \epsilon_f \frac{L(w,f) L(w,f \otimes \chi)}{L(2w,\chi) L(1,\sym^2 f)} h(t_f)	\\
+ \frac{1}{2\pi} \int_{-\infty}^{\infty} \frac{\zeta(w + it) \zeta(w - it) L(w + it,\chi) L(w - it,\chi)}{L(2w,\chi) \zeta(1 + 2it) \zeta(1 - 2it)} h(t) \, dt
\end{multline}
is equal to the sum of
\begin{equation}
\label{residueweq}
\frac{2 \widehat{\Ks^{-} h}(2(1 - w))}{L(2w,\chi)} \left(\frac{\chi(-1) \tau(\chi)^2 L(1,\overline{\chi})}{q^{2w}} + L(1,\chi)\right)
\end{equation}
and of
\begin{multline}
\label{shiftedweq}
\frac{2 \tau(\chi)}{q^{2w} L(2w,\chi)} \sum_{\pm} \chi(\mp 1) \sum_{q_1 q_2 = q} q_2^{w - \frac{1}{2}} \sum_{\substack{m = 1 \\ m \neq \mp q_2}}^{\infty} \frac{\lambda_{\chi_1,\overline{\chi_2}}(m,0)}{m^{\frac{1}{2} - w}} \sum_{d \mid (m \pm q_2)} d^{1 - 2w} \overline{\chi_1}\left(\frac{m \pm q_2}{d}\right) \chi_2(d)	\\
\times \frac{1}{2\pi i} \int_{\sigma_1 - i\infty}^{\sigma_1 + i\infty} \widehat{\Ks^{-} h}(s) \widehat{\JJ_0^{\pm}}(2 - s - 2w) \left(\frac{m}{q_2}\right)^{\frac{s - 1}{2}} \, ds
\end{multline}
for $\kappa = 0$, where $1 - M < \sigma_1 < -2\Re(w)$; the same holds for $\kappa = 1$ with $\widehat{\JJ_0^{+}}$ and $\widehat{\JJ_{0}^{-}}$ replaced by $-\widehat{\JJ_1^{\hol}}$ and $0$ respectively.
\end{lemma}

\begin{proof}
We set $n = 1$ in the opposite sign Kuznetsov formula, \hyperref[Kuznetsovthm]{Theorem \ref*{Kuznetsovthm}}, multiply through by $\lambda_{\chi,1}(m,0) m^{-w}$ with $5/4 < \Re(w) < (M - 1)/2$, and sum over $m \in \N$. The Maa\ss{} cusp form and Eisenstein contributions are equal to \eqref{momentweq} by the Ramanujan identities
\begin{align*}
\sum_{m = 1}^{\infty} \frac{\lambda_f(m) \lambda_{\chi,1}(m,0)}{m^w} & = \frac{L(w,f) L(w,f \otimes \chi)}{L(2w,\chi)},	\\
\sum_{m = 1}^{\infty} \frac{\lambda(m,t) \lambda_{\chi,1}(m,0)}{m^w} & = \frac{\zeta(w + it) \zeta(w - it) L(w + it, \chi) L(w - it, \chi)}{L(2w,\chi)}.
\end{align*}
For the Kloosterman term, we use Mellin inversion and \hyperref[MellinKsextendlemma]{Lemma \ref*{MellinKsextendlemma}} to write
\[(\Ks^{-} h)(x) = \frac{1}{2\pi i} \int_{\sigma_0 - i\infty}^{\sigma_0 + i\infty} \widehat{\Ks^{-} h}(s) x^{-s} \, ds\]
for $1 - M < \sigma_0 < M - 1$, so that the Kloosterman term is
\[\frac{1}{2\pi i} \int_{\sigma_0 - i\infty}^{\sigma_0 + i\infty} \widehat{\Ks^{-} h}(s) \DD_{1,\chi}^{-}(s,w) \, ds = \sum_{c = 1}^{\infty} \frac{1}{2\pi i} \int_{\sigma_0 - i\infty}^{\sigma_0 + i\infty} \widehat{\Ks^{-} h}(s) \sum_{m = 1}^{\infty} \frac{\lambda_{\chi,1}(m,0)}{m^{\frac{s}{2} + w}} \frac{S(m,-1;c)}{c^{1 - s}} \, ds\]
for $2 - 2\Re(w) < \sigma_0 < -1/2$. The condition on $\sigma_0$ ensures the absolute convergence of this via the Weil bound, which allows us to interchange the order of integration and summation.

We observe that
\[\sum_{m = 1}^{\infty} \frac{\lambda_{\chi,1}(m,0)}{m^{\frac{s}{2} + w}} \frac{S(m,-1;c)}{c^{1 - s}} = \frac{1}{c^{1 - s}} \sum_{d \in (\Z/c\Z)^{\times}} e\left(-\frac{\overline{d}}{c}\right) L\left(\frac{s}{2} + w, E_{\chi,1}, \frac{d}{c}\right)\]
for $\Re(s) > 2 - 2\Re(w)$. We shift the contour to $\Re(s) = \sigma_1$ with $1 - M < \sigma_1 < -2\Re(w)$; the bounds \eqref{Phrageq} and \eqref{KsMellinboundeq} ensure that the ensuing integral converges absolutely. From \eqref{Dswchiresidueeq}, the integrand has a pole at $s = 2(1 - w)$; the ensuing sum over $c \in \N$ of this residue is \eqref{residueweq}. The contour integral is equal to \eqref{shiftedweq} by \eqref{DswchipostVoronoieq}, noting that the absolute convergence of the sum over $c \in \N$ and the integral over $\Re(s) = -\sigma_1$ is guaranteed via \eqref{Phrageq} and \eqref{KsMellinboundeq}, which allows us to interchange the order of integration and summation.
\end{proof}

Now we specialise \hyperref[Maassmomentwlemma]{Lemma \ref*{Maassmomentwlemma}} to $q = |D|$ and $\chi = \chi_D$.

\begin{corollary}
\label{momentidentitycor}
Let $D$ be a squarefree fundamental discriminant and let $\chi_D$ be the primitive quadratic character modulo $|D|$, so that $\chi_D(-1) = \sgn(D)$. Let $h(t)$ be an even holomorphic function in the strip $-2M < \Im(t) < 2M$ for some $M \geq 20$ with zeroes at $\pm (n - 1/2)i$ for $n \in \{1,\ldots,2M\}$ and satisfies $h(t) \ll (|t| + 1)^{-2M}$ in this region. Then the moment
\begin{equation}
\label{momentidentityeq}
\sum_{f \in \BB_0(\Gamma)} \frac{L\left(\frac{1}{2},f\right) L\left(\frac{1}{2},f \otimes \chi_D\right)}{L(1,\sym^2 f)} h(t_f) + \frac{1}{2\pi} \int_{-\infty}^{\infty} \left|\frac{\zeta\left(\frac{1}{2} + it\right) L\left(\frac{1}{2} + it,\chi_D\right)}{\zeta(1 + 2it)}\right|^2 h(t) \, dt
\end{equation}
is equal to the sum of the main term
\begin{equation}
\label{momentidentitymaineq}
2 L(1,\chi_D) \int_{-\infty}^{\infty} h(t) \, d_{\spec}t
\end{equation}
and the shifted convolution sum
\begin{multline}
\label{momentidentityshiftedeq}
\frac{2}{\sqrt{D}} \sum_{\pm} \sum_{D_1 D_2 = |D|} \sum_{\substack{m = 1 \\ m \neq \mp D_2}}^{\infty} \chi_1(\sgn(m \pm D_2)) \lambda_{\chi_1,\chi_2}(m,0) \lambda_{\chi_1,\chi_2}(|m \pm D_2|,0)	\\
\times \frac{1}{2\pi i} \int_{\sigma_1 - i\infty}^{\sigma_1 + i\infty} \widehat{\Ks^{-} h}(s) \widehat{\JJ_0^{\pm}}(1 - s) \left(\frac{m}{D_2}\right)^{\frac{s - 1}{2}} \, ds
\end{multline}
for $D > 0$, where $1 - M < \sigma_1 < -1$; the same holds for $D < 0$ with $\widehat{\JJ_0^+}$ and $\widehat{\JJ_0^-}$ replaced by $-\widehat{\JJ_1^{\hol}}$ and $0$ respectively.
\end{corollary}

\begin{proof}
We use \hyperref[Maassmomentwlemma]{Lemma \ref*{Maassmomentwlemma}} and holomorphically extend to $w = 1/2$. As a function of the complex variable $w$, \eqref{momentweq} extends holomorphically to $w = 1/2$; we may use the Cauchy--Schwarz inequality and the spectral large sieve to see the absolute convergence of the Maa\ss{} and Eisenstein contributions for $\Re(w) \geq 1/2$. Two additional terms arise from the Eisenstein contribution due to the poles of $\zeta(w \pm it)$ for $t = \mp i(1 - w)$, yet these terms vanish at $w = 1/2$ since $h(\pm i/2) = 0$. The Maa\ss{} contribution only includes terms from even Maa\ss{} forms at $w = 1/2$ due to the fact that $L(\tfrac{1}{2},f) = 0$ when $\epsilon_f = -1$, as the root number of $f$ is $\epsilon_f$. The holomorphic extension of \eqref{residueweq} is clear, observing that
\[\widehat{\Ks^{-} h}(1) = \int_{-\infty}^{\infty} \widehat{\JJ_t^{-}}(1) h(t) \, d_{\spec}t = \frac{1}{2} \int_{-\infty}^{\infty} h(t) \, d_{\spec}t,\]
as is the holomorphic extension of \eqref{shiftedweq}, noting additionally that $\tau(\chi_D) = \sqrt{D}$.
\end{proof}

\subsection{An identity for a moment of \texorpdfstring{$L$}{L}-functions associated to holomorphic forms}

\begin{lemma}
\label{holinfty1momentwlemma}
Suppose that $q > 1$ and $N > 1$ are squarefree and coprime and $\chi$ is an odd primitive character modulo $q$ satisfying $\chi(p) = -1$ for all $p \mid N$. Let $h^{\hol} : 2\N \to \C$ be a compactly supported function vanishing at $k = 2$. Then for $5/4 < \Re(w) < 3/2$,
\begin{equation}
\label{Peterssoninfty1momentweq}
\sum_{N_1 N_2 = N} \frac{\varphi(N_2) \mu(N_2)}{N_2^{3/2}} \prod_{p \mid N_2} \left(1 - p^{1 - 2w}\right) \sum_{f \in \BB_{\hol}^{\ast}(\Gamma_0(N_1))} \eta_f(N_1) \lambda_f(N_2) \frac{L(w,f) L(w,f \otimes \chi)}{L^N(2w,\chi) L^{N_2}(1,\sym^2 f)} h^{\hol}(k_f)
\end{equation}
is equal to the sum of
\begin{equation}
\label{Peterssoninfty1deltaweq}
\frac{2 \widehat{\Ks^{\hol} h^{\hol}}(2(1 - w)) N^{\frac{3}{2} - w}}{L^N(2w,\chi)} \left(\frac{\mu(N) \tau(\chi)^2 L(1,\overline{\chi})}{q^{2w}} + L(1,\chi)\right)
\end{equation}
and of
\begin{multline}
\label{Peterssoninfty1shiftedweq}
-\frac{2 \tau(\chi) N^{\frac{3}{2} - w}}{q^{2w} L^N(2w,\chi)} \sum_{q_1 q_2 = q} \chi_1(N) q_2^{w - \frac{1}{2}} \sum_{\substack{m = 1 \\ m \equiv 0 \hspace{-.25cm} \pmod{N}}}^{\infty} \frac{\lambda_{\chi_1,\overline{\chi_2}}\left(\frac{m}{N},0\right)}{m^{\frac{1}{2} - w}}	\\
\times \sum_{d \mid (m - q_2)} d^{1 - 2w} \overline{\chi_1}\left(\frac{m - q_2}{d}\right) \chi_2(d) \frac{1}{2\pi i} \int_{\sigma_1 - i\infty}^{\sigma_1 + i\infty} \widehat{\Ks^{\hol} h^{\hol}}(s) \widehat{\JJ_1^{\hol}}(2 - s - 2w) \left(\frac{m}{q_2}\right)^{\frac{s - 1}{2}} \, ds,
\end{multline}
where $-3 < \sigma_1 < -2\Re(w)$.
\end{lemma}

\begin{proof}
We set $n = 1$ in the Petersson formula associated to the $(\infty,1)$-pair of cusps, \hyperref[Peterssoninfty1thm]{Theorem \ref*{Peterssoninfty1thm}}, multiply by $\lambda_{\chi,1}(m,0) m^{-w}$ with $5/4 < \Re(w) < 3/2$, and sum over $m \in \N$. Upon making the change of variables $m \mapsto m v_1$, the resulting sum over $m$ occurring in the holomorphic cusp form contribution is
\[\sum_{\substack{m = 1 \\ (m,v_1) = 1}}^{\infty} \frac{\lambda_f(m) \lambda_{\chi,1}(m,0)}{m^{w}} \prod_{p \mid v_1} \sum_{j = 0}^{\infty} \frac{\lambda_f(p^j) \lambda_{\chi,1}(p^{j + 1},0)}{p^{jw}}.\]
We now use the Hecke relations: for $(mn,N_1) = 1$,
\begin{align*}
\lambda_{\chi,1}(mn,0) & = \sum_{a \mid (m,n)} \mu(a) \chi(a) \lambda_{\chi,1}\left(\frac{m}{a},0\right) \lambda_{\chi,1}\left(\frac{n}{a},0\right),	\\
\lambda_f(mn) & = \sum_{a \mid (m,n)} \mu(a) \lambda_f\left(\frac{m}{a}\right) \lambda_f\left(\frac{n}{a}\right).
\end{align*}
We take $m = p^j$ and $n = p$. Using the former identity, then the latter, we find that
\[\sum_{j = 0}^{\infty} \frac{\lambda_f(p^j) \lambda_{\chi,1}(p^{j + 1},0)}{p^{jw}} = \frac{\lambda_{\chi,1}(p,0) - \chi(p) \lambda_f(p) p^{-w}}{1 - \chi(p) p^{-2w}} \sum_{j = 0}^{\infty} \frac{\lambda_f(p^j) \lambda_{\chi,1}(p^j,0)}{p^{jw}}.\]
Using this, the Ramanujan identity
\[\sum_{m = 1}^{\infty} \frac{\lambda_f(m) \lambda_{\chi,1}(m,0)}{m^w} = \frac{L(w,f) L(w,f \otimes \chi)}{L^{N_1}(2w,\chi)}\]
and recalling the assumption that $\chi(p) = -1$ for all $p \mid N$, the holomorphic cusp form contribution is simplified to \eqref{Peterssoninfty1momentweq}. There is no delta term. The Kloosterman term is
\[\frac{1}{2\pi i} \int_{\sigma_0 - i\infty}^{\sigma_0 + i\infty} \widehat{\Ks^{\hol} h^{\hol}}(s) \DD_{\overline{N},\chi}^{+}(s,w) N^{\frac{s + 1}{2}} \, ds\]
for $2 - 2\Re(w) < \sigma_0 < -1/2$. We shift the contour to $\Re(s) = \sigma_1$ with $-3 < \sigma_1 < -2\Re(w)$. From \eqref{Doverlineswchiresidueeq}, the integrand has a pole at $s = 2(1 - w)$ with residue \eqref{Peterssoninfty1deltaweq}. The contour integral is equal to \eqref{Peterssoninfty1shiftedweq} by \eqref{DoverlineswchipostVoronoieq}.
\end{proof}

Now we specialise \hyperref[holinfty1momentwlemma]{Lemma \ref*{holinfty1momentwlemma}} to $N = 2$, $q = |D|$, and $\chi = \chi_D$.

\begin{corollary}
\label{holomorphicmomentidentitycor}
Let $D \equiv 1 \pmod{4}$ be a negative squarefree fundamental discriminant and let $\chi_D$ be the primitive quadratic character modulo $|D|$, so that $\chi_D(-1) = -1$. Let $h^{\hol} : 2\N \to \C$ be a compactly supported function vanishing at $k = 2$. Then the moment
\begin{equation}
\label{holomorphicmomentidentityeq}
\sum_{f \in \BB_{\hol}^{\ast}(\Gamma_0(2))} \left(-\eta_f(2)\right) \frac{L\left(\frac{1}{2},f\right) L\left(\frac{1}{2}, f \otimes \chi_D\right)}{L(1,\sym^2 f)} h^{\hol}(k_f)
\end{equation}
is equal to the sum of the main term
\begin{equation}
\label{holomorphicmomentidentitymaineq}
4 L(1,\chi_D) \sum_{\substack{k = 2 \\ k \equiv 0 \hspace{-.25cm} \pmod{2}}}^{\infty} \frac{k - 1}{2\pi^2} \left(-i^k\right) h^{\hol}(k)
\end{equation}
and the shifted convolution sum
\begin{multline}
\label{holomorphicmomentidentityshiftedeq}
\frac{4i}{\sqrt{|D|}} \sum_{D_1 D_2 = |D|} \chi_1(-2) \sum_{\substack{m = 1 \\ m \equiv 0 \hspace{-.25cm} \pmod{2}}}^{D_2 - 1} \lambda_{\chi_1,\chi_2}\left(\frac{m}{2},0\right) \lambda_{\chi_1,\chi_2}(D_2 - m,0)	\\
\times \frac{1}{2\pi i} \int_{\sigma_1 - i\infty}^{\sigma_1 + i\infty} \widehat{\Ks^{\hol} h^{\hol}}(s) \widehat{\JJ_1^{\hol}}(1 - s) \left(\frac{m}{D_2}\right)^{\frac{s - 1}{2}} \, ds,
\end{multline}
where $-3 < \sigma_1 < -1$.
\end{corollary}

\begin{proof}
We use \hyperref[holinfty1momentwlemma]{Lemma \ref*{holinfty1momentwlemma}} and holomorphically extend to $w = 1/2$. As a function of the complex variable $w$, \eqref{Peterssoninfty1momentweq} extends holomorphically to $w = 1/2$. Moreover, the product over $p \mid N_2$ vanishes at $w = 1/2$ unless $N_2 = 1$. The holomorphic extension of the remaining terms are also clear. We then multiply both sides by $-L^N(1,\chi_D) = -L(1,\chi_D) \nu(N)/N$. It remains to note that for $m > D_2$, the integral occurring in the shifted convolution sum vanishes, since we may shift the contour to the left, noting that the poles of $\widehat{\JJ_k^{\hol}}(s)$ at $s = 3 - k - 2\ell$ are cancelled by the zeroes of $\widehat{\JJ_1^{\hol}}(1 - s)$ at $s = 1 - 2\ell$ for $\ell \in \N$.
\end{proof}

\begin{remark}
The condition that $h^{\hol}$ vanishes at $k = 2$ may be removed with the effect of contributing an additional main term equal to
\[-\frac{24 \sqrt{|D|} L(1,\chi_D)^2}{\pi^3} h^{\hol}(2).\]
See \cite[Theorem 1]{MR12} and \cite[Theorem 6.5]{FW09}, where this is observed via the relative trace formula. We expect that one can prove this via the method of analytic continuation with a little extra care by using the ``Hecke trick'' of replacing $k$ with a complex variable having large real part, then meromorphically extending to $k = 2$, as is done in \cite{BF17}.
\end{remark}

\section{Bounds for moments of \texorpdfstring{$L$}{L}-functions}
\label{Proofmomentsboundsect}

We now prove the bounds for moments of $L$-functions stated in \hyperref[momentprop]{Proposition \ref*{momentprop}}. The proofs of these bounds depend on the ranges involved. For the range $T \gg |D|^{1/12}$, our key inputs are the exact identities for moments of $L$-functions stated in \hyperref[momentidentitycor]{Corollaries \ref*{momentidentitycor}} and \ref{holomorphicmomentidentitycor}. For the range $T \ll |D|^{1/12}$, on the other hand, our key inputs are the following bounds for third moments of $L$-functions.

\begin{lemma}[{Young \cite{You17}, Petrow--Young \cite{PY19,PY20}}]
\label{lem:PYcubic}
Let $D$ be a squarefree fundamental discriminant, and let $\chi_D$ denote the quadratic character modulo $|D|$.
\begin{enumerate}[leftmargin=*]
\item[\textnormal{(1)}] For $T \geq 1$, we have that
\begin{equation}
\label{PYMaasscubiceq}
\sum_{\substack{f \in \BB_0(\Gamma) \\ T \leq t_f \leq 2T}} \frac{L\left(\frac{1}{2},f \otimes \chi_D\right)^3}{L(1,\sym^2 f)} + \frac{1}{2\pi} \int\limits_{T \leq |t| \leq 2T} \left|\frac{L\left(\frac{1}{2} + it,\chi_D\right)^3}{\zeta(1 + 2it)}\right|^2 \, dt \ll_{\e} |D|^{1 + \e} T^{2 + \e}.
\end{equation}
\item[\textnormal{(2)}] For $D < 0$ and $T \geq 1$, we have that
\begin{equation}
\label{PYholcubiceq}
\sum_{\substack{f \in \BB_{\hol}^{\ast}(\Gamma_0(2)) \\ T \leq k_f \leq 2T \\ k_f \equiv 2 \hspace{-.25cm} \pmod{4}}} \frac{L\left(\frac{1}{2},f \otimes \chi_D\right)^3}{L(1,\sym^2 f)} \ll_{\e} |D|^{1 + \e} T^{2 + \e}.
\end{equation}
\end{enumerate}
\end{lemma}

\begin{proof}
These bounds follow, with a little effort from work of Young \cite{You17} and Petrow and Young \cite{PY19,PY20}, which build on the work of Conrey and Iwaniec \cite{CI00}. The bound \eqref{PYMaasscubiceq} is \cite[Theorem 1.1]{You17}. For $T \ll |D|^{\delta}$ for some sufficiently small $\delta > 0$, the bound \eqref{PYholcubiceq} is a special case of \cite[Theorem 1]{PY19} with $r = 2$ and $q = |D|$. To prove the bound \eqref{PYholcubiceq} for $T \gg |D|^{\delta}$, the proof of \cite[Theorem 1]{PY19} must be modified to give explicit dependence on $T$; this is done in \cite{You17,PY19} for level $1$ forms, whereas we require this for level $2$ forms.

We briefly sketch how the methods of these papers are combined to prove \eqref{PYholcubiceq} for $T \gg |D|^{\delta}$. The bound \eqref{PYholcubiceq} is implied by the bound $\MM(2,|D|) \ll_{\e} |D|^{\e} T^{1 + \e} \Delta$, where
\[\MM(2,|D|) \coloneqq \sum_{D_1 D_2 = |D|} \sum_{f \in \BB_{\hol}^{\ast}(2D_1)} \omega_f L\left(\frac{1}{2}, f \otimes \chi_D\right)^3 w\left(\frac{k_f - 1 - 2T}{\Delta}\right)\]
with $\omega_f$ equal to certain weights as in \cite[(65)]{PY19} satisfying $\omega_f \ll_{\e} |D|^{-1 + \e} k_f^{\e}$ and the test function $w$ as in \cite[Section 4]{You17} with $\Delta = T^{\e}$, so that $w$ is smooth and supported on $[\tfrac{1}{2},3]$. Via the approximate functional equation, we write $\MM(2,|D|)$ as in \cite[(69)]{PY19}, taking $r = 2$ and $q = |D|$, multiplied by $w(\tfrac{k_f - 1 - 2T}{\Delta})$ and summed over $k \in 2\N$ satisfying $\tfrac{1}{2} \leq \tfrac{k_f - 1 - 2T}{\Delta} \leq 3$. We may restrict the sums over $m$ and $n$ in \cite[(69)]{PY19} to $m \ll |D|^{2 + \e} T^{2 + \e} d^{-2}$ and $n \ll |D|^{1 + \e} T^{1 + \e}$ due to the rapid decay of the functions $V_1$ and $V_2$ arising from the approximate functional equation. We proceed as in \cite[Section 8.3]{PY19}, where now $Y$ is a large power of $|D|T$ rather than just of $|D|$; in this way, upon applying the Petersson formula, we are led to bound the term $\SS$ as in \cite[(80)]{PY19} except with the Bessel function $J_{\kappa - 1}$ in \cite[(80)]{PY19} replaced by the weighted sum of Bessel functions $B^{\mathrm{holo}}$ as in \cite[(5.10)]{You17}. We continue to follow \cite{PY19} by opening up the divisor function $\tau(m) = \sum_{n_1 n_1 = m} 1$ appearing in $\SS$ and applying a dyadic partition of unity to the sums over $n_1,n_2,n,c$ appearing in $\SS$, we break up $\SS$ into summands $\SS_{N_1,N_2,N_3,C}$ as in \cite[(87)]{PY19}. We then apply Poisson summation, as in \cite[Section 8.6]{PY19}, which breaks up these sums into an arithmetic part, defined in \cite[(91)]{PY19}, and an analytic part, defined in \cite[(92)]{PY19} except with $J_{\kappa - 1}$ replaced by $B^{\mathrm{holo}}$. We invoke the method of \cite[Section 9]{PY19} unaltered to deal with the arithmetic part, while the method of \cite[Section 8]{You17}, mildly corrected expanded upon further in \cite[Section 13]{PY20}, deals with the analytic part. These methods combine to complete the proof of the bound \eqref{PYholcubiceq}.
\end{proof}

\begin{proof}[Proof of {\hyperref[momentprop]{Proposition \ref*{momentprop} (1)} for $T \ll |D|^{1/12}$}]
We apply H\"{o}lder's inequality with exponents $(2,3,6)$ to the moment
\begin{equation}
\label{momenteq}
\sum_{\substack{f \in \BB_0(\Gamma) \\ T \leq t_f \leq 2T}} \frac{L\left(\frac{1}{2},f\right) L\left(\frac{1}{2},f \otimes \chi_D\right)}{L(1,\sym^2 f)} + \frac{1}{2\pi} \int\limits_{T \leq |t| \leq 2T} \left|\frac{\zeta\left(\frac{1}{2} + it\right) L\left(\frac{1}{2} + it,\chi_D\right)}{\zeta(1 + 2it)}\right|^2 \, dt,
\end{equation}
making use of the fact that $L(\tfrac{1}{2},f)$ and $L(\tfrac{1}{2},f \otimes \chi_D)$ are nonnegative via the work of Waldspurger \cite{Wal81}. A standard application of the spectral large sieve in conjunction with the approximate functional equation implies that
\begin{equation}
\label{secondmomentboundeq}
\sum_{\substack{f \in \BB_0(\Gamma) \\ T \leq t_f \leq 2T}} \frac{L\left(\frac{1}{2},f\right)^2}{L(1,\sym^2 f)} + \frac{1}{2\pi} \int\limits_{T \leq |t| \leq 2T} \left|\frac{\zeta\left(\frac{1}{2} + it\right)^2}{\zeta(1 + 2it)}\right|^2 \, dt \ll_{\e} T^{2 + \e}.
\end{equation}
Next, we have the hybrid Weyl-strength subconvex bounds \eqref{PYMaasscubiceq} for the third moment of $L(\tfrac{1}{2},f \otimes \chi_D)$ and the sixth moment of $L(\tfrac{1}{2} + it,\chi_D)$. Finally, the Weyl law implies that
\begin{equation}
\label{Weylboundeq}
\sum_{\substack{f \in \BB_0(\Gamma) \\ T \leq t_f \leq 2T}} \frac{1}{L(1,\sym^2 f)} + \frac{1}{2\pi} \int\limits_{T \leq |t| \leq 2T} \left|\frac{1}{\zeta(1 + 2it)}\right|^2 \, dt \ll T^2,
\end{equation}
which is a straightforward application of the Kuznetsov formula. Combined, this yields the bound $O_{\e}(|D|^{1/3 + \e} T^{2 + \e})$ for the moment \eqref{momenteq}.
\end{proof}

\begin{proof}[Proof of {\hyperref[momentprop]{Proposition \ref*{momentprop} (1)} for $T \gg |D|^{1/12}$}]
We use \hyperref[momentidentitycor]{Corollary \ref*{momentidentitycor}} with the test function
\[h(t) \coloneqq e^{-\frac{t^2}{T^2}} \prod_{j = 1}^{2M} \left(\frac{t^2 + \left(j - \frac{1}{2}\right)^2}{T^2}\right)^2\]
as in \cite[(1.16)]{BLM19}. This satisfies $h(t) \gg 1$ for $t \in [T,2T]$, so that the moment \eqref{momenteq} is bounded by a constant multiple of \eqref{momentidentityeq}, while the main term \eqref{momentidentitymaineq} satisfies
\[2 L(1,\chi_D) \int_{-\infty}^{\infty} h(t) \, d_{\spec}t \ll_{\e} |D|^{\e} T^{2 + \e}.\]
To bound the shifted convolution sum \eqref{momentidentityshiftedeq}, we use the bounds \eqref{JJ0boundseq} for $\widehat{\JJ_0^{\pm}}(1 - s)$ and $\widehat{\JJ_1^{\hol}}(1 - s)$ together with the bound
\[\widehat{\Ks^{-} h}(s) \ll_{\sigma} T^{1 + \sigma} (|\tau| + 1)^{-2M}\]
for $s = \sigma + i\tau$ with $-M < \sigma < M$, which follows from \cite[Lemma 4]{BLM19}. From this, we find that the shifted convolution sum is bounded by $O_{\e}(|D|^{1/2 + \e})$ upon taking $\sigma_1 = -1 - \e$ in \eqref{momentidentityshiftedeq} and using the bounds $\lambda_{\chi_1,\chi_2}(m,0) \ll_{\e} m^{\e}$ and $L(1,\chi_D) \gg_{\e} |D|^{-\e}$.
\end{proof}

\begin{proof}[Proof of {\hyperref[momentprop]{Proposition \ref*{momentprop} (2)}} for $T \ll |D|^{1/12}$]
This follows exactly as in the proof of \hyperref[momentprop]{Proposition \ref*{momentprop} (1)} for $T \ll |D|^{1/12}$. We apply H\"{o}lder's inequality with exponents $(2,3,6)$ to the moment
\begin{equation}
\label{momenteq'}
\sum_{\substack{f \in \BB_{\hol}^{\ast}(\Gamma_0(2)) \\ T \leq k_f \leq 2T \\ k_f \equiv 2 \hspace{-.25cm} \pmod{4}}} \frac{L\left(\frac{1}{2},f\right) L\left(\frac{1}{2},f \otimes \chi_D\right)}{L(1,\sym^2 f)},
\end{equation}
again making use of the fact that $L(\tfrac{1}{2},f)$ and $L(\tfrac{1}{2},f \otimes \chi_D)$ are nonnegative via the work of Waldspurger \cite{Wal81}. The analogue of the bound \eqref{secondmomentboundeq} for holomorphic cusp forms again holds via the large sieve and the approximate functional equation, while we have the hybrid Weyl-strength subconvex bounds \eqref{PYholcubiceq} for the third moment of $L(\tfrac{1}{2},f \otimes \chi_D)$. Finally, the analogue of the bound \eqref{Weylboundeq} for holomorphic cusp forms is valid since there are $\ll k$ elements of $\BB_{\hol}^{\ast}(\Gamma_0(2))$ of weight $k$; it is a straightforward consequence of the Petersson formula. The desired bound $O_{\e}(|D|^{1/3 + \e} T^{2 + \e})$ for \eqref{momenteq'} thereby follows.
\end{proof}

\begin{proof}[Proof of {\hyperref[momentprop]{Proposition \ref*{momentprop} (2)}} for $T \gg |D|^{1/12}$]
We use \hyperref[holomorphicmomentidentitycor]{Corollary \ref*{holomorphicmomentidentitycor}} with the test function
\[h^{\hol}(k) \coloneqq \frac{1 - i^k}{2} \tilde{h}\left(\frac{k - 1}{T}\right),\]
so that $h^{\hol}(k) = 0$ if $k \equiv 0 \pmod{4}$, where $\tilde{h} : \R \to [0,1]$ is a smooth function supported on $(1/2,5/2)$, equal to $1$ on $[1,2]$, and satisfying $\tilde{h}^{(j)}(x) \ll_j 1/x^j$. Since $L(\tfrac{1}{2},f) = 0$ when $\eta_f(2) = 1$ and $k_f \equiv 2 \pmod{4}$, as the root number of $f$ is $i^{k_f} \eta_f(2)$ \cite[Lemma A.2]{HK20}, the moment
\begin{equation}
\label{holomorphicmomenteq}
\sum_{\substack{f \in \BB_{\hol}^{\ast}(\Gamma_0(2)) \\ T \leq k_f \leq 2T \\ k_f \equiv 2 \hspace{-.25cm} \pmod{4}}} \frac{L\left(\frac{1}{2},f\right) L\left(\frac{1}{2}, f \otimes \chi_D\right)}{L(1,\sym^2 f)}
\end{equation}
is bounded by \eqref{holomorphicmomentidentityeq}, while the main term \eqref{holomorphicmomentidentitymaineq} is readily seen to be $O_{\e}( |D|^{\e} T^2)$.

It remains to bound the shifted convolution sum \eqref{holomorphicmomentidentityshiftedeq}. Via Mellin inversion, it suffices to show that
\begin{equation}
\label{holshiftedtoboundeq}
\sum_{D_1 D_2 = |D|} \sum_{m < D_2} \left|\int_{0}^{\infty} \left(\Ks^{\hol} h^{\hol}\right)(x) J_0\left(4\pi \sqrt{\frac{m}{D_2}} x\right) \, dx\right|
\end{equation}
is $O_{\e}(|D|^{1 + \e} T^{\e})$. We first note that
\begin{equation}
\label{KsholtoLegendreeq}
\int_{0}^{\infty} \left(\Ks^{\hol} h^{\hol}\right)(x) J_0\left(4\pi \sqrt{\frac{m}{D_2}} x\right) \, dx = \frac{1}{4\pi^2} \sum_{\substack{k = 2 \\ k \equiv 0 \hspace{-.25cm} \pmod{2}}}^{\infty} (k - 1) i^{-k} P_{\frac{k}{2} - 1}\left(1 - \frac{2m}{D_2}\right) h^{\hol}(k)
\end{equation}
for $m < D_2$ by \cite[6.512.4]{GR15}. At this point, we observe that we can alter $h^{\hol}(k)$ to be $-i^k \tilde{h}((k - 1)/T)$ without changing the required estimates by using the fact that for $k \in 2\N$,
\begin{equation}
\label{Legendretrickeq}
P_{\frac{k}{2} - 1}\left(1 - \frac{2m}{D_2}\right) = i^k P_{\frac{k}{2} - 1}\left(1 - \frac{2(D_2 - m)}{D_2}\right),
\end{equation}
and making the change of variables $m \mapsto D_2 - m$. We now proceed to bound \eqref{holshiftedtoboundeq} by breaking up this double sum and integral into different ranges and bounding each range separately.

\noindent\textit{Range I: $m \leq D_2/T^{2 - \e}$.} We merely note that $|P_m(\cos \theta)| \leq 1$; \eqref{KsholtoLegendreeq} shows that these terms are bounded by $O_{\e}(|D|^{1 + \e} T^{\e})$.

\noindent\textit{Range II: $D_2/T^{2 - \e} < m < D_2$ and $x \leq \frac{T}{4\pi e} \exp(-\frac{5 \log T}{T})$.} We observe that
\[\left(\Ks^{\hol} h^{\hol}\right)(x) \ll T \sum_{\frac{T}{2} \leq k \leq \frac{5T}{2}} \left|J_k(4\pi x)\right| \ll \frac{1}{T}\]
since
\[J_k(4\pi x) \ll \frac{(2\pi x)^k}{k!} \ll \frac{1}{\sqrt{k}} \left(\frac{2\pi e x}{k}\right)^k \ll \frac{1}{T^3}\]
by \cite[8.440]{GR15} and Stirling's formula. Together with the bound \eqref{J0zboundseq} for $J_0(x)$, this shows that the contribution of this to \eqref{holshiftedtoboundeq} is $O_{\e}(|D|^{1 + \e} T^{\e})$.

\noindent\textit{Range III: $D_2/T^{2 - \e} < m < D_2$ and $x \geq T^2$.} We must bound
\begin{equation}
\label{xlargetoboundeq}
T \sum_{D_1 D_2 = |D|} \sum_{\frac{D_2}{T^{2 - \e}} < m < D_2} \sum_{\frac{T}{2} \leq k \leq \frac{5T}{2}} \left|\int_{T^2}^{\infty} J_k(4\pi x) J_0\left(4\pi \sqrt{\frac{m}{D_2}} x\right) \, dx\right|.
\end{equation}
We claim that
\begin{equation}
\label{xlargeinttoboundeq}
\int_{T^2}^{\infty} J_k(4\pi x) J_0\left(4\pi \sqrt{\frac{m}{D_2}} x\right) \, dx \ll \frac{D_2^{7/4}}{m^{3/4} (D_2 - m) T^2}.
\end{equation}
Inserting this bound into \eqref{xlargetoboundeq} yields the bound $O_{\e}(|D|^{1 + \e} T^{\e})$. To prove the bound \eqref{xlargeinttoboundeq}, we write
\begin{align}
\label{WatsonJkeq}
J_k(4\pi x) & = \frac{e(2x)}{\sqrt{x}} W_k(4\pi x) + \frac{e(-2x)}{\sqrt{x}} \overline{W_k}(4\pi x),	\\
\label{WatsonWkeq}
W_k(x) & \coloneqq \frac{e\left(\frac{k}{4} - \frac{1}{8}\right)}{\Gamma\left(k + \frac{1}{2}\right)} \frac{1}{\sqrt{2} \pi} \int_{0}^{\infty} e^{-y} y^{k + \frac{1}{2}} \left(1 + \frac{iy}{2x}\right)^{k - \frac{1}{2}} \, \frac{dy}{y}
\end{align}
via \cite[Section 7.3]{Wat44}. We integrate by parts in \eqref{xlargeinttoboundeq}, antidifferentiating $e(2(1 \pm \sqrt{m/D_2})x)$ and differentiating the rest. Since $1 + x^2 \leq e^x$ for $x > 0$, we have that
\[W_k(x) \ll \frac{1}{\Gamma\left(k + \frac{1}{2}\right)} \int_{0}^{\infty} e^{-y \left(1 - \frac{k - \frac{1}{2}}{4x}\right)} y^{k + \frac{1}{2}} \, \frac{dy}{y} = \left(1 - \frac{k - \frac{1}{2}}{4x}\right)^{-k - \frac{1}{2}}.\]
In particular, $W_k(x) \ll 1$ for $x \gg k^2$. Similarly, $W_k'(x) \ll k^2/x^2 \ll 1/x$ for $x \gg k^2$, while $W_0(x) \ll 1$ and $W_0'(x) \ll 1/x^2 \ll 1/x$ for $x \geq 1$. From this, we deduce \eqref{xlargeinttoboundeq}.

\noindent\textit{Range IV: $D_2/T^{2 - \e} < m < D_2$ and $\frac{T}{4\pi e} \exp(-\frac{5 \log T}{T}) < x < T^2$.} We use the method of \cite[Section 5.5]{Iwa97}, which shows that
\begin{equation}
\label{Ksholidentityeq}
\left(\Ks^{\hol} h^{\hol}\right)(x) = - \frac{T^2 i}{\pi} \int_{-\infty}^{\infty} \sin (4\pi x \sin (2\pi u)) \int_{-\infty}^{\infty} \tilde{h}(y) y e(-Tuy) \, dy \, du.
\end{equation}
We break up the integral over $u$ into the ranges $|u| \leq v$ and $|u| > v$ for a parameter $v \in (0,1)$ to be chosen. For the portion of the integral with $|u| > v$, we integrate by parts $j + 1$ times with respect to $y$, antidifferentiating $e(-Tuy)$ and differentiating the rest, giving rise to a term of size $O_j(T^{1 - j} v^{-j})$ for any $j \in \N$. Next, we use a Taylor expansion to write
\[\sin(4\pi x \sin (2\pi u)) = \sin\left(8 \pi^2 ux\right) - \frac{16 \pi^4 u^3 x}{3} \cos\left(8 \pi^2 ux\right) + O\left(x|u|^5 + x^2 u^6\right).\]
The error term gives us an additional term of size $O(T^2 v^6 x(1 + vx))$. For the main term, we extend the integration over $u$ back to all of $\R$; for the portion of the integral with $|u| > v$, we again integrate by parts, obtaining an additional term of size $O_j(T^{1 - j} v^{-j})$. Evaluating the ensuing double integral via Fourier inversion, we find that
\begin{multline}
\label{Ksholasympeq}
\left(\Ks^{\hol} h^{\hol}\right)(x) = -2 \tilde{h}\left(\frac{4\pi x}{T}\right) x - \frac{1}{T^2} \tilde{h}''\left(\frac{4\pi x}{T}\right) x - \frac{4\pi}{3 T^3} \tilde{h}'''\left(\frac{4\pi x}{T}\right) x^2	\\
+ O_j\left(T^{1 - j} v^{-j} + T^2 v^6 x (1 + vx)\right).
\end{multline}

We take $v = T^{-1 + \frac{8}{j + 7}} x^{-\frac{2}{j + 7}}$ and $j = \lceil \frac{28}{\e} \rceil - 7$; this together with the bound \eqref{J0zboundseq} for $J_0(x)$ shows that the error term in \eqref{Ksholasympeq} contributes to \eqref{holshiftedtoboundeq} at most $O_{\e}(|D|^{1 + \e} T^{\e})$. Finally, we claim that
\[\int_{\frac{T}{4\pi e} \exp\left(-\frac{5 \log T}{T}\right)}^{T^2} \tilde{h}\left(\frac{4\pi x}{T}\right) x J_0\left(4\pi \sqrt{\frac{m}{D_2}} x\right) \, dx \ll_j T^2 \left(\frac{D_2}{mT^2}\right)^{\frac{5}{4}},\]
and similarly for the other two main terms in \eqref{Ksholasympeq}. To see this, we observe that we may extend the integral over $x$ back to all of $\R$ due to the support of $\tilde{h}$, make the change of variables $x \mapsto Tx$, insert the identity \eqref{WatsonJkeq} for $J_0(x)$, and integrate by parts twice, antidifferentiating $e(2xT\sqrt{m/D_2})$ and differentiating the rest, while noting that $W_0^{(j)}(x) \ll_j 1/x^{j + 1}$ for $x \geq 1$. We thereby find that the main terms in \eqref{Ksholasympeq} contribute to \eqref{holshiftedtoboundeq} a term of size $O_{\e}(|D|^{1 + \e} T^{\e})$.
\end{proof}

\section{Asymptotics for moments of \texorpdfstring{$L$}{L}-functions}
\label{asympsect}

\subsection{Proof of \texorpdfstring{\hyperref[variancethm2]{Theorem \ref*{variancethm2}}}{Theorem \ref{variancethm2}}}

The proof of \hyperref[variancethm2]{Theorem \ref*{variancethm2}} proceeds in a series of steps. First, we construct a test function that both satisfies the requirements of \hyperref[momentidentitycor]{Corollary \ref*{momentidentitycor}} and closely approximates $h_{r,R}(t)^2$. Next, we estimate the difference between $\Var(\Lambda_D; A_{r,R})$ and a moment of $L$-functions with our chosen test function. We then apply \hyperref[momentidentitycor]{Corollary \ref*{momentidentitycor}}. The main term \eqref{momentidentitymaineq} is readily shown to provide the desired asymptotic in \hyperref[variancethm2]{Theorem \ref*{variancethm2}}. The last step, which is the most taxing, is to bound the shifted convolution sum \eqref{momentidentityshiftedeq} and show that it is smaller than the main term.

\subsubsection{Construction of a test function}

In order to make use of \hyperref[momentidentitycor]{Corollary \ref*{momentidentitycor}}, we require stringent conditions on the test function; in particular, we cannot merely take the test function $h(t)$ to be $h_{r,R}(t)^2$. The conditions of \hyperref[momentidentitycor]{Corollary \ref*{momentidentitycor}} require that the test function $h(t)$ extends holomorphically to $|\Im(t)| < 2M$ with zeroes at $\pm i(n - 1/2)$ for $n \in \{1,\ldots,2M\}$. We shall also localise $h(t)$ to the region $[-T_2,-T_1] \cup [T_1,T_2]$ with $T_1 = (R - r)^{-1 + \alpha}$ and $T_2 = (R - r)^{-1 - \alpha}$ for a small fixed constant $\alpha > 0$ for which $T_1 \gg \max\{|D|^{5/12}, 1/r^2\}$; this is due to the fact that the main contribution to the size of $\Var(\Lambda_D; A_{r,R})$ comes from this range. Inspired by \cite[Section 3.9]{BK17a}, we can ensure these requirements are met by multiplying by the entire function
\[h_1(t) \coloneqq e^{-\left(\frac{t}{T_2}\right)^{2M}} \left(1 - e^{-\left(\frac{t}{T_1}\right)^{2M}}\right),\]
where $M \in \N$ is a large fixed constant. For $|\Im(t)| < 2M$, this satisfies
\begin{equation}
\label{h1asympeq}
h_1(t) = \begin{dcases*}
O\left(\left(\frac{\Re(t)}{T_1}\right)^{2M}\right) & for $|\Re(t)| \leq T_1$,	\\
1 + O\left(\left(\frac{\Re(t)}{T_2}\right)^{2M} + e^{-\left(\frac{\Re(t)}{T_1}\right)^{2M}}\right) & for $T_1 \leq |\Re(t)| \leq T_2$,	\\
O\left(e^{-\left(\frac{\Re(t)}{T_2}\right)^{2M}}\right) & for $|\Re(t)| \geq T_2$.
\end{dcases*}
\end{equation}
Moreover, for $j \in \{1,\ldots,2M\}$ and $t \in \R$,
\begin{equation}
\label{h1deriveq}
h_1^{(j)}(t) \ll_j \begin{dcases*}
\frac{(|t| + 1)^{2M - j}}{T_1^{2M}} & for $|t| \leq T_1$,	\\
\frac{|t|^{2M - j}}{T_2^{2M}} + \frac{|t|^{(2M - 1)j}}{T_1^{2Mj}} e^{-\left(\frac{t}{T_1}\right)^{2M}} & for $T_1 \leq |t| \leq T_2$,	\\
\frac{|t|^{(2M - 1)j}}{T_2^{2Mj}} e^{-\left(\frac{t}{T_2}\right)^{2M}} & for $|t| \geq T_2$.
\end{dcases*}
\end{equation}

Next, recalling \eqref{hrRtrefinedeq}, we must introduce the presence of a function that is asymptotic to $1/|t|^3$. We achieve this by multiplying by
\[h_2(t) \coloneqq (2\pi)^{-4M - 2} (4M + 3)^{-3} \frac{\Gamma\left(\frac{2M}{4M + 3} + \frac{it}{4M + 3}\right)^{4M + 3} \Gamma\left(\frac{2M}{4M + 3} - \frac{it}{4M + 3}\right)^{4M + 3}}{\Gamma\left(\frac{1}{2} + it\right) \Gamma\left(\frac{1}{2} - it\right)},\]
which is holomorphic in the strip $|\Im(t)| < 2M$, in which it has zeroes at $\pm i(n - 1/2)$ for $n \in \{1,\ldots, 2M\}$ and satisfies
\begin{equation}
\label{h2asympeq}
h_2(t) = \frac{1}{(t^2 + 4M^2)^{3/2}} + O\left(\frac{1}{(|\Re(t)| + 1)^4}\right)
\end{equation}
by Stirling's formula. Moreover, for $j \in \N$ and $t \in \R$,
\begin{equation}
\label{h2deriveq}
h_2^{(j)}(t) \ll_j (|t| + 1)^{-j - 3}.
\end{equation}

Finally, we take the entire function
\begin{equation}
\label{h3defeq}
h_3(t) \coloneqq \sin^2 \frac{(R - r) t}{2} \sin^2 \frac{(R + r) t}{2}.
\end{equation}
For $|\Im(t)| < 2M$, this satisfies
\begin{equation}
\label{h3asympeq}
h_3(t) = \begin{dcases*}
\frac{(R^2 - r^2) t^4}{16} + O\left(r^4 (R - r)^2 (|\Re(t)| + 1)^6\right) & for $|\Re(t)| \leq \dfrac{1}{r}$,	\\
O\left((R - r)^2 |\Re(t)|^2\right) & for $\dfrac{1}{r} \leq |\Re(t) \leq \dfrac{1}{R - r}$,	\\
O(1) & for $|\Re(t) \geq \dfrac{1}{R - r}$.
\end{dcases*}
\end{equation}

We choose the test function
\begin{equation}
\label{htdefeq}
h(t) \coloneqq h_1(t) h_2(t) h_3(t).
\end{equation}
Combing \eqref{h1asympeq}, \eqref{h2asympeq}, and \eqref{h3asympeq}, we obtain upper bounds and asymptotics for $h(t)$.

\begin{lemma}
For $|\Im(t)| < 2M$, we have that
\begin{equation}
\label{htupperboundseq}
h(t) \ll \begin{dcases*}
\frac{r^2 (R - r)^2 (|\Re(t)| + 1)^{2M + 1}}{T_1^{2M}} & for $|\Re(t)| \leq \dfrac{1}{r}$,	\\
\frac{(R - r)^2 |\Re(t)|^{2M - 1}}{T_1^{2M}} & for $\dfrac{1}{r} \leq |\Re(t)| \leq T_1$,	\\
\frac{(R - r)^2}{|\Re(t)|} & for $T_1 \leq |\Re(t)| \leq \dfrac{1}{R - r}$,	\\
\frac{1}{|\Re(t)|^3} & for $\dfrac{1}{R - r} \leq |\Re(t)| \leq T_2$,	\\
\frac{e^{-\left(\frac{\Re(t)}{T_2}\right)^{2M}}}{|\Re(t)|^3} & for $|\Re(t)| \geq T_2$.
\end{dcases*}
\end{equation}
Moreover, for $t \in \R$,
\begin{equation}
\label{htrefinedeq}
\begin{split}
h(t) & = \frac{1}{|t|^3} \sin^2 \frac{(R - r) t}{2} \sin^2 \frac{(R + r) t}{2}	\\
& \qquad + \begin{dcases*}
O\left(\frac{(R - r)^2}{|t|^2} + \frac{(R - r)^2 |t|^{2M - 1}}{T_2^{2M}} + \frac{(R - r)^2 e^{-\left(\frac{t}{T_1}\right)^{2M}}}{|t|}\right) & for $T_1 \leq |t| \leq \dfrac{1}{R - r}$,	\\
O\left(\frac{1}{|t|^4} + \frac{|t|^{2M - 3}}{T_2^{2M}}\right) & for $\dfrac{1}{R - r} \leq |t| \leq T_2$.
\end{dcases*}
\end{split}
\end{equation}
\end{lemma}

For future reference, we record the following definitions and bounds:
\begin{gather}
\label{defandboundseq}
T_1 \coloneqq (R - r)^{-1 + \alpha} \gg \max\left\{|D|^{\frac{5}{12}},\frac{1}{r^2}\right\}, \qquad T_2 = (R - r)^{-1 - \alpha},	\\
\notag
|D|^{-\frac{1}{12} + \delta} \ll r \ll 1, \qquad R - r \ll |D|^{-\frac{5}{12} - \delta}.
\end{gather}
Here $\alpha,\delta > 0$ are small fixed constants (with $\alpha$ sufficiently small dependent on $\delta$, namely $\alpha \leq 12\delta/5$), while $M \in \N$ is a large fixed constant. In particular, we may take $\alpha = \min\{12\delta/5,2/35\}$ and $M = 100$, say.

\subsubsection{Comparison of the variance to a moment of $L$-functions}

\begin{lemma}
\label{comparisonlemma}
Fix $\delta > 0$, and suppose that $|D|^{-1/12 + \delta} \ll r \ll 1$ and $\mu(A_{r,R}) \ll r |D|^{-5/12 - \delta}$. Then for $h(t)$ as in \eqref{htdefeq}, we have that
\begin{equation}
\label{Vartomomenteq}
\begin{split}
\Var(\Lambda_D; A_{r,R}) & = \frac{\mu(\Gamma \backslash \Hb)}{\mu(A_{r,R}) \#\Lambda_D} \frac{2\pi}{\sinh \frac{R - r}{2} L(1,\chi_D)} \sum_{f \in \BB_0(\Gamma)} \frac{L\left(\frac{1}{2},f\right) L\left(\frac{1}{2},f \otimes \chi_D\right)}{L(1,\sym^2 f)} h(t_f)	\\
& \qquad + \frac{\mu(\Gamma \backslash \Hb)}{\mu(A_{r,R}) \#\Lambda_D} \frac{1}{\sinh \frac{R - r}{2} L(1,\chi_D)} \int_{-\infty}^{\infty} \left|\frac{\zeta\left(\frac{1}{2} + it\right) L\left(\frac{1}{2} + it,\chi_D\right)}{\zeta(1 + 2it)}\right|^2 h(t) \, dt	\\
& \qquad \quad + O_{\e}\left(\frac{1}{r|D|^{1/12 - \e}} + \frac{T_1^{1 + \e}}{r|D|^{1/2 - \e}} + \frac{1}{r(R - r)^2 T_2^{1 - \e} |D|^{1/2 - \e}}\right).
\end{split}
\end{equation}
\end{lemma}

Recalling \eqref{defandboundseq}, we see that the first term in the error term in \eqref{Vartomomenteq} is smaller than the desired asymptotic by at least $O_{\e}(D^{-\delta + \e})$ and the second and third terms are smaller by at least $O_{\e}((R - r)^{\alpha - \e} |D|^{\e})$.

\begin{proof}
This follows from the spectral expansion \eqref{VarD<0eq} of the variance, the upper bounds \eqref{hrRtupperboundseq} and asymptotics \eqref{hrRtrefinedeq} for $h_{r,R}(t)^2$, the upper bounds \eqref{htupperboundseq} and asymptotics \eqref{htrefinedeq} for $h(t)$, and the bounds in \hyperref[momentprop]{Proposition \ref*{momentprop}} for moments of $L$-functions.
\end{proof}

Via \hyperref[momentidentitycor]{Corollary \ref*{momentidentitycor}}, the first two terms on the right-hand side of \eqref{Vartomomenteq} are equal to the sum of the main term
\begin{equation}
\label{Varmaintermeq}
\frac{\mu(\Gamma \backslash \Hb)}{\mu(A_{r,R}) \# \Lambda_D} \frac{4\pi}{\sinh \frac{R - r}{2}} \int_{-\infty}^{\infty} h(t) \, d_{\spec}t
\end{equation}
and the shifted convolution sum
\begin{multline}
\label{Varshiftedeq}
\frac{\mu(\Gamma \backslash \Hb) }{\mu(A_{r,R}) \# \Lambda_D} \frac{4\pi i}{\sinh \frac{R - r}{2} \sqrt{|D|} L(1,\chi_D)} \sum_{D_1 D_2 = |D|} \sum_{m = 1}^{\infty} \lambda_{\chi_1,\chi_2}(m,0) \lambda_{\chi_1,\chi_2}(m + D_2,0)	\\
\times \frac{1}{2\pi i} \int_{\sigma_1 - i\infty}^{\sigma_1 + i\infty} \widehat{\Ks^{-} h}(s) \widehat{\JJ_1^{\hol}}(1 - s) \left(\frac{m}{D_2}\right)^{\frac{s - 1}{2}} \, ds.
\end{multline}

\subsubsection{Asymptotics for the main term}

\begin{lemma}
\label{maintermlemma}
For $h(t)$ as in \eqref{htdefeq}, the main term \eqref{Varmaintermeq} is equal to
\begin{equation}
\label{maintermeq}
\frac{\mu(\Gamma \backslash \Hb)}{\mu(A_{r,R}) \# \Lambda_D} + O_{\e}\left(\frac{1}{r(R - r)^2 T_2 |D|^{1/2 - \e}} + \frac{1}{r^2 (R - r)^2 T_1^2 |D|^{1/2 - \e}} + \frac{T_1}{r |D|^{1/2 - \e}}\right).
\end{equation}
\end{lemma}

Again recalling \eqref{Vartomomenteq}, the first and third terms in the error term in \eqref{maintermeq} are smaller than the main term by at least $O_{\e}((R - r)^{\alpha} |D|^{\e})$, while the second term is smaller by at least $O_{\e}(r^{-1} (R - r)^{1 - 2\alpha} |D|^{\e})$.

\begin{proof}
Via the upper bounds \eqref{htupperboundseq} and asymptotics \eqref{htrefinedeq} for $h(t)$, the main term \eqref{Varmaintermeq} is
\begin{multline*}
\frac{\mu(\Gamma \backslash \Hb)}{\mu(A_{r,R}) \# \Lambda_D} \frac{4}{\pi \sinh \frac{R - r}{2}} \int_{T_1}^{T_2} \frac{\sin^2 \frac{(R - r) t}{2} \sin^2 \frac{(R + r) t}{2}}{t^2} \, dt	\\
+ O_{\e}\left(\frac{T_1}{r |D|^{1/2 - \e}} + \frac{1}{r (R - r)^2 T_2 |D|^{1/2 - \e}}\right).
\end{multline*}
After integrating by parts, antidifferentiating $1/t^2$ and differentiating the rest, and using the fact that
\[\Si(x) \coloneqq \int_{0}^{x} \frac{\sin t}{t} \, dt = \begin{dcases*}
x + O(x^3) & for $0 \leq x \ll 1$,	\\
\frac{\pi}{2} - \frac{\cos x}{x} + O\left(\frac{1}{x^2}\right) & for $x \gg 1$,
\end{dcases*}\]
we find that
\begin{equation}
\label{mainterminteq}
\int_{T_1}^{T_2} \frac{\sin^2 \frac{(R - r) t}{2} \sin^2 \frac{(R + r) t}{2}}{t^2} \, dt = \frac{\pi (R - r)}{8} + O\left(\frac{1}{T_2} + \frac{1}{rT_1^2} + (R - r)^2 T_1\right).\qedhere
\end{equation}
\end{proof}

\subsubsection{Bounds for the shifted convolution sum}

The shifted convolution sum takes more work to bound. Our strategy is similar to the proof of \hyperref[momentprop]{Proposition \ref*{momentprop} (2)} for $T \gg |D|^{1/12}$, though it is more involved due to the oscillatory behaviour of the test function.

\begin{lemma}
\label{shiftedlemma}
For $h(t)$ as in \eqref{htdefeq}, the shifted convolution sum \eqref{Varshiftedeq} is
\[O_{\e}\left(\frac{1}{r(R - r)^{1 - \alpha + \e} |D|^{1/2 - \e}}\right).\]
\end{lemma}

This is smaller than the main term by at least $O_{\e}((R - r)^{\alpha - \e} |D|^{\e})$.

\begin{proof}
Via Mellin inversion, the result will follow upon showing that
\begin{equation}
\label{shiftedtoboundeq}
\sum_{D_1 D_2 = |D|} \sum_{\substack{m = 1 \\ m \neq D_2}}^{\infty} m^{\e} \left| \int_{0}^{\infty} \left(\Ks^{-} h\right)(x) J_0\left(4\pi \sqrt{\frac{m}{D_2}} x\right) \, dx\right|
\end{equation}
is $O_{\e}((R - r)^{1 + \alpha - \e} |D|^{1/2 + \e})$. As in the proof of \hyperref[momentprop]{Proposition \ref*{momentprop} (2)} for $T \gg |D|^{1/12}$, we break up the sums and integrals into different ranges and bound each individually.

\noindent\textit{Range I: $m \leq (R - r)^{\alpha} \sqrt{D_2}$.} The integral in \eqref{shiftedtoboundeq} is equal to
\[\int_{-\infty}^{\infty} h(t) \int_{0}^{\infty} 4\cosh \pi t K_{2it}(4\pi x) J_0\left(4\pi \sqrt{\frac{m}{D_2}} x\right) \, dx \, d_{\spec}t.\]
By \cite[8.432.1, 6.611.1, and 3.517.1]{GR15}, the inner integral is equal to
\[\frac{2\sqrt{2}}{\pi} \cosh \pi t \int_{0}^{\infty} \frac{\cos tu}{\sqrt{\cosh u + 1 + \frac{2m}{D_2}}} \, du = 2 P_{-\frac{1}{2} + it}\left(1 + \frac{2m}{D_2}\right).\]
We use \eqref{associatedLegendreboundseq} and \eqref{J0zboundseq} to bound this and \eqref{htupperboundseq} to bound $h(t)$. From this, the contribution of \eqref{shiftedtoboundeq} involving terms with $m \leq (R - r)^{\alpha} \sqrt{D_2}$ is $O_{\e}((R - r)^{1 + \alpha - \e} |D|^{1/2 + \e})$.

\noindent\textit{Range II: $m > D_2$ and $x \leq 1$.} Via integration by parts, the integral over $x$ in \eqref{shiftedtoboundeq} is equal to
\begin{equation}
\label{shiftedintbypartseq}
\frac{D_2}{16\pi^2 m} \int_{0}^{\infty} \frac{1}{x^2} \Ls(x) J_2\left(4\pi \sqrt{\frac{m}{D_2}} x\right) \, dx,
\end{equation}
where
\[\Ls(x) \coloneqq 3 \left(\Ks^{-} h\right)(x) - 3x \left(\Ks^{-} h\right)'(x) + x^2 \left(\Ks^{-} h\right)''(x).\]
By \cite[(A.2) and (A.4)]{BLM19}, we write
\[\frac{d^j}{dx^j} \JJ_t^{-}(x) = \frac{(2\pi)^j \pi i}{\sinh \pi t} \sum_{n = 0}^{j} \binom{j}{n} \left(I_{2it - j + 2n}(4\pi x) - I_{-2it - j + 2n}(4\pi x)\right).\]
We use the bound
\[\sech \pi t I_{2it - j + 2n}(4\pi x) \ll_{\Im(t),j} \frac{x^{-j + 2(n - \Im(t))}}{(|\Re(t)| + 1)^{\frac{1}{2} - j + 2(n - \Im(t))}},\]
which is valid for $0 < x \leq \sqrt{|t| + 1}$ \cite[(A.6)]{BLM19}. So by shifting the contour, we have that
\begin{equation}
\label{Ksderivsxsmalleq}
x^j \frac{d^j}{dx^j} \left(\Ks^{-} h\right)(x) \ll_j \sum_{\pm} \sum_{n = 0}^{j} x^{2(n - c_n)} \int\limits_{\Im(t) = \pm c_n} |h(t)| (|\Re(t)| + 1)^{j - 2(n - c_n) + \frac{1}{2}} \, dt
\end{equation}
for any choice of $c_n \in (-2M,2M)$. Choosing $c_n \leq n - 2M + \e$ and using the bounds \eqref{htupperboundseq}, we deduce that for $x \leq 1$,
\[\Ls(x) \ll_{\e} \frac{r^2 (R - r)^2 x^{4M - \e}}{T_1^{2M}}.\]
Since
\begin{equation}
\label{J2boundseq}
J_2(x) \ll \begin{dcases*}
x^2 & for $x \ll 1$,	\\
\frac{1}{\sqrt{x}} & for $x \gg 1$,
\end{dcases*}
\end{equation}
\cite[8.440 and 8.451.1]{GR15}, we see that the contribution to \eqref{shiftedintbypartseq} of the portion of the integral for which $x \leq 1$ is $O(r^2 (R - r)^2 T_1^{-2M} (D_2/m)^{5/4})$, which easily suffices to adequately bound the corresponding contribution to \eqref{shiftedtoboundeq}.

\noindent\textit{Range III: $m > D_2$ and $x \geq T_2 \log T_2$.} We write
\[\frac{d^j}{dx^j} \JJ_t^{-}(x) = (-2\pi)^j \sum_{n = 0}^{j} \binom{j}{n} 4 \cosh \pi t K_{2it - j + 2n}(4\pi x)\]
via \cite[(A.1)]{BLM19} and use the uniform bounds
\[4 \cosh \pi t K_{2it - j + 2n}(4\pi x) \ll_{\Im(t),j} e^{\min\{0,-\pi (4x - |\Re(t)|)\}} \left(\frac{1 + |\Re(t)| + 4\pi x}{4\pi x}\right)^{|2\Im(t) + j - 2n| + \frac{1}{10}}\]
for $t \in \C$ \cite[(A.3)]{BLM19}. From this and again using \eqref{htupperboundseq}, we see that for $x \geq 3T_2/4$,
\[\Ls(x) \ll (R - r) x^2 e^{-\frac{2\pi x}{3}} + \frac{T_2^2}{x} e^{-\left(\frac{4x}{3T_2}\right)^{2M}}.\]
So again using \eqref{J2boundseq}, we see that the contribution to \eqref{shiftedintbypartseq} of the portion of the integral for which $x \geq T_2 \log T_2$ is $O_{A}(T_2^{-A} (D_2/m)^{5/4})$ for any $A > 0$, which is more than enough to obtain the bound $O_{\e}((R - r)^{1 + \alpha - \e} |D|^{1/2 + \e})$ for the ensuing contribution to \eqref{shiftedtoboundeq}.

\noindent\textit{Range IV: $m > D_2$ and $1 < x < T_2 \log T_2$.} We begin with the identity
\[\left(\Ks^{-} h\right)(x) = 2\pi \int_{-\infty}^{\infty} e(2x \sinh \pi u) \int_{-\infty}^{\infty} h(t) e(-ut) \, d_{\spec}t \, du\]
from \cite[(A.8)]{BLM19} (cf.~\cite[Lemma 3.8]{BK17a} and \cite[Lemma 3.4]{BK17b}). Recalling the definition \eqref{h3defeq} of $h_3(t)$ and writing $\sin(2\pi x) = (2i)^{-1}(e(x) - e(-x))$, this is equal to
\begin{multline}
\label{Ks-hidentityeq}
\left(\Ks^{-} h\right)(x) = \frac{1}{16 \pi} \sum_{\pm} \left(\sum_{\rho \in \{R,r\}} - 2 \sum_{\rho \in \left\{\frac{R - r}{2}, \frac{R + r}{2}\right\}} + 2 \sum_{\rho = 0}\right)	\\
\times \int_{-\infty}^{\infty} e(2x \sinh(\pi u \pm \rho)) \int_{-\infty}^{\infty} \tilde{h}(t) e(-ut) \, dt \, du,
\end{multline}
where
\[\tilde{h}(t) \coloneqq h_1(t) h_2(t) t \tanh \pi t.\]
Note in particular that $\rho \ll 1$ in all cases. We now integrate by parts with respect to $u$, antidifferentiating $4\pi^2 ix \cosh(\pi u \pm \rho) e(2x \sinh(\pi u \pm \rho))$ and differentiating the rest, then multiply by $x$ and differentiate with respect to $x$. Doing this once more and taking an appropriate linear combination of the ensuing expressions yields the identity
\begin{multline}
\label{Kxidentityeq}
\Ls(x) = \frac{1}{16 \pi} \sum_{\pm} \left(\sum_{\rho \in \{R,r\}} - 2 \sum_{\rho \in \left\{\frac{R - r}{2}, \frac{R + r}{2}\right\}} + 2 \sum_{\rho = 0}\right)	\\
\times \int_{-\infty}^{\infty} e(2x \sinh(\pi u \pm \rho)) \int_{-\infty}^{\infty} \tilde{h}(t) (c_0 + c_1 t + c_2 t^2) e(-ut) \, dt,
\end{multline}
where
\begin{align*}
c_0 & \coloneqq 8 - 8 \tanh^2(\pi u \pm \rho) + 3\tanh^4(\pi u \pm \rho)	\\
c_1 & \coloneqq -14i \tanh(\pi u \pm \rho) + 6i \tanh^3(\pi u \pm \rho)	\\
c_2 & \coloneqq -4\tanh^2(\pi u \pm \rho).
\end{align*}

We break up the integrals over $u$ in \eqref{Kxidentityeq} into the ranges $|u| \leq v$ and $|u| > v$ for a parameter $v \in (0,1)$ to be chosen. For the portion of the integrals with $|u| > v$, we integrate by parts $2M + 1$ times with respect to $t$, antidifferentiating $e(-ut)$ and differentiating the rest, giving rise to a term of size $O((T_1 v)^{-2M} \log T_1)$ upon recalling \eqref{h1deriveq} and \eqref{h2deriveq}. Next, we employ a Taylor expansion to write
\[c_k e(2x \sinh(\pi u \pm \rho)) = e\left(\pm 2x \sinh\rho\right) e\left(2\pi xu \cosh \rho\right) \sum_{j = 0}^{2(J - 1)} u^j \sum_{\ell = 0}^{\lfloor \frac{j}{2}\rfloor} c_{j,k,\ell,\rho} x^{\ell} + O_J\left(x^J u^{2J}\right)\]
for $J \in \N$ and $k \in \{0,1,2\}$, where $c_{j,k,\ell,\rho} \in \C$ are uniformly bounded constants. The error term in this Taylor expansion contributes to \eqref{Kxidentityeq} a term of size $O_{J}(T_2 x^J v^{2J + 1})$. We extend the range of integration back to all of $\R$; if $J \leq M$, the portion of the integral with $|u| > v$ gives us a term of size $O(T_1^{-2M} v^{-2(M + 1 - J)} \log T_2)$ via integrating by parts $2M + 1$ times with respect to $t$. We take $v = T_1^{-\frac{2M}{2M + 3}} T_2^{-\frac{1}{2M + 3}} x^{-\frac{J}{2M + 3}}$, set $J = 2$, and evaluate the ensuing double integral via Fourier inversion, yielding
\begin{multline}
\label{Lsmaineq}
\Ls(x) = \frac{1}{16\pi} \sum_{\pm} \left(\sum_{\rho \in \{R,r\}} - 2 \sum_{\rho \in \left\{\frac{R - r}{2}, \frac{R + r}{2}\right\}} + 2 \sum_{\rho = 0}\right) \sum_{j = 0}^{2} (2\pi i)^{-j} \sum_{\ell = 0}^{\lfloor \frac{j}{2} \rfloor} x^{\ell} e\left(\pm 2x \sinh\rho\right)	\\
\times \sum_{k = 0}^{2} c_{j,k,\ell,\rho} \left. \frac{d^j}{dt^j} \right|_{t = 2\pi x \cosh \rho} \tilde{h}(t) t^k + O\left(x^{2 - \frac{10}{2M + 3}} T_1^{-5 + \frac{15}{2M + 3}} T_2^{1 - \frac{5}{2M + 3}} \log T_2\right).
\end{multline}

We insert this identity into \eqref{shiftedintbypartseq}, where the integral has been restricted to the range $1 \leq x \leq T_2 \log T_2$. Using \eqref{J2boundseq} to bound $J_2(x)$, the contribution from the error term in \eqref{Lsmaineq} to \eqref{shiftedtoboundeq} is
\[O_{\e}\left(T_1^{-5 + \frac{15}{2M + 3}} T_2^{\frac{3}{2} - \frac{15}{2M + 3} + \e} |D|^{1 + \e}\right).\]
Recalling \eqref{defandboundseq}, this is sufficient if $\alpha > 0$ is sufficiently small (in particular, it is readily checked that any $\alpha \leq 13/75$ suffices).

For the contribution from the main terms in \eqref{Lsmaineq}, we first break up the sum over $m > D_2$ dependent on $\rho$. For the terms for which either $D_2 < m \leq (1 - T_1^{-1}) D_2 \sinh^2 \rho$ or $m \geq (1 + T_1^{-1}) D_2 \sinh^2 \rho$, we use the identity \eqref{WatsonJkeq} for $J_2(x)$ in \eqref{shiftedintbypartseq} and integrate by parts, antidifferentiating $e(\pm_1 2x(\sqrt{m/D_2} \pm_2 \sinh \rho))$ and differentiating the rest. Bounding the main term in \eqref{Lsmaineq} via \eqref{h1deriveq} and \eqref{h2deriveq} and noting that for $x \gg 1$, $W_2(x) \ll 1$ and $W_2'(x) \ll 1/x^2$ with $W_2$ as in \eqref{WatsonWkeq}, we find that the integral is 
\[O\left(\frac{D_2^{7/4}}{T_1^{5/2} m^{5/4} \left|\sqrt{m} - \sqrt{D_2} \sinh \rho\right|}\right),\]
since the main contribution occurs when $x \asymp T_1$. The sums over $m$ in these ranges therefore contribute $O_{\e}(T_1^{-5/2 + \e} |D|^{1 + \e})$, which is sufficient provided that $\alpha > 0$ is sufficiently small (in particular, it is readily checked that $\alpha \leq 3/35$ suffices).

Finally, we bound the terms for which $(1 - T_1^{-1}) D_2 \sinh^2 \rho < m < (1 + T_1^{-1}) D_2 \sinh^2 \rho$. We use the bounds \eqref{h1deriveq} and \eqref{h2deriveq} to bound the main term in \eqref{Lsmaineq} and the bounds \eqref{J2boundseq} to bound $J_2(x)$ in \eqref{shiftedintbypartseq}, which combine to yield the bound $O_{\e}(T_1^{-5/2} |D|^{1 + \e})$ towards \eqref{shiftedtoboundeq} (again, the main contribution from the integral is when $x \asymp T_1$).

\noindent\textit{Range V: $(R - r)^{\alpha} \sqrt{D_2} < m < D_2$ and $x \leq 1$.} We follow the same strategy as for Range II, though we do not need to first integrate by parts in \eqref{shiftedtoboundeq}. We simply use \eqref{Ksderivsxsmalleq} with $j = 0$ and $c_0 = -2M + \e$ together with the bounds \eqref{J0zboundseq} for $J_0(x)$ to obtain the bounds $O_{\e}(r^2 (R - r)^2 T_1^{-2M} |D|^{1 + \e})$ towards \eqref{shiftedtoboundeq}.

\noindent\textit{Range VI: $(R - r)^{\alpha} \sqrt{D_2} < m < D_2$ and $x \geq T_2 \log T_2$.} Again, we follow the strategy as for Range III, from which we find that for $x \geq 3T_2/4$,
\[\left(\Ks^{-} h\right)(x) \ll (R - r) e^{-\frac{2\pi x}{3}} + \frac{T_2^2}{x^3} e^{-\left(\frac{4x}{3T_2}\right)^{2M}}.\]
From this and \eqref{J0zboundseq}, the contribution of this to \eqref{shiftedtoboundeq} is easily sufficiently small.

\noindent\textit{Range VII: $(R - r)^{\alpha} \sqrt{D_2} < m < D_2$ and $1 < x < T_2 \log T_2$.} Once more, our strategy is that of Range IV, from which we find that
\begin{multline}
\label{Ksmaineq}
\left(\Ks^{-} h\right)(x) = \frac{1}{16\pi} \sum_{\pm} \left(\sum_{\rho \in \{R,r\}} - 2 \sum_{\rho \in \left\{\frac{R - r}{2}, \frac{R + r}{2}\right\}} + 2 \sum_{\rho = 0}\right) \sum_{j = 0}^{2} (2\pi i)^{-j} \sum_{\ell = 0}^{\lfloor \frac{j}{2} \rfloor} x^{\ell} e\left(\pm 2x \sinh\rho\right)	\\
\times c_{j,\ell,\rho} \tilde{h}^{(j)}(2\pi x \cosh \rho) + O\left(x^{2 - \frac{10}{2M + 3}} T_1^{-6 + \frac{20}{2M + 3}} \log T_2\right).
\end{multline}

We use the bound \eqref{J0zboundseq} for $J_0(x)$ and recall the bounds \eqref{defandboundseq} in order to see that the contribution to \eqref{shiftedtoboundeq} from the error term in \eqref{Ksmaineq} is $O_{\e}((R - r)^{1 + \alpha - \e} |D|^{1/2 + \e})$ if $\alpha$ is sufficiently small (in particular, $\alpha \leq 3/10$ suffices).

For the main term, the integral in \eqref{shiftedtoboundeq} is trivially bounded for $1 < x \leq (4\pi)^{-1} \sqrt{D_2/m}$ by using the bounds \eqref{h1deriveq} and \eqref{h2deriveq} for $\tilde{h}^{(j)}(2\pi x \cosh \rho)$ together with the bound \eqref{J0zboundseq} for $J_0(x)$, noting that $m > (R - r)^{\alpha} \sqrt{D_2}$ implies that $x < T_1$.

In the remaining range $(4\pi)^{-1} \sqrt{D_2/m} < x < T_2 \log T_2$, we break up the sum over $(R - r)^{\alpha} \sqrt{D_2} < m < D_2$ dependent on $\rho$. For the terms for which either $m < (1 - T_1^{-1}) D_2 \sinh^2 \rho$ or $m > (1 + T_1^{-1}) D_2 \sinh^2 \rho$, we bound the integral in \eqref{shiftedtoboundeq} by inserting the identity \eqref{WatsonJkeq} for $J_0(x)$ and integrating by parts twice, antidifferentiating $e(\pm_1 2x(\sqrt{m/D_2} \pm_2 \sinh \rho))$ and differentiating the rest. Since $W_0(x) \ll 1$, $W_0'(x) \ll 1/x^2$, and $W_0''(x) \ll 1/x^3$ for $x \gg 1$ with $W_0$ as in \eqref{WatsonWkeq}, the integral is 
\[O\left(\frac{D_2^{5/4}}{T_1^{7/2} m^{1/4} \left|\sqrt{m} - \sqrt{D_2} \sinh \rho\right|^2}\right),\]
where once again the main contribution is when $x \asymp T_1$. The contributions from the ensuing sums over $m$ are
\[O_{\e}\left(T_1^{-7/2} (R - r)^{-\alpha/4 - \e} |D|^{9/8 + \e} + r^{-1} T_1^{-7/2} (R - r)^{\alpha/4} |D|^{9/8 + \e} + r^{-1/2} T_1^{-5/2} |D|^{1 + \e}\right).\]
This is sufficient provided that $\alpha$ is sufficiently small (in particular, $\alpha \leq 2/35$ suffices).

Finally, we must deal with the remaining terms for which $(1 - T_1^{-1}) D_2 \sinh^2 \rho < m < (1 + T_1^{-1}) D_2 \sinh^2 \rho$. We use \eqref{h1deriveq} and \eqref{h2deriveq} to bound $\tilde{h}^{(j)}(x)$ and \eqref{J0zboundseq} to bound $J_0(x)$; the ensuing integral over $(4\pi)^{-1} \sqrt{D_2/m} < x < T_2 \log T_2$ is $O(T_1^{-3/2} (D_2/m)^{1/4})$, and so the contribution to \eqref{shiftedtoboundeq} for the ensuing sum over $m$ in this range is $O_{\e}(r^{3/2} T_1^{-5/2} |D|^{1 + \e})$.
\end{proof}

\subsection{Proof of \texorpdfstring{\hyperref[variancethm1]{Theorem \ref*{variancethm1}}}{Theorem \ref{variancethm1}}}

The proof of \hyperref[variancethm1]{Theorem \ref*{variancethm1}} is similar to that of \hyperref[variancethm2]{Theorem \ref*{variancethm2}}, so we simply highlight the main differences.

\subsubsection{Construction of a test function}

Once again, we construct a test function that both satisfies the requirements of \hyperref[holomorphicmomentidentitycor]{Corollary \ref*{holomorphicmomentidentitycor}} and closely approximates $\tilde{h}_{r,R}(m)^2$; things are slightly simplified by the fact that we may choose this test function to be compactly supported. In particular, we take $h_1(x)$ to be a smooth compactly supported function that is bounded by $1$, equal to $1$ on $[T_1,T_2]$, vanishes for $x \leq T_1/2$ and $x \geq 2T_2$, and whose derivatives satisfy $h_1^{(j)}(x) \ll_j T_1^{-j}$ for $T_1/2 \leq x \leq T_1$ and $h_1^{(j)}(x) \ll_j T_2^{-j}$ for $T_2 \leq x \leq 2T_2$. We then take
\[h_2(k) \coloneqq \left(\frac{k - 1}{2}\right)^{-3}, \quad h_3(k) \coloneqq \sin^2 \frac{(R - r) (k - 1)}{4} \sin^2 \frac{(R + r) (k - 1)}{4}, \quad h_4(k) \coloneqq \frac{1 - i^k}{2},\]
and set
\[h^{\hol}(k) \coloneqq h_1(k) h_2(k) h_3(k) h_4(k).\]

\subsubsection{Comparison of the variance to a moment of $L$-functions}

Analogously to \hyperref[comparisonlemma]{Lemma \ref*{comparisonlemma}}, we find that $\Var(\widehat{\EE}(n); A_{r,R})$ is asymptotic to
\[\frac{\sigma(S^2)}{\sigma(A_{r,R}) \# \widehat{\EE}(n)} \frac{4\pi}{\sin \frac{R - r}{2} L(1,\chi_{-n})} \sum_{f \in \BB_{\hol}^{\ast}(\Gamma_0(2))} \left(-\eta_f(2)\right) \frac{L\left(\frac{1}{2},f\right) L\left(\frac{1}{2},f \otimes \chi_{-n}\right)}{L(1,\sym^2 f)} h^{\hol}(k).\]
We apply \hyperref[holomorphicmomentidentitycor]{Corollary \ref*{holomorphicmomentidentitycor}} to see that the right-hand side is the sum of the main term
\begin{equation}
\label{Varmaintermholeq}
\frac{\sigma(S^2)}{\sigma(A_{r,R}) \# \widehat{\EE}(n)} \frac{16\pi}{\sin \frac{R - r}{2}} \sum_{\substack{k = 2 \\ k \equiv 0 \hspace{-.25cm} \pmod{2}}}^{\infty} \frac{k - 1}{2\pi^2} h^{\hol}(k)
\end{equation}
and the shifted convolution sum
\begin{multline}
\label{Varshiftedholeq}
\frac{\sigma(S^2)}{\sigma(A_{r,R}) \# \widehat{\EE}(n)} \frac{16\pi i}{\sin \frac{R - r}{2} \sqrt{|D|} L(1,\chi_D)} \sum_{D_1 D_2 = |D|} \chi_1(-2)	\\
\times \sum_{\substack{m = 1 \\ m \equiv 0 \hspace{-.25cm} \pmod{2}}}^{D_2 - 1} \lambda_{\chi_1,\chi_2}\left(\frac{m}{D_2},0\right) \lambda_{\chi_1,\chi_2}(D_2 - m,0) \frac{1}{2\pi i} \int_{\sigma_1 - i\infty}^{\sigma_1 + i\infty} \widehat{\Ks^{\hol} h^{\hol}}(s) \widehat{\JJ_1^{\hol}}(1 - s) \left(\frac{m}{D_2}\right)^{\frac{s - 1}{2}} \, ds,
\end{multline}
where $D = -n$.

\subsubsection{Asymptotics for the main term}

Let $g(x) = (4x + 1) h_1(4x + 2) h_2(4x + 2) h_3(4x + 2)$, so that the main term \eqref{Varmaintermholeq} is
\[\frac{\sigma(S^2)}{\sigma(A_{r,R}) \# \widehat{\EE}(n)} \frac{4}{\pi \sin \frac{R - r}{2}} \sum_{m = -\infty}^{\infty} g(m).\]
We use the Poisson summation formula on the sum over $m$. From \eqref{mainterminteq}, we have that
\[\widehat{g}(0) = \int_{-\infty}^{\infty} \frac{1}{x^2} h_1(2x + 1) \sin^2 \frac{(R - r) x}{2} \sin^2 \frac{(R + r) x}{2} \, dx \sim \frac{\pi(R - r)}{8},\]
while for $m \in \N$, a simple integration by parts argument shows that for any $j \in \N$,
\[\widehat{g}(m) + \widehat{g}(-m) \ll_j \frac{1}{m^j T_1^{j + 1}}.\]
Thus \eqref{Varmaintermholeq} is asymptotic to $\sigma(S^2) / \sigma(A_{r,R}) \# \widehat{\EE}(n)$.

\subsubsection{Bounds for the shifted convolution sum}

We bound the shifted convolution sum \eqref{Varshiftedholeq} by the same method as in the proof of \hyperref[momentprop]{Proposition \ref*{momentprop} (2)} for $T \gg |D|^{1/12}$. Again, we may alter $h^{\hol}(k)$ to be $-i^k h_1(k) h_2(k) h_3(k)$ with impunity by \eqref{KsholtoLegendreeq} and \eqref{Legendretrickeq}, then break up the double sum and integral into four ranges. In this setting, Range I is $m \leq (R - r)^{\alpha} \sqrt{D_2}$, Range II is $(R - r)^{\alpha} \sqrt{D_2} < m < D_2$ and $x \leq \frac{T_1}{4\pi e} \exp(-\frac{5 \log T_1}{T_1})$, Range III is $(R - r)^{\alpha} \sqrt{D_2} < m < D_2$ and $x \geq T_2^2$, and Range IV is $(R - r)^{\alpha} \sqrt{D_2} < m < D_2$ and $\frac{T_1}{4\pi e} \exp(-\frac{5 \log T_1}{T_1}) < x < T_2^2$.

Ranges I, II, and III are bounded by the same methods as in the proof of \hyperref[momentprop]{Proposition \ref*{momentprop} (2)} for $T \gg |D|^{1/12}$. For Range IV, we recall the definition of $h_3(x)$ and write $\sin(2\pi x) = (2i)^{-1}(e(x) - e(-x))$ to obtain an identity akin to \eqref{Ks-hidentityeq} for $(\Ks^{\hol} h^{\hol})(x)$ instead of \eqref{Ksholidentityeq}. We then proceed just as in the proof of \hyperref[shiftedlemma]{Lemma \ref*{shiftedlemma}} for Range VII.

\section{Connections to subconvexity}
\label{connectionssect}

The bounds in \hyperref[momentprop]{Proposition \ref*{momentprop}} can be refined by taking test functions that localise to shorter intervals. In particular, one can show that
\begin{multline*}
\sum_{\substack{f \in \BB_0(\Gamma) \\ T \leq t_f \leq T + U}} \frac{L\left(\frac{1}{2},f\right) L\left(\frac{1}{2},f \otimes \chi_D\right)}{L(1,\sym^2 f)} + \frac{1}{2\pi} \int\limits_{T \leq |t| \leq T + U} \left|\frac{\zeta\left(\frac{1}{2} + it\right) L\left(\frac{1}{2} + it,\chi_D\right)}{\zeta(1 + 2it)}\right|^2 \, dt	\\
\ll_{\e} \begin{dcases*}
|D|^{\frac{1}{3} + \e} T^{1 + \e} U^{1 + \e} & for $1 \ll U \ll \min\left\{|D|^{\frac{1}{12}},T\right\}$,	\\
\frac{|D|^{\frac{1}{2} + \e} T^{1 + \e}}{U} & for $|D|^{\frac{1}{12}} \ll U \ll \min\left\{|D|^{\frac{1}{4}},T\right\}$,	\\
|D|^{\e} T^{1 + \e} U^{1 + \e} & for $|D|^{\frac{1}{4}} \ll U \ll T$
\end{dcases*}
\end{multline*}
and that the same bounds hold for
\[\sum_{\substack{f \in \BB_{\hol}^{\ast}(\Gamma_0(2)) \\ T \leq k_f \leq T + U \\ k_f \equiv 2 \hspace{-.25cm} \pmod{4}}} \frac{L\left(\frac{1}{2},f\right) L\left(\frac{1}{2},f \otimes \chi_D\right)}{L(1,\sym^2 f)}.\]
Choosing $U$ appropriately and dropping all but one term shows that for $D$ a squarefree fundamental discriminant and for $f \in \BB_0(\Gamma)$ or $t \in \R$,
\begin{align*}
L\left(\frac{1}{2},f\right) L\left(\frac{1}{2},f \otimes \chi_D\right) & \ll_{\e} \begin{dcases*}
|D|^{\frac{1}{3} + \e} (|t_f| + 1)^{1 + \e} & for $|t_f| \ll |D|^{\frac{1}{6}}$,	\\
|D|^{\frac{1}{2} + \e} & for $|D|^{\frac{1}{6}} \ll |t_f| \ll |D|^{\frac{1}{4}}$,	\\
|D|^{\frac{1}{4} + \e} |t_f|^{1 + \e} & for $|t_f| \gg |D|^{\frac{1}{4}}$,
\end{dcases*}	\\
\left|\zeta\left(\frac{1}{2} + it\right) L\left(\frac{1}{2} + it,\chi_D\right)\right|^2 & \ll_{\e} \begin{dcases*}
|D|^{\frac{1}{3} + \e} (|t| + 1)^{1 + \e} & for $|t| \ll |D|^{\frac{1}{6}}$,	\\
|D|^{\frac{1}{2} + \e} & for $|D|^{\frac{1}{6}} \ll |t| \ll |D|^{\frac{1}{4}}$,	\\
|D|^{\frac{1}{4} + \e} |t|^{1 + \e} & for $|t| \gg |D|^{\frac{1}{4}}$,
\end{dcases*}
\end{align*}
while for $D < 0$ a squarefree fundamental discriminant and for $f \in \BB_{\hol}^{\ast}(\Gamma_0(2))$,
\[L\left(\frac{1}{2},f\right) L\left(\frac{1}{2},f \otimes \chi_D\right) \ll_{\e} \begin{dcases*}
|D|^{\frac{1}{3} + \e} k_f^{1 + \e} & for $k_f \ll |D|^{\frac{1}{6}}$,	\\
|D|^{\frac{1}{2} + \e} & for $|D|^{\frac{1}{6}} \ll k_f \ll |D|^{\frac{1}{4}}$,	\\
|D|^{\frac{1}{4} + \e} k_f^{1 + \e} & for $k_f \gg |D|^{\frac{1}{4}}$.
\end{dcases*}\]
These bounds for these products of $L$-functions should be compared to the convexity bounds $O_{\e}(|D|^{1/2 + \e} (|t_f| + 1)^{1 + \e})$, $O_{\e}(|D|^{1/2 + \e} (|t| + 1)^{1 + \e})$, and $O_{\e}(|D|^{1/2 + \e} k_f^{1 + \e})$ respectively. Thus this gives hybrid subconvex bounds provided that there exists some $A > 0$ such that $|t_f| \ll |D|^A$, $|t| \ll |D|^A$, and $k_f \ll |D|^A$ respectively. Notably, these subconvex bounds are of Weyl-strength when $|t_f| \asymp |D|^{1/4}$, $|t| \asymp |D|^{1/4}$, and $k_f \asymp |D|^{1/4}$ respectively; that is, the bound is the conductor raised to the one-sixth. In general, Weyl-strength hybrid subconvex bounds can be achieved by bounding third moments of $L(\tfrac{1}{2},f)$ and of $L(\tfrac{1}{2},f \otimes \chi_D)$ \cite{Ivi01,Pen01,PY19,PY20,You17}.

Hybrid subconvex bounds via the first moment of $L(\tfrac{1}{2},f) L(\tfrac{1}{2},f \otimes \chi_D)$ (and more generally for Rankin--Selberg $L$-functions $L(\tfrac{1}{2},f \otimes g)$ with $g$ the theta lift of a class group Gr\"{o}\ss{}encharakter) have been previously achieved by Michel and Ramakrishnan \cite[Corollary 2]{MR12} using the relative trace formula; hybrid subconvex bounds in the level level aspect have also been obtained by Holowinsky and Templier \cite[Corollary 1]{HT14}. Michel and Ramakrishnan comment that they do not know of any application of such a hybrid subconvex bound \cite[p.~443]{MR12}. Here we find an application, \hyperref[variancethm1]{Theorems \ref*{variancethm1}} and \ref{variancethm2},
of bounds for the \emph{moment} that imply subconvexity, rather than an application of the \emph{individual} subconvex bounds for $L$-functions.

Finally, we comment on the obstacles towards improving \hyperref[aethm1]{Theorems \ref*{aethm1}} and \hyperref[aethm2]{\ref*{aethm2} (1)} to allow for the degenerate case of balls, so that $r = 0$. Proceeding via bounds for the variances and noting that $h_{0,R}(t) \gg 1$ for $t \leq 1/R$ and $\tilde{h}_{0,R}(m) \gg 1$ for $m \leq 1/R$, this would require showing that for any fixed $\delta > 0$ and for all $R \gg |D|^{-1/4 + \delta}$,
\begin{gather*}
\sum_{\substack{f \in \BB_0(\Gamma) \\ t_f \leq \frac{1}{R}}} \frac{L\left(\frac{1}{2},f\right) L\left(\frac{1}{2},f \otimes \chi_D\right)}{L(1,\sym^2 f)} + \frac{1}{2\pi} \int\limits_{|t| \leq \frac{1}{R}} \left|\frac{\zeta\left(\frac{1}{2} + it\right) L\left(\frac{1}{2} + it,\chi_D\right)}{\zeta(1 + 2it)}\right|^2 \, dt = o\left(\sqrt{|D|} L(1,\chi_D)^2\right),	\\
\sum_{\substack{f \in \BB_{\hol}^{\ast}(\Gamma_0(2)) \\ k_f \leq \frac{1}{R}}} \frac{L\left(\frac{1}{2},f\right) L\left(\frac{1}{2},f \otimes \chi_D\right)}{L(1,\sym^2 f)} = o\left(\sqrt{|D|} L(1,\chi_D)^2\right)
\end{gather*}
for squarefree fundamental discriminants $D < 0$. Were we able to obtain a stronger error term, namely a power-savings of the form $O(|D|^{1/2 - \alpha})$ for some $\alpha > 0$, then dropping all but one term would in particular imply the bounds
\[L\left(\frac{1}{2},f\right) L\left(\frac{1}{2},f \otimes \chi_D\right) \ll_{\e} |D|^{\frac{1}{2} - \alpha + \e} R^{-\e}\]
for $f \in \BB_0(\Gamma)$ with $t_f \approx 1/R$ or $f \in \BB_{\hol}^{\ast}(\Gamma_0(2))$ with $k_f \approx 1/R$. The conductor of the product of $L$-functions is $|D|^2 R^{-4}$, so for $R \asymp |D|^{-1/4 + \delta}$ with $\delta < 3\alpha/2$, this is a subconvex bound of sub-Weyl-strength. Proving such strong subconvex bounds is a well-known open problem; for this reason, improving \hyperref[variancethm1]{Theorems \ref*{variancethm1}} and \ref{variancethm2} via bounds for the variances appears to be highly challenging.

\phantomsection
\addcontentsline{toc}{section}{Acknowledgements}
\hypersetup{bookmarksdepth=-1}

\subsection*{Acknowledgements}

The authors would like to thank Alexandre de Faveri, Kimball Martin, Ian Petrow, Ze\'{e}v Rudnick, Abhishek Saha, Peter Sarnak, Rainer Schulze-Pillot, and Matthew Young for helpful discussions and comments, and to the referee for their careful reading of this paper. Special thanks are owed to Eren Mehmet K\i{}ral for preliminary work on this problem that led to this paper being written.

\hypersetup{bookmarksdepth}

\end{document}